\documentclass[reqno]{amsart}

\usepackage[utf8]{inputenc} 
\usepackage[T1]{fontenc}
\usepackage{ae}
\usepackage{aecompl}
\usepackage{bbold}
\usepackage{amssymb}
\usepackage[all]{xy}
\usepackage[english]{babel}
\usepackage{extarrows}
\usepackage[numbers]{natbib}
\usepackage{xparse}
\usepackage{color}
\usepackage{hyperref}
\hypersetup{bookmarksdepth=1, colorlinks, citecolor=blue,
linkcolor=blue,pdfinfo={Title={An isomorphism of motivic Galois groups},
Author={Utsav Choudhury and Martin Gallauer Alves de Souza}}}
\usepackage{xspace}
\usepackage{extarrows}
\usepackage{enumitem}
\usepackage{graphics}
\newtheorem{thm}[equation]{Theorem}
\newtheorem{lem}[equation]{Lemma}
\newtheorem{fac}[equation]{Fact}
\newtheorem{cor}[equation]{Corollary}

\newtheorem{pro}[equation]{Proposition}
\theoremstyle{definition}
\newtheorem{dfi}[equation]{Definition}

\newtheorem{rem}[equation]{Remark}
\numberwithin{equation}{section}

\makeatletter
\DeclareRobustCommand*\cal{\@fontswitch\relax\mathcal}
\makeatother

\providecommand{\eg}{\mbox{e.\,g.}\xspace}
\providecommand{\ie}{\mbox{i.\,e.}\xspace}
\providecommand{\ugm}{\mbox{u.\,g.\,m.}\xspace}
\providecommand{\inj}{\hookrightarrow}
\providecommand{\surj}{\twoheadrightarrow}

\NewDocumentCommand\ealga{g}{\ensuremath{\mathcal{H}^{\mathrm{eff}}_{\textsc{a}\IfNoValueTF{#1}{}{,#1}}}}
\NewDocumentCommand\alga{g}{\ensuremath{\mathcal{H}_{\textsc{a}\IfNoValueTF{#1}{}{,#1}}}}
\newcommand{\Haors}{\ensuremath{\mathbf{H}_{\textsc{a}}^{\textnormal{\tiny
      (eff)}}}}
\NewDocumentCommand\Has{g}{\ensuremath{\mathbf{H}_{\textsc{a}\IfNoValueTF{#1}{}{,#1}}}}
\NewDocumentCommand\Ha{g}{\ensuremath{\IfNoValueTF{#1}{\mathbf{H}_{\textsc{a}}^{\textnormal{\tiny eff}}}{\mathbf{H}_{\textsc{a},#1}^{\textnormal{\tiny eff}}}}}
\NewDocumentCommand\Hn{g}{\ensuremath{\IfNoValueTF{#1}{\mathbf{H}_{\textsc{n}}^{\textnormal{\tiny eff}}}{\mathbf{H}_{\textsc{n},#1}^{\textnormal{\tiny eff}}}}}
\NewDocumentCommand\Hns{g}{\ensuremath{\mathbf{H}_{\textsc{n}\IfNoValueTF{#1}{}{,#1}}}}
\NewDocumentCommand\grpn{g}{\ensuremath{\mathcal{G}_{\textsc{n}\IfNoValueTF{#1}{}{,#1}}}}
\NewDocumentCommand\grpa{g}{\ensuremath{\mathcal{G}_{\textsc{a}\IfNoValueTF{#1}{}{,#1}}}}
\providecommand{\cu}{\ensuremath{\mathrm{cu}}}
\providecommand{\cm}{\ensuremath{\mathrm{cm}}}
\providecommand{\coa}{\ensuremath{\mathrm{ca}}}

\providecommand{\sa}{\ensuremath{s_{\textsc{\tiny a}}}}
\providecommand{\sn}{\ensuremath{s_{\textsc{\tiny n}}}}

\newcommand{\aff}{\ensuremath{\mathrm{AffVar}}}
\NewDocumentCommand\sm{g}{\ensuremath{\mathrm{Sm}\IfNoValueTF{#1}{}{/#1}}}
\NewDocumentCommand\open{g}{\ensuremath{\mathrm{Open}\IfNoValueTF{#1}{}{/#1}}}
\NewDocumentCommand\smaff{g}{\ensuremath{\mathrm{SmAff}\IfNoValueTF{#1}{}{/#1}}}
\NewDocumentCommand\smaffcor{g}{\ensuremath{\mathrm{SmAffCor}\IfNoValueTF{#1}{}{/#1}}}
\NewDocumentCommand\ansm{g}{\ensuremath{\IfNoValueTF{#1}{\mathrm{Man}_{\C}}{\mathrm{AnSm}/{#1}}}}
\newcommand{\anmfd}{\ensuremath{\mathrm{Man}_{\R}^{\omega}}}
\newcommand{\pt}{\ensuremath{\mathrm{pt}}}
\newcommand{\sch}{\ensuremath{\mathrm{Var}/k}}

\newcommand{\mm}{\ensuremath{\mathcal{MM}}}
\newcommand{\ehm}{\ensuremath{\mathbf{HM}^{\textnormal{\tiny eff}}}}
\newcommand{\hm}{\ensuremath{\mathbf{HM}}}
\newcommand{\iehm}{\ensuremath{\ind\ehm}}
\newcommand{\ihm}{\ensuremath{\ind\hm}}

\newcommand{\daenis}[1][\Lambda]{\ensuremath{\mathbf{DA}^{\textnormal{\tiny eff},\mathrm{Nis}}}}
\NewDocumentCommand\da{g}{\ensuremath{\IfNoValueTF{#1}{\mathbf{DA}}{\mathbf{DA}_{#1}}}}
\NewDocumentCommand\dm{g}{\ensuremath{\IfNoValueTF{#1}{\mathbf{DM}}{\mathbf{DM}_{#1}}}}
\NewDocumentCommand\dae{g}{\ensuremath{\IfNoValueTF{#1}{\mathbf{DA}^{\textnormal{\tiny
          eff}}}{\mathbf{DA}^{\textnormal{\tiny eff}}_{#1}}}}
\NewDocumentCommand\daeaff{g}{\ensuremath{\IfNoValueTF{#1}{\mathbf{DA}^{\textnormal{\tiny
          eff}}_{\textnormal{\tiny
          aff}}}{\mathbf{DA}^{\textnormal{\tiny
          eff}}_{\textnormal{\tiny aff,}#1}}}}
\NewDocumentCommand\dme{g}{\ensuremath{\IfNoValueTF{#1}{\mathbf{DM}^{\textnormal{\tiny eff}}}{\mathbf{DM}^{\textnormal{\tiny eff}}_{#1}}}}
\NewDocumentCommand\andae{g}{\ensuremath{\IfNoValueTF{#1}{\mathbf{AnDA}^{\textnormal{\tiny eff}}}{\mathbf{AnDA}^{\textnormal{\tiny eff}}_{#1}}}}
\NewDocumentCommand\anda{g}{\ensuremath{\IfNoValueTF{#1}{\mathbf{AnDA}}{\mathbf{AnDA}_{#1}}}}
\newcommand{\danis}[1][\Lambda]{\ensuremath{\mathbf{DA}^{\mathrm{Nis}}}}
\newcommand{\dmeaff}{\ensuremath{\mathbf{DM}_{\textnormal{\tiny
        aff}}^{\textnormal{\tiny eff}}}}
\newcommand{\dmaff}{\ensuremath{\mathbf{DM}_{\textnormal{\tiny aff}}}}

\newcommand{\daaff}{\ensuremath{\mathbf{DA}_{\textnormal{\tiny aff}}}}

\usepackage{mathrsfs}
\providecommand{\diaeff}{\ensuremath{\mathscr{D}_{\textsc{\tiny n}}}}
\providecommand{\diagoodeff}{\ensuremath{\mathscr{D}^{\textnormal{\tiny
      g}}_{\textsc{\tiny
        n}}}}

\newcommand{\ffbti}{\ensuremath{\ff_{\mathrm{Bti}}}}
\providecommand{\ff}{\ensuremath{o}}
\newcommand{\An}{\ensuremath{\mathrm{an}}}
\newcommand{\Anal}{\ensuremath{\mathrm{An}}}
\NewDocumentCommand\sus{g}{\ensuremath{\mathrm{Sus}^{\IfNoValueTF{#1}{0}{#1}}}}
\NewDocumentCommand\ev{g}{\ensuremath{\mathrm{Ev}_{\IfNoValueTF{#1}{0}{#1}}}}
\providecommand{\na}{\ensuremath{\varphi_{\textsc{\tiny a}}}}
\providecommand{\an}{\ensuremath{\varphi_{\textsc{\tiny n}}}}
\providecommand{\nae}{\ensuremath{\varphi_{\textsc{\tiny a}}'}}
\providecommand{\ane}{\ensuremath{\varphi_{\textsc{\tiny n}}'}}
\providecommand{\rep}{\ensuremath{\mathcal{R}_{\textsc{a}}}}
\providecommand{\rea}{\ensuremath{\mathcal{R}_{\textsc{n}}}}
\providecommand{\reas}{\ensuremath{\mathcal{R}_{\textsc{n},s}}}
\providecommand{\reatrs}{\ensuremath{\mathcal{R}_{\textsc{n},\mathrm{tr},s}}}
\providecommand{\reaTRs}{\ensuremath{\mathcal{R}_{\textsc{n},(\mathrm{tr}),
      s}}}
\providecommand{\reah}{\ensuremath{\mathcal{R}_{\textsc{h}}}}
\providecommand{\reahtr}{\ensuremath{\mathcal{R}_{\textsc{h}, \mathrm{tr}}}}

\NewDocumentCommand\rbti{g}{\ensuremath{\IfNoValueTF{#1}{\mathrm{Bti}_{*}}{\mathrm{Bti}_{#1,*}}}}
\NewDocumentCommand\bti{g}{\ensuremath{\IfNoValueTF{#1}{\mathrm{Bti}^{*}}{\mathrm{Bti}^{*}_{#1}}}}
\newcommand{\rbtie}{\ensuremath{\mathrm{Bti}^{\textnormal{\tiny eff}}_{*}}}
\newcommand{\rbtinis}{\ensuremath{\mathrm{Bti}^{\mathrm{Nis}}_{*}}}
\newcommand{\tbti}{\ensuremath{\widetilde{\mathrm{Bti}}^{*}}}
\newcommand{\tbtinis}{\ensuremath{\widetilde{\mathrm{Bti}}^{\mathrm{Nis},*}}}
\newcommand{\btinis}{\ensuremath{\mathrm{Bti}^{\mathrm{Nis},*}}}
\newcommand{\btie}{\ensuremath{\mathrm{Bti}^{\textnormal{\tiny eff},*}}}
\newcommand{\tbtie}{\ensuremath{\widetilde{\mathrm{Bti}}^{\textnormal{\tiny eff},*}}}

\providecommand{\sget}{\ensuremath{\underline{\mathrm{Sg}}_{\mathrm{\acute{e}t}}^{\mathbb{D}}}}
\providecommand{\sgetn}{\ensuremath{\underline{\mathrm{Sg}}_{\mathrm{\acute{e}t}}^{\mathbb{D}^{\leq
    n}}}}
\providecommand{\nsget}{\ensuremath{{}^{\mathrm{n}}\sget}}
\providecommand{\nsgetn}{\ensuremath{{}^{\mathrm{n}}\sgetn}}

\providecommand{\N}{\ensuremath{\mathbb{N}}}
\providecommand{\Q}{\ensuremath{\mathbb{Q}}}
\providecommand{\C}{\ensuremath{\mathbb{C}}}
\providecommand{\R}{\ensuremath{\mathbb{R}}}
\providecommand{\Z}{\ensuremath{\mathbb{Z}}}
\providecommand{\A}{\ensuremath{\mathbb{A}}}
\providecommand{\G}{\ensuremath{\mathbb{G}}}
\providecommand{\Pe}{\ensuremath{\mathbb{P}}}

\providecommand{\Done}{\ensuremath{\mathbb{D}^{1}}}
\providecommand{\Disk}{\ensuremath{\mathbb{D}}}
\providecommand{\Dbar}{\ensuremath{\overline{\mathbb{D}}}}
\providecommand{\sgd}{\ensuremath{\mathrm{Sg}^{\mathbb{D}}}}
\providecommand{\sgi}{\ensuremath{\mathrm{Sg}^{\mathbb{I}}}}
\providecommand{\sgs}{\ensuremath{\mathrm{Sg}}}
\providecommand{\sgsd}{\ensuremath{\mathrm{Sg}^{\vee}}}
\providecommand{\sgsdbb}{\ensuremath{\mathbf{Sg}^{\vee}}}
\providecommand{\coMod}[1]{\ensuremath{\mathbf{coMod}(#1)}}
\providecommand{\grcoMod}[1]{\ensuremath{\mathbf{coMod}^{\Z}(#1)}}
\providecommand{\fcoMod}[1]{\ensuremath{\mathbf{coMod}^{\mathrm{f}}(#1)}}
\providecommand{\Mod}[1]{\ensuremath{\mathbf{Mod}(#1)}}
\providecommand{\fMod}[1]{\ensuremath{\mathbf{Mod}^{\mathrm{f}}(#1)}}
\providecommand{\Der}[1]{\ensuremath{\mathbf{D}(#1)}}
\providecommand{\Derb}[1]{\ensuremath{\mathbf{D}^{\mathrm{b}}(#1)}}
\providecommand{\totp}{\ensuremath{\mathrm{Tot}^{\scriptscriptstyle\prod}}}
\providecommand{\tots}{\ensuremath{\mathrm{Tot}^{\oplus}}}
\providecommand{\spt}{\ensuremath{\mathbf{Spt}^{\Sigma}}}

\providecommand{\hot}{\ensuremath{\mathbf{Hot}}}
\providecommand{\et}{\ensuremath{\mathrm{\acute{e}t}}}
\newcommand{\one}{\ensuremath{\mathbb{1}}}
\newcommand{\id}{\ensuremath{\mathrm{id}}}
\newcommand{\cupproduct}{\mathbin{\smile}}
\providecommand{\dR}{\ensuremath{\mathrm{R}}}
\providecommand{\dL}{\ensuremath{\mathrm{L}}}
\providecommand{\h}{\ensuremath{\mathrm{H}}}
\providecommand{\U}{\ensuremath{\mathbf{U}}}

\DeclareMathOperator{\spec}{Spec}
\DeclareMathOperator{\uhm}{\underline{Hom}}
\DeclareMathOperator{\uhom}{\underline{hom}}

\DeclareMathOperator{\End}{End}

\DeclareMathOperator{\ch}{\mathbf{Cpl}}
\DeclareMathOperator{\sh}{\mathbf{Sh}}
\DeclareMathOperator{\psh}{\mathbf{Psh}}
\DeclareMathOperator{\ind}{Ind}

\usepackage{amsaddr}
\title{An isomorphism of motivic Galois groups}
\author{Utsav Choudhury$^{1}$, Martin Gallauer Alves de Souza$^{2}$}
\email{prabrishik@gmail.com}
\email{gallauer@math.ucla.edu}
\thanks{The first author was supported by the Alexander von Humboldt Foundation, the second author was partially supported by the Swiss National Science Foundation (grant number 144372)}
\thanks{Present addresses: $^{1}$ Department of Mathematics, School of Mathematical Sciences, Ramakrishan Mission Vivekananda University, Belur Math, Howrah 711202, India, $^{2}$ Department of Mathematics, UCLA, 520 Portola Plaza,
MS 6363, Los Angeles, CA 90095, USA}

\date{}

\begin{document}
\begin{abstract}
  In characteristic 0 there are essentially two approaches to the
  conjectural theory of mixed motives, one due to Nori and the other
  one due to, independently, Hanamura, Levine, and Voevodsky. Although
  these approaches are apriori quite different it is expected that
  ultimately they can be reduced to one another. In this article we
  provide some evidence for this belief by proving that their
  associated motivic Galois groups are canonically isomorphic.
\end{abstract}
\maketitle{}
\tableofcontents

\section{Introduction}
\label{sec:introduction}
\subsection*{Motives and motivic Galois groups} Let $k$ be a field of
characteristic 0. Following Beilinson, Deligne, Grothendieck among
others, there should be a $\Q$-linear abelian monoidal category
$\mm(k)$ of \emph{(mixed) motives} over $k$ together with a monoidal
functor $M:(\sch)^{\mathrm{op}}\to\mm(k)$, associating to each variety
$X/k$ its motive $M(X)$, the universal cohomological invariant of
$X$. Every cohomology theory $h:(\sch)^{\mathrm{op}}\to \mathcal{A}$
for varieties over $k$ should factor through a realization functor
$R_{h}:\mm(k)\to\mathcal{A}$, \ie{} $h(X)=R_{h}(M(X))$. For some
cohomology theories one would expect this realization functor to
present $\mm(k)$ as a neutral Tannakian category with Tannakian dual
$\mathcal{G}(k)$, a pro-algebraic group called the \emph{motivic
  Galois group} of $k$. One of the main practical advantages of the
Tannakian description of motives is that it would allow the translation
of arithmetic and geometric questions about $k$-varieties into
questions about (pro-)algebraic groups and their
representations. Moreover, the maximal pro-reductive quotient of this
group is supposed to coincide with what was classically known as the
motivic Galois group, namely the group associated to the Tannakian
subcategory of \emph{pure} motives over $k$ (\ie{} the universal
cohomology theory for \emph{smooth projective} varieties;
see~\cite{serre:motivic-galois-group} for the philosophy underlying
this smaller group).

Although this picture is still conjectural, there are candidates for
these objects and related constructions. Assume there is an embedding
$\sigma:k\inj\C$. In this situation there are essentially two existing
approaches to motives, one due to Nori and another due to several
mathematicians, including Voevodsky. Nori constructed a diagram of
pairs of varieties together with the Betti representation into finite
dimensional $\Q$-vector spaces, and applied to it his theory of
Tannaka duality for diagrams. It yields a universal factorization for
the Betti representation through a $\Q$-linear abelian category
$\hm(k)$ with a faithful exact $\Q$-linear functor
$\ffbti:\hm(k)\to\fMod{\Q}$ to finite dimensional $\Q$-vector
spaces. $\hm(k)$ is a neutral Tannakian category with fiber functor
$\ffbti$, whose Tannakian dual $\grpn(k)$ is Nori's motivic Galois
group.

On the other hand, there is the better known construction of $\dm(k)$,
the triangulated category of Voevodsky motives. This is a candidate
not for the category of motives but its derived category, and from
which the former should be obtained as the heart of a
t-structure.\footnote{More precisely, its full subcategory of
  geometric motives (\ie{} compact objects) is a candidate for the
  bounded derived category of $\mm(k)$.} Ayoub constructed a Betti
realization functor $\bti:\dm(k)\to\Der{\Q}$ to the category of graded
$\Q$-vector spaces. He also proved that this functor together with its
right adjoint $\rbti$ satisfies the assumptions of his weak Tannakian
formalism which in the case at hand endows $\bti\rbti\Q$ with the
structure of a Hopf algebra. And finally, he established that its
homology is concentrated in non-negative degrees, hence the Hopf
algebra structure passes to its homology in degree 0. Ayoub's motivic
Galois group $\grpa(k)$ is the spectrum of this Hopf algebra.

\subsection*{Main result}

Our main goal is to prove that the two motivic Galois groups just
described are isomorphic, thus answering a question of Ayoub
in~\cite{ayoub:galois1}.  \newtheorem*{thm:main}{Theorem (instance of
  \ref{thm:main})}
\begin{thm:main}
  There is an isomorphism of affine pro-algebraic groups over
  $\spec(\Q)$:
\begin{equation*}
  \grpa(k)\cong\grpn(k).
\end{equation*}
\end{thm:main}

Let us try to put this result into perspective. As explained
in~\cite{ayoub:icm}, the difference between the two approaches to
motives is extreme. Nori's construction relies on transcendental data
in an essential way whereas $\dm(k)$ is defined purely in terms of
algebro-geometric data. This has the effect that while morphisms in
$\dm(k)$ can be related to previously known algebro-geometric
invariants of varieties, morphisms and extensions in $\hm(k)$ are
intractable. On the other hand, the universal property defining
$\hm(k)$ would be enough to characterize the true category of motives
(if the latter exists), whereas in Voevodsky's approach it is
difficult even to extract an abelian category. One of the ultimate
goals in the theory of motives therefore is to create a bridge
connecting these two approaches. This goal is considered to be far out
of reach at the moment, but the result above can be seen as providing
a weak link while sidestepping the more difficult and deep issues.
Moreover, it suggests breaking up the goal into two subgoals:
understanding the relation between (compact) Voevodsky motives and
comodules over $\bti\rbti\Q$ on the one hand, and proving that this
Hopf algebra is homologically concentrated in degree 0 on the other
hand; see~\cite{ayoub:icm} and \cite[§2.4]{ayoub:galois1} for further
discussion.

Even if the link we provide here is a weak one, it can still be seen
as evidence for the ``correctness'' of the two approaches to
motives. Moreover, although both constructions of motivic Galois
groups are based on some form of Tannaka duality, the precise form is
quite different in the two cases
(cf.~\cite[Introduction]{ayoub:galois1}); therefore the isomorphism in
the theorem can be seen as a surprising phenomenon. Finally, the
identification of the two groups allows for transfer of techniques and
results, not easily available on both sides without the
identification. We plan to use this fact in the future to give a more
elementary description of the Kontsevich-Zagier period algebra with
fewer generators and relations.  More precisely, we intend to show
that the algebra considered in~\cite[§2.2]{ayoub-periods14} is
canonically isomorphic to the Kontsevich-Zagier period algebra, as was
claimed in \textit{loc.\,cit.}

We would like to remark that a conditional proof of our main result
has been given independently by Jon
Pridham. In~\cite[Exa.~3.20]{pridham-dg-tannaka} he sketches how the
existence of a motivic $t$-structure (which renders the Betti
realization $t$-exact) would imply the isomorphism of motivic Galois
groups. The argument uses the theory of Tannaka duality for
dg categories developed in \textit{loc.\,cit.}

\subsection*{About the proof}

As one would expect from the relation between motives and their
associated Galois groups, proving our main result involves
``comparing'' Nori motives with Voevodsky motives. As we remarked
above, this is a non-trivial task and we can hope to relate these two
categories only indirectly:
\begin{itemize}
\item We construct a realization of Nori motives in the category of
  linear representations of Ayoub's Galois group:
  \begin{equation*}
    \mathbf{Rep}(\grpn(k))\to\mathbf{Rep}(\grpa(k)).
  \end{equation*}
  The main ingredients used in this construction are the six functors
  formalism for motives without transfers due to Ayoub and Voevodsky,
  and its compatibility with the Betti realization, proved by Ayoub.
\item We construct a realization of motives without transfers in the
  category of graded ${\cal O}(\grpn(k))$-comodules:
  \begin{equation*}
    \da(k)\to\grcoMod{{\cal O}(\grpn(k))}.
  \end{equation*}
  In this construction the main tool used is the Basic Lemma due to
  Nori and, independently, to Beilinson.
\end{itemize}
In fact, we work throughout with arbitrary principal ideal domains as
coefficients, not only $\Q$. Since we also use extensively the six
functors formalism, we are forced to work with $\da(k)$, motives
without transfers, instead of $\dm(k)$. In any case, the motivic
Galois group of Ayoub does not see the difference between these two
categories.

The two realizations will induce morphisms between the two Galois
groups, and the hard part is to prove that these are inverses to each
other. In one direction, we rely heavily on one of the main results of
Ayoub's approach to motivic Galois groups, namely a specific model he
has given for the object in $\da(k)$ representing Betti
cohomology. Analysing this model closely we can show that the
coordinate ring of $\grpa(k)$ as a $\grpa(k)$-representation is
generated by $\grpn(k)$-representations
$\h^i_{\mathrm{Betti}}(X,Z;\Q(j))$. This will allow us to prove the
morphism $\grpa(k)\to\grpn(k)$ a closed immersion. For the other
direction we will prove that $\grpn(k)\to\grpa(k)$ is a section to
$\grpa(k)\to\grpn(k)$, and here the idea is to reduce all
verifications to a class of pairs of varieties whose relative motive
in $\da(k)$ (and its effective version) are easier to handle. We have
found that the pairs $(X,Z)$ where $X$ is smooth and $Z$ a simple
normal crossings divisor work well for our purposes, and we study
their motives without transfers in detail.

\subsection*{Outline of the article}
\label{sec:outline}

We now give a more detailed account of the article. In~§\ref{sec:N} we
recall the construction and basic properties of Nori motives and the
associated Galois group. We also state a monoidal version of the
universal property of his category of motives the proof of which is
given in appendix~\ref{sec:N-ugm}. In~§\ref{sec:betti} we recall the
construction and basic properties of Morel-Voevodsky motives (or
motives without transfer) and the Betti realization. We also explain
in detail in which sense the functor $\bti$ \emph{is} a Betti
realization. We briefly recall the construction of Ayoub's Galois
group for Morel-Voevodsky motives in~§\ref{sec:A}.

In~§\ref{sec:NA} we construct motives $\rep(X,Z,n)$ in $\da(k)$ for
$X$ a variety, $Z \subset X$ a closed subvariety and $n$ a
non-negative integer. These motives are defined in terms of the six
functors, and have the property that $\h_0(\bti \rep(X,Z,n)) \cong
\h_n(X(\C), Z(\C))$ naturally. Here is where our decision to use the
six functors formalism pays off as its compatibility with the Betti
realization immediately reduces us to prove the existence of a natural
isomorphism between sheaf cohomology and singular cohomology of pairs
of (locally compact) topological spaces. We were not able to find the
required proofs for this last comparison in the literature, and we
therefore decided to provide them in a separate appendix (to wit,
appendix~\ref{sec:coh-pair}). We end this section by showing how this
construction yields a morphism of Hopf algebras
$\na:\mathcal{O}(\grpn(k))\to\mathcal{O}(\grpa(k))$.

The following two sections~\ref{sec:basic-lemma} and~\ref{sec:AN} are
devoted to defining a morphism in the other direction, at least on the
``effective'' bialgebras (the Hopf algebras are obtained from these
effective bialgebras by inverting a certain element). For this,
in~§\ref{sec:basic-lemma}, we recall Nori's version of the Basic
Lemma, and explain how it leads to algebraic cellular decompositions
of the singular homology of affine varieties. As an application we
obtain a functor from smooth affine schemes to the derived category of
effective Nori motives. In~§\ref{sec:AN.construction} we show how Kan extensions in the context of dg categories allow us to extend it to a functor $\dL C^{*}$ defined on the category of
effective Morel-Voevodsky motives.
$\dL C^{*}$ is then shown to give rise to the sought after morphism
of bialgebras (§\ref{sec:AN.an}).

We also collect additional results on realizing (Morel-)Voevodsky
motives in Nori motives which are not strictly necessary for our main
theorem but, we believe, of independent interest. In~§\ref{sec:v2n},
we extend our constructions to take into account correspondences thus
obtaining a variant of $\dL C^{*}$ for effective Voevodsky motives
(\ie{} effective motives with transfers). Then,
in~§\ref{sec:a2n-stable}, we also prove that these realizations pass
to the stable categories of motives (with and without transfers). From
this we finally deduce mixed Hodge realizations on motives with and
without transfers.

The next section is all about explicit computations involving
Morel-Voevodsky motives associated to pairs of schemes. The recurrent
theme is that these computations are feasible if one restricts to the
pairs $(X,Z)$ where $X$ is smooth and $Z$ is a simple normal crossings
divisor. We call these almost smooth pairs, and resolution of
singularities implies that there are enough of them. This allows us to
reduce computations for general pairs to these more manageable
ones. In~§\ref{sec:asp.eff} we give models for the latter on the
effective level, and determine their image under the functor $\dL
C^{*}$ explicitly. This allows us to compare their comodule structure
(with respect to Nori's effective bialgebra mentioned above) to the
one of the Betti homology of the pair. As a corollary, we see that the
morphism of bialgebras passes to the Hopf algebras
$\an:\mathcal{O}(\grpa(k))\to\mathcal{O}(\grpn(k))$. In~§\ref{sec:asp.stb}
we give good models for $\rep(X,Z,n)$ and their duals on the stable
level, when $(X,Z)$ is almost smooth, and we describe their Betti
realization.

§\ref{sec:main} is the heart of the article. Ayoub has given a
``singular'' model for the object in $\da(k)$ representing Betti
cohomology. Using our description of $\rep(X,Z,n)$ and performing a
close analysis of Ayoub's model we establish that the Hopf algebra
$\mathcal{O}(\grpa(k))$ as a comodule over itself is a filtered
colimit of Nori motives $\h^i(X(\C),Z(\C); \Q(j))$, where $(X,Z)$ is
almost smooth and $i,j\in \Z$. This will be seen to imply surjectivity
of $\na$, while on the other hand we also prove that $\an\na$ is the
identity by proving that it is so on motives of almost smooth pairs.

\subsection*{Acknowledgments}
First and foremost we would like to thank Joseph Ayoub for suggesting
that we work together on this problem, as well as for his support
during the gestation of this article. The crucial idea of using his
model for the motive representing Betti cohomology in the proof of
Theorem~\ref{thm:N-gen-A} is due to him. We also profited from
discussions with Marc Levine, Simon Pepin Lehalleur, Oliver
R\"{o}ndigs, Markus Spitzweck, Vaibhav Vaish, Alberto Vezzani. Special
thanks go to Simon for a very detailed reading of the first version of
our paper. The first author would like to thank Marc Levine and Oliver
R\"{o}ndigs for their support and encouragement. Part of the research
was conducted during a stay of the second author at the University of
Osnabr\"{u}ck, and he would like to thank the research group Topology
and Geometry for the hospitality.

\subsection*{Notation and conventions}

We fix a field $k$ of characteristic 0 together with an embedding
$\sigma:k\inj\C$. By a scheme we mean a quasi-projective scheme over
$k$. A variety is a reduced scheme. Rings are always assumed
commutative and unital. Monoidal categories (resp.\ functors,
transformations) are assumed symmetric and unitary if not stated
otherwise. Algebras and coalgebras in monoidal categories (also known
as monoids and comonoids, respectively) are assumed unitary resp.\
counitary.

$\Lambda$ throughout denotes a fixed ring, assumed noetherian if not
stated otherwise. The symbol $\ch(\Lambda)$ denotes the category of
(unbounded) complexes of $\Lambda$-modules. Our conventions are
homological, \ie{} the differentials decrease the indices, and the
shift operator satisfies $(A[p])_{n}=A_{p+n}$. For an abelian category
$\mathcal{A}$, $\Der{\mathcal{A}}$ denotes its derived category, and
$\Der{\Lambda}:=\Der{\ch(\Lambda)}$. Also, $\ind{\cal A}$ denotes
the category of ind objects in ${\cal A}$.

\section{Nori's Galois group}
\label{sec:N}
We begin by recalling the construction of Nori motives and the
associated motivic Galois group
(cf.~\cite{nori-lectures,huber-mueller:nori,huber-mueller:periods-nori}). We
also describe the universal property of the category of Nori motives
in a monoidal setting (Theorem~\ref{pro:N-universal}).

\subsection{Effective Nori motives}
A theory of motives, at the very least, should assign to any pair of
varieties $(X,Z)$ objects $\h_{n}(X,Z)$ ($n\in\N$), the $n$th relative
(homological) motive of the pair. Moreover, this assignment should be
functorial in the pair, and come with a ``boundary'' morphism
$\h_{n}(X,Z)\to \h_{n-1}(Z)$. We can formalize this using the
diagram\footnote{By a diagram we mean a directed graph.} $\diaeff$:
\begin{itemize}
\item its vertices are triples $(X,Z,n)$ where $X$ is a variety, $Z$
  is a closed subvariety of $X$, and $n$ is an integer;
\item there are two types of edges: a single edge from $(X,Z,n)$ to
  $(Z,W,n-1)$ for any triple $X\supset Z\supset W$, and edges
  $(X,Z,n)\to (X',Z',n)$ indexed by morphisms $f:X\to X'$ which
  restrict to $f:Z\to Z'$.
\end{itemize}
A representation of such a diagram is simply a morphism of directed
graphs $T:\diaeff\to {\cal C}$ into a category ${\cal C}$.

Instead of specifying explicitly what axioms such a representation
should satisfy, Nori's idea was to impose \emph{all} axioms satisfied
by a fixed homology theory. This is made precise by his Tannakian
theory for diagrams which asserts that associated to a representation
$T:\mathscr{D}\to\fMod{\Lambda}$ of any diagram $\mathscr{D}$ in the
category of finitely generated $\Lambda$-modules, there is a
$\Lambda$-linear abelian category ${\cal C}(T)$ together with a
factorization
\begin{equation*}
  \mathscr{D}\xrightarrow{\tilde{T}} {\cal C}(T)\xrightarrow{\ff}\fMod{\Lambda}
\end{equation*}
where $\ff$ is a faithful exact $\Lambda$-linear functor. Moreover,
this category is universal for such a factorization. This is applied
to the singular homology representation
$\h_{\bullet}:\diaeff\to\fMod{\Lambda}$ which takes a vertex $(X,Z,n)$
to the relative singular homology $\Lambda$-module
$\h_{n}(X^{\An},Z^{\An};\Lambda)$ of the associated topological spaces
on the $\C$-points (this uses $\sigma:k\inj\C$). To the single edge
$(X,Z,n)\to (Z,W,n-1)$ it associates the boundary map of the long
exact sequence of a triple
$\h_{n}(X^{\An},Z^{\An})\to \h_{n-1}(Z^{\An},W^{\An})$, and to an edge
$(X,Z,n)\to (X',Z',n)$ corresponding to $f:X\to X'$ it associates the
morphism in homology induced by
$f^{\An}:X^{\An}\to X^{\prime\An}$.\footnote{Here, and in the
  sequel we refrain from writing the coefficients in the homology
  when these can be guessed from the context. Also, we sometimes write
  $\h_{\bullet}(X,Z)$ instead of $\h_{\bullet}(X^{\An},Z^{\An})$.}
\begin{dfi}
  The category $\mathcal{C}(\h_{\bullet})$ is denoted by $\ehm$. It is
  the category of \emph{effective (homological) Nori motives}.
\end{dfi}

\begin{rem}\label{rem:hm-description}
  The construction of ${\cal C}(T)$ is easy to describe. A finite
  (full) subdiagram $\mathscr{F}\subset\mathscr{D}$ gives rise to a
  $\Lambda$-algebra
  $\End(T|_{\mathscr{F}})\subset \prod_{v}\End_{\Lambda}(T(v))$ of
  families of compatible (with respect to the edges) endomorphisms
  indexed over the vertices of $\mathscr{F}$. We then set
  \begin{equation*}
    \mathcal{C}(T)=\varinjlim_{\mathscr{F}\subset\mathscr{D}}\fMod{\End(T|_{\mathscr{F}})},
  \end{equation*}
  this filtered 2-colimit being indexed by the finite subdiagrams
  ordered by inclusion.

  In case $\Lambda$ is a principal ideal domain and $T$ takes values
  in finitely generated free $\Lambda$-modules, the dual
  $\mathcal{A}(T|_{\mathscr{F}})=\End(T|_{\mathscr{F}})^{\vee}$
  carries a canonical coalgebra structure for any finite subdiagram
  $\mathscr{F}\subset\mathscr{D}$. Moreover, ${\cal C}(T)$ can then be
  described as the category $\fcoMod{\mathcal{A}(T)}$ of
  $\mathcal{A}(T)$-comodules in $\fMod{\Lambda}$,\footnote{For our
    conventions regarding comodules see appendix~\ref{sec:comod}.}
  where
  \begin{equation*}
    \mathcal{A}(T)=\varinjlim_{\mathscr{F}\subset\mathscr{D}}\mathcal{A}(T|_{\mathscr{F}}).
  \end{equation*}
\end{rem}

\subsection{Stabilization}

In order to talk about Tate twists and to stabilize the category of
effective Nori motives we will need a monoidal version of Nori's
Tannakian theory for diagrams. The problem one faces is that the
tensor product of relative motives $\h_{n}(X,Z)\otimes \h_{n'}(X',Z')$
is typically not equal to $\h_{n+n'}((X,Z)\times (X',Z'))$ but only
related to the latter via the K\"{u}nneth spectral sequence. It is
therefore natural to restrict to pairs whose homology is concentrated
in a single degree.

Consider the full subdiagram $\diagoodeff$ of $\diaeff$ consisting of
\emph{good pairs}, \ie{} vertices $(X,Z,n)$ with $X\backslash Z$
smooth and $\h_{\bullet}(X,Z;\Z)$ a free abelian group concentrated in
degree $n$. It follows essentially from the Basic Lemma (recalled
in~§\ref{sec:basic-lemma}) that
${\cal C}(\h_{\bullet}|_{\diagoodeff})$ is canonically equivalent to
$\ehm$ (see~\cite[Pro.~3.2]{nori-lectures}
or~\cite[Cor.~1.7]{huber-mueller:nori} for a proof). Moreover, on
$\diagoodeff$ there is a ``commutative product structure with unit''
in the sense of~\cite{huber-mueller:nori} induced by the cartesian
product of varieties, and $\h_{\bullet}|_{\diagoodeff}$ is canonically
a \ugm representation (see appendix~\ref{sec:N-ugm} for a recollection
on these notions). This endows $\ehm$ with a monoidal structure such
that the functor $o$ mapping to $\fMod{\Lambda}$ is monoidal
(\cite[Thm.~4.1]{nori-lectures},
\cite[Pro.~B.16]{huber-mueller:nori}). As in the non-monoidal case it
has a universal property which we state in the following instance of
Nori's Tannaka duality theorem in the monoidal
setting~\ref{thm:N-ugm-universal} (cf.\
also~\cite{bruguieres-tannaka-nori,huber-mueller:nori,huber-mueller:periods-nori}).
\begin{thm}\label{pro:N-universal} 
  Let $\Lambda$ be a principal ideal domain. Suppose we are given a
  right exact monoidal abelian $\Lambda$-linear category
  $\mathcal{A}$\footnote{Hence $(A,\otimes)$ is a monoidal abelian
    $\Lambda$-linear category such that $\otimes$ is right exact
    $\Lambda$-linear in each variable.} together with a monoidal
  faithful exact $\Lambda$-linear functor
  $o:\mathcal{A}\to\fMod{\Lambda}$ and a \ugm representation
  $T:\diagoodeff\to\mathcal{A}$ such that the following diagram of
  solid arrows commutes.
  \begin{equation*}
    \xymatrix{\diagoodeff\ar[d]_{\tilde{\h}_{\bullet}}\ar[r]^-{T}&\mathcal{A}\ar[d]^{o}\\
      \ehm\ar[r]_-{\ff}\ar@{.>}[ur]&\fMod{\Lambda}}
  \end{equation*}
  Then there exists a monoidal functor $\ehm\to\mathcal{A}$ (unique up
  to unique monoidal isomorphism), represented by the dotted arrow in
  the diagram rendering the two triangles commutative (up to monoidal
  isomorphism).

  Moreover, this functor is faithful exact $\Lambda$-linear.
\end{thm}

Assume now that $\Lambda$ is a principal ideal domain. By the
discussion preceding the Theorem and Remark~\ref{rem:hm-description},
$\ehm$ can be identified with the category of comodules over the
coalgebra $\mathcal{A}(\h_{\bullet}|_{\diagoodeff})$. The monoidal
structure on $\ehm$ endows this coalgebra with the structure of a
commutative algebra turning it into a (commutative) bialgebra
(\cite[§4.2]{nori-lectures}, \cite[Pro.~B.16]{huber-mueller:nori}).
\begin{dfi}
  \begin{enumerate}
  \item We denote by
    $\Hn:=\Hn{\Lambda}:=\mathcal{A}(\h_{\bullet}|_{\diagoodeff})$
    \emph{Nori's effective motivic bialgebra}.
  \item \emph{Nori's motivic Hopf algebra} $\Hns:=\Hns{\Lambda}$ is the
    commutative Hopf algebra obtained from $\Hn$ by localizing (as an
    algebra) with respect to an element $\sn\in\Hn$ described below
    (see also~\cite[p.~13]{nori-lectures}).
  \item \emph{Nori's motivic Galois group} $\grpn:=\grpn{\Lambda}$ is
    the spectrum of $\Hns$. Thus it is a pro-group scheme over
    $\spec(\Lambda)$.
  \item The category $\hm:=\hm_{\Lambda}:=\fcoMod{\Hns{\Lambda}}$ is
    the category of \emph{(homological) Nori motives}.
  \end{enumerate}
\end{dfi}
It remains to describe the element $\sn\in\Hn$ corresponding to the
Tate twist. Choose an isomorphism
\begin{equation}\label{eq:SN}
  \h_{1}(\mathbb{G}_{m},\{1\})\xrightarrow{\sim}\Lambda.
\end{equation}
Then $\sn\in\Hn$ is the image of $1\in\Lambda$ under the composition
\begin{equation*}
  \Lambda\xleftarrow[\sim]{\eqref{eq:SN}}\h_{1}(\mathbb{G}_{m},\{1\})\xrightarrow{\coa}\Hn\otimes\h_{1}(\mathbb{G}_{m},\{1\})\xrightarrow[\sim]{\eqref{eq:SN}}\Hn\otimes\Lambda\cong\Hn,
\end{equation*}
where $\coa$ denotes the coaction of $\Hn$ on
$\h_{1}(\mathbb{G}_{m},\{1\})$. Clearly, $\sn$ does not depend on the
choice of~\eqref{eq:SN}.

\section{Betti realization for Morel-Voevodsky motives}
\label{sec:betti}

As explained in the introduction, we will work with motives without
transfers in order to use the six functors formalism. Based on
Voevodsky's original construction of the triangulated category of
motives with transfers $\dm(k)$, and following an insight by Morel,
this formalism has been worked out by Ayoub
in~\cite{ayoub07-thesis}. In this section we briefly recall the
construction of this category of motives, and the associated Betti
realization from~\cite{ayoub:betti} and~\cite{ayoub:galois1}. We also
prove a few results not stated there explicitly. In particular we
give a dg model for the Betti realization.

\subsection{Effective Morel-Voevodsky motives}

For the category of effective motives without transfers, and just as
with Nori motives, we will start with a certain category of
varieties. Instead of imposing axioms satisfied by the Betti
realization however, Morel and Voevodsky give the axioms explicitly
(Mayer-Vietoris and homotopy invariance). The precise construction may
profitably be viewed from the perspective of \emph{universal model dg
  categories} in the sense of~\cite{choudhury-gallauer:dg-hty}.

We start with a small category $\mathcal{C}$ with finite products,
endowed with a Grothendieck topology $\tau$ and $I\in\mathcal{C}$ an
``object parametrizing homotopies''; also fix any ring $\Lambda$. The
category $\U\mathcal{C}=\psh(\mathcal{C},\ch(\Lambda))$ of presheaves
on $\mathcal{C}$ with values in complexes of $\Lambda$-modules can be
endowed with three model structures (among others):
\begin{itemize}
\item The \emph{projective model structure} whose fibrations
  (resp.~weak equivalences) are objectwise epimorphisms
  (resp.~quasi-isomorphisms). This defines the model category underlying the universal model dg category associated to $\mathcal{C}$. Its homotopy category is just the
  derived category $\Der{\U\mathcal{C}}$.
\item The \emph{projective $\tau$-local model structure} arises from
  the projective model structure by Bousfield localization with
  respect to $\tau$-hypercovers. This defines the model category underlying the universal $\tau$-local model dg category associated to $\mathcal{C}$. Its homotopy category is equivalent
  to the derived category of $\tau$-sheaves on $\mathcal{C}$.
\item The \emph{projective $(I,\tau)$-local model structure} arises as
  a further Bousfield localization with respect to arrows
  $\Lambda(I\times Y)[i]\to\Lambda(Y)[i]$, where
  $\Lambda:\mathcal{C}\to \U\mathcal{C}$ denotes the ``Yoneda
  embedding'', and $Y\in\mathcal{C}$ and $i\in\Z$ are arbitrary. This defines the model category underlying the universal $(\tau,I)$-local model dg category associated to $\mathcal{C}$. Its
  homotopy category is a $\Lambda$-linear unstable (or ``effective'')
  $I$-homotopy theory of $(\mathcal{C},\tau)$.
\end{itemize}
In each case, the model category is stable and monoidal (for the
objectwise tensor product) hence the homotopy categories are
triangulated monoidal. The following examples will be of interest to
us (notation is explained subsequently):
\begin{center}
\setlength{\tabcolsep}{1em}
\renewcommand{\arraystretch}{1.2}
    \begin{tabular}{ l l  l  l }
      $\mathcal{C}$ & $\tau$ & $I$ & $\Lambda$-linear unstable $I$-homotopy theory \\ \hline
      $\sm{X}$ & $\mathrm{Nis}$ or $\et$ & $\A^{1}_{X}$ & $\dae(X)=\dae{\Lambda}(X)$\\
      $\smaff{X}$ & $\mathrm{Nis}$ or $\et$ & $\A^{1}_{X}$ & $\daeaff(X)=\daeaff{\Lambda}(X)$\\
      $\ansm{X}$ & $\mathrm{usu}$ & $\Done_{X}$ & $\andae(X)=\andae{\Lambda}(X)$ \\ 
      $\open{X}$ & $\mathrm{usu}$ & $X$ & $\Der{X,\Lambda}$
    \end{tabular}
\end{center}
Here, in the first two examples, $X$ is a scheme, and $\sm{X}$ (resp.\
$\smaff{X}$) denotes the category of smooth schemes over $X$ (resp.\
which are affine in the absolute sense) endowed with the Nisnevich or
the étale topology. In case $X=\spec(k)$, we denote this category by
$\sm$ (resp.\ $\smaff$).
\begin{dfi}
  $\dae(X)$ is the category of \emph{effective Morel-Voevodsky
    $X$-motives}.
\end{dfi}
In the few instances where the topology chosen plays any role, we will
make this explicit. Also, if $X=\spec(k)$ then we simply write
$\dae$. Since every scheme is covered by affine open subschemes, we
obtain the following easy fact
(see~\cite[Cor.~5.16]{choudhury-gallauer:dg-hty} for a proof).
\begin{lem}\label{lem:affine-da}
  The canonical inclusion $\smaff{X}\to\sm{X}$ induces a triangulated
  monoidal equivalence $\daeaff(X)\xrightarrow{\sim}\dae(X)$.
\end{lem}

In the third example above, $X$ is a complex analytic space, \ie{} a
``complex space'' in the sense of~\cite{grauert-remmert-1984} which is
supposed to be denumerable at infinity, and $\ansm{X}$ denotes the
category of complex analytic spaces smooth over $X$ with the topology
$\mathrm{usu}$ given by open covers. If $X$ is the terminal object
$\pt$, then $\ansm{X}$ is denoted simply by $\ansm$. $\Done$ denotes
the open unit disk considered as a complex analytic space. As above,
the $\Done$-homotopy theory is denoted by $\andae$ in case
$X=\pt$.

Finally, in the fourth example $X$ denotes a topological space,
$\open{X}$ the category associated to the preorder of open subsets of
$X$. It is endowed with the topology $\mathrm{usu}$ given by open
covers. The $(X,\mathrm{usu})$-local and the $\mathrm{usu}$-local
model structures evidently agree, and their homotopy category
$\Der{X}=\Der{X,\Lambda}$ is (canonically identified with) the derived
category of sheaves on $X$.

\subsection{Effective Betti realization}

The Betti realization due to Ayoub in~\cite{ayoub:betti} will now link
the categories just introduced, as follows. For a complex analytic
space $X$ there is an obvious inclusion
$\iota_{X}:\open{X}\to\ansm{X}$ which defines a morphism of sites and
induces a Quillen equivalence (\cite[Thm.~1.8]{ayoub:betti})
\begin{equation*}
  (\iota_{X}^{*},\iota_{X*}):\U(\open{X})/\mathrm{usu}\to\U(\ansm{X})/(\Done_{X},\mathrm{usu}).
\end{equation*}
If $X=\spec(k)$, the left adjoint takes a complex to the associated
constant presheaf and is denoted by $(\bullet)_{\mathrm{cst}}$, while
the right adjoint is the global sections functor and accordingly
denoted by $\Gamma$.

Any scheme $Y$ gives rise to a complex analytic space $Y^{\An}$,
namely the topological space $(Y\times_{k,\sigma}\C)(\C)$ with the
natural complex analytic structure. We obtain an analytification
functor $\Anal_{X}:\sm{X}\to\ansm{X^{\An}}$ which induces Quillen
adjunctions
\begin{align*}
  (\Anal_{X}^{*},\Anal_{X*}):&\U(\sm{X})\to
  \U(\ansm{X^{\An}})
\end{align*}
for the corresponding model structures considered above. The left
adjoint $\Anal_{X}^{*}$ in fact preserves $(I,\tau)$-local weak
equivalences (see~\cite[Rem.~2.57]{ayoub:galois1}).

\begin{dfi}\label{dfi:btie}
  The \emph{effective Betti realization} is the composition
  \begin{equation*}
    \btie:\dae(X)\xrightarrow{\Anal_{X}^{*}}\andae(X)\xrightarrow[\sim]{\dR\iota_{X*}}\Der{X}.
  \end{equation*}
\end{dfi}
By construction, this is a triangulated monoidal functor.

\subsection{Stabilization}
\label{sec:betti.st}

Motives will be obtained from effective motives by a stabilization
process which we again describe in the abstract setting first. Let
${\cal M}$ be a cellular left-proper monoidal model category and
$T\in{\cal M}$ a cofibrant object. The category $\spt_{T}\mathcal{M}$
of symmetric $T$-spectra in $\mathcal{M}$ admits the following two
model structures (among others):
\begin{itemize}
\item The \emph{projective unstable model structure} whose fibrations
  (resp.\ weak equivalences) are levelwise fibrations (resp.\ weak
  equivalences).
\item The \emph{projective stable model structure} arises from the
  unstable one by Bousfield localization with respect to morphisms
  $\sus{n+1}_{T}(T\otimes K)\to \sus{n}_{T}(K)$ for cofibrant objects
  $K\in {\cal M}$.
\end{itemize}
Here, $(\sus{i}_{T},\ev{i}):\U\mathcal{C}\to\spt_{T}\U\mathcal{C}$
denotes the canonical adjunction, $\ev{i}$ being evaluation at level
$i$. For the details (also concerning the existence of the model
structures) we refer to~\cite{hovey:spectra}. Again, the model
categories are both monoidal, and if ${\cal M}$ was stable then so is
$\spt_{T}{\cal M}$. If not mentioned explicitly otherwise, when we
refer to \emph{the} model structure on $\spt_{T}{\cal M}$ we mean the
stable one. In $\spt_{T}{\cal M}$, tensoring with $T$ becomes a
Quillen equivalence, and it should be thought of as the universal such
model category although this is not quite true in the obvious sense
(cf.~\cite[§9]{hovey:spectra}; see also \cite{robalo:kthy-motives} for the $(\infty,1)$-categorical version of the picture, and where such a universal property \emph{can} be proven). Still, we call it the
$T$-stabilization of ${\cal M}$.

In the algebraic geometric examples above we choose $T_{X}$ to be a
cofibrant replacement of $\Lambda(\A^{1}_{X})/\Lambda(\G_{m,X})$.
\begin{dfi}
  The resulting $T_{X}$-stable $\A^{1}$-homotopy theory of
  $(\sm{X},\tau)$, denoted by $\da(X)=\da{\Lambda}(X)$, is the
  category of \emph{Morel-Voevodsky $X$-motives}.
\end{dfi}
As before (Lemma~\ref{lem:affine-da}), the affine version
$\daaff(X)$ is canonically equivalent. Again, we leave the topology
implicit most of the time, and in case $X=k$ we also write
$\da$. There is the notion of a compact motive, namely an object in
$\da$ which is compact in the sense of triangulated categories, and
the full subcategory of compact motives forms a thick triangulated
subcategory.

In the analytic setting, stabilization is performed with respect to a
cofibrant replacement $T_{X}$ of the quotient presheaf
$\Lambda(\A^{1,\An}_{X})/\Lambda(\G^{\An}_{m,X})$. The resulting
homotopy category is denoted by $\anda(X)=\anda{\Lambda}(X)$ (and
again simply $\anda$ in case $X=\pt$). \cite[Lem.~1.10]{ayoub:betti}
together with \cite[Thm.~9.1]{hovey:spectra} show that the adjunction
\begin{equation*}
  (\sus_{T_{X}},\ev):\U(\ansm{X})/(\Done_{X},\mathrm{usu})\to \spt_{T_{X}}\U(\ansm{X})/(\Done_{X},\mathrm{usu})
\end{equation*}
defines a Quillen equivalence. Moreover, the analytification functor
passes to the level of symmetric spectra and preserves stable
$(I,\tau)$-local equivalences.
\begin{dfi}\label{dfi:bti}
  The \emph{Betti realization} is the composition
  \begin{equation*}
    \bti:\da(X)\xrightarrow{\Anal_{X}^{*}}\anda(X)\xrightarrow[\sim]{\dR\iota_{X*}\dR\ev}\Der{X}.
  \end{equation*}
\end{dfi}
Again, it is a triangulated monoidal functor.

\subsection{Six functor formalism}

We recall that the six functors constitute a formalism on the
categories $\da(X)$ for schemes $X$ which associates to any morphism
of schemes $f:X\to Y$ adjunctions
\begin{equation}\label{eq:4ff-adjunctions}
  (\dL f^{*},\dR f_{*}):\da(Y)\to\da(X),\qquad (\dL f_{!},f^{!}):\da(X)\to\da(Y),
\end{equation}
and which endows $\da(X)$ with a closed monoidal structure
\begin{equation*}
  (\otimes^{\dL}, \dR\uhm).\footnote{In the literature the symbols
    $\dL$ and $\dR$ indicating that some left or right derivation takes
    place are often dropped from the notation. For us however, the
    distinction between the derived and underived functors will be
    important which is why we stick to the clumsier notation.}
\end{equation*}
All these functors are triangulated. The formalism governs the
relation between them, \eg{} under what conditions two of these
functors can be identified or when they commute. Some of these
relations are given explicitly
in~\cite[Sch.~1.4.2]{ayoub07-thesis}. We will also heavily use
the part concerning duality. Recall that on compact motives there is a
contravariant autoequivalence $(\bullet)^{\vee}$ which exchanges the
two adjunctions in~\eqref{eq:4ff-adjunctions} so that for example
$(\dR f_{*}M)^{\vee}\cong \dL f_{!}M^{\vee}$ for any compact motive
$M$ (see~\cite[Thm.~2.3.75]{ayoub07-thesis}). 

The same formalism is available in the analytic
(see~\cite{ayoub:betti}) and in the topological setting (at least if
the topological space is locally compact, see
\eg{}~\cite{kashiwara-schapira-90}). The main result of Ayoub
in~\cite{ayoub:betti} is that the Betti realization is compatible with
these, at least if one restricts to compact motives.

\subsection{dg enhancement}
\label{sec:betti.betti}

In the remainder of the section we will exhibit the (effective) Betti
realization as the derived functor of a left Quillen dg functor
(Proposition~\ref{pro:sgs-bti}). This will be used in §\ref{sec:AN} to
construct a motivic realization $\dae\to\Der{\iehm}$. Our argument here relies on our discussion of left dg Kan extensions in~\cite{choudhury-gallauer:dg-hty}, and we will frequently refer to that paper.

Let $X$ be a complex analytic space. Denote by $\sgs(X)$ the complex
of singular chains in $X$ (with $\Lambda$-coefficients). This extends
to a lax monoidal functor $\sgs$ on topological spaces in virtue of
the Eilenberg-Zilber map (cf.~\cite[VI,~12]{dold:lectures-at}). Its
``left dg Kan extension'' (rather, the functor underlying
the left dg Kan extension of~\cite[Fact~2.1]{choudhury-gallauer:dg-hty})
is denoted by
\begin{equation*}
  \sgs^{*}:\U(\ansm)\to\ch(\Lambda).
\end{equation*}
It possesses an induced lax-monoidal structure, by~\cite[Lemma~2.2]{choudhury-gallauer:dg-hty}. 
Moreover, for each complex manifold $X$, $\sgs(X)$ is projective
cofibrant hence $\bullet\otimes \sgs(X)$ is a left Quillen functor. It
follows from~\cite[Lemma~2.5]{choudhury-gallauer:dg-hty}
that $\sgs^{*}$ is also a left Quillen functor with respect to the
projective model structures.
\begin{pro}\label{pro:sgs-bti}
  $\dL\sgs^{*}$ takes $(\Done,\mathrm{usu})$-local equivalences to
  quasi-isomorphisms and the induced functor
  \begin{equation*}
    \dae\xrightarrow{\Anal^{*}}\andae\xrightarrow{\dL\sgs^{*}}\Der{\Lambda}
  \end{equation*}
  is isomorphic to $\btie$ as monoidal triangulated functor. In
  particular, the following triangle commutes up to a monoidal
  isomorphism:
  \begin{equation*}
    \xymatrix{\sm\ar[r]^-{\Lambda(\bullet)}\ar[dr]_{\sgs\circ\Anal}&\dae\ar[d]^{\btie}\\
      &\Der{\Lambda}}
  \end{equation*}
\end{pro}
\begin{proof}
  In fact, we will deduce the first statement from the second.

  For this let us recall the ``singular analytic complexes''
  constructed in~\cite[§2.2.1]{ayoub:galois1}. We denote by $\Done(r)$
  the open disk of radius $r$ centered at the origin (thus
  $\Done=\Done(1)$) and by $\Disk^{n}(r)$ the $n$-fold cartesian
  product ($n\geq 0$). Letting $r>1$ vary we obtain pro-complex
  manifolds $\overline{\Disk}^{n}=(\Disk^{n}(r))_{r>1}$. There is an
  obvious way to endow the family $(\overline{\Disk}^{n})_{n\geq 0}$
  with the structure of a cocubical object in the category of
  pro-complex manifolds (see~\cite[Déf.~2.19]{ayoub:galois1}). For any
  complex manifold $X$ one then deduces a cubical $\Lambda$-module
  $\uhom(\overline{\Disk}^{\bullet},X)$, where the latter in degree
  $n$ is given by $\varinjlim_{r>1}\Lambda\ansm(\Disk^{n}(r),X)$. The
  associated simple complex (see~\cite[Déf.~A.4]{ayoub:galois1}) is
  called the singular analytic complex associated to $X$, and is
  denoted by $\sgd(X)$. It clearly extends to a functor
  \begin{equation*}
    \sgd:\ansm\to \ch(\Lambda),
  \end{equation*}
  and admits a natural lax monoidal structure induced by the
  association
  \begin{equation*}
    (a:\Disk^{m}(r)\to X, b:\Disk^{n}(r)\to Y)\longmapsto (a\times
    b:\Disk^{m+n}(r)\to X\times Y).
  \end{equation*}
  We would now like to prove that $\sgd$ and $\sgs$ are monoidally
  quasi-isomorphic, and for this we need a third, intermediate
  singular complex.

  For any real number $r>1$, denote by $\mathbb{I}^{1}(r)$ the open
  interval $(-r,r)$. Set $\mathbb{I}^{1}=\mathbb{I}^{1}(1)$. There is
  an obvious analytic embedding $\mathbb{I}^{1}(r)\to \Done(r)$ of
  real analytic manifolds. Denote by $\mathbb{I}^{n}(r)$ the $n$-fold
  cartesian product of $\mathbb{I}^{1}(r)$. Letting $r>1$ vary we
  obtain a pro-real analytic manifold
  $\overline{\mathbb{I}}^{n}$. There is an obvious embedding of
  pro-real analytic manifolds
  $\overline{\mathbb{I}}^{n}\to\overline{\mathbb{D}}^{n}$ for each
  $n$, and by restriction this induces the structure of a cocubical
  object in pro-real analytic manifolds on
  $\overline{\mathbb{I}}^{\bullet}$.

  Denoting by $i$ the inclusion $\ansm\inj\anmfd$ we
  obtain a monoidal natural transformation
  \begin{equation}
    \sgd\to\sgi\circ i,\label{eq:sg-D-I}
  \end{equation}
  and it suffices to prove that this is sectionwise a
  quasi-isomorphism. Indeed, a similar argument as in~\cite[App.~A,
  §2, Thm.~2.1]{massey-basic-alg-top-1991} shows that the right hand
  side is monoidally quasi-isomorphic to the analogous functor of
  cubical complexes of \emph{continuous} functions. And the latter is
  in turn monoidally quasi-isomorphic to $\sgs$,
  by~\cite[Thm.~5.1]{GNPR-monoidal-acyclic-models}.

  Following~\cite{ayoub:galois1}, we denote the ``left dg Kan
  extension'' of $\sgd$ again by the same symbol
  $\sgd:\U\ansm\to\ch(\Lambda)$. (That this indeed coincides with the
  functor in~\cite{ayoub:galois1} follows from~\cite[Lemma~3.21]{choudhury-gallauer:dg-hty}. 
  Cocontinuity is a consequence of~\cite[Lem.~A.3]{ayoub:galois1}.)
  \cite[Thm.~2.23]{ayoub:galois1} together with~\cite[Cor.~2.26,
  2.27]{ayoub:galois1} show that $\sgd$ takes
  $(\Done,\mathrm{usu})$-local equivalences to
  quasi-iso\-mor\-phisms. The same argument also shows that
  $\mathrm{Sg}^{\mathbb{I}}$ takes $(\mathbb{I},\mathrm{usu})$-local
  equivalences to quasi-isomorphisms. Now, let's start with a complex
  manifold $X$. We want to prove that~\eqref{eq:sg-D-I} applied to $X$
  is a quasi-isomorphism. For this, choose a $\mathrm{usu}$-hypercover
  $X_{\bullet}\to X$ of complex manifolds such that each representable
  in each degree is contractible. This can also be considered as a
  $\mathrm{usu}$-hypercover of real analytic manifolds, and by what we
  just discussed, the two horizontal arrows in the following
  commutative square are quasi-isomorphisms:
  \begin{equation*}
    \xymatrix{\sgd(X_{\bullet})\ar[r]\ar[d]&\sgd(X)\ar[d]\\
      \sgi(i(X_{\bullet}))\ar[r]&\sgi(i(X))}
  \end{equation*}
  By~\cite[Lemma~4.2]{choudhury-gallauer:dg-hty}, 
  we reduce to show that for any contractible complex manifold $Y$,
  $\sgd(Y)\to\sgi(i(Y))$ is a quasi-isomorphism, which is
  easy. 

  By~\cite[Cor.~2.26, Pro.~2.83]{ayoub:galois1}, $\btie$ is isomorphic
  to $\sgd\circ\Anal^{*}$ as triangulated monoidal functor hence the
  discussion above implies the second statement of the
  proposition. The first statement can now be deduced as follows. From
  the monoidal quasi-isomorphism $\sgd\sim \sgs$ we obtain
  triangulated monoidal isomorphisms
  (cf.~\cite[Lemma~2.2]{choudhury-gallauer:dg-hty})
  \begin{equation*}
    \sgd\cong\dL\sgd\cong \dL(\sgs)^{*}:\Der{\U\ansm}\to\Der{\Lambda}.
  \end{equation*}
  Indeed, the second isomorphism can be checked on representables
  (these are compact generators of $\Der{\U\ansm}$ by~\cite[Lemma~3.20]{choudhury-gallauer:dg-hty}) and these objects are
  cofibrant.
\end{proof}

\begin{rem}\label{rem:bounded-above-usu-fibrant}
  Using the topological singular complex we can construct an explicit
  fibrant model for the unit spectrum in $\anda$ as follows. Denote by
  $\sgsd$ the presheaf of complexes on $\ansm$ which takes a complex
  manifold $X$ to $\sgs(X)^{\vee}$. Let $U=(\Pe^{1}\times
  \Pe^{1})\backslash\Delta(\Pe^{1})$, and let $u$ be a rational point
  of $U$ over $\Pe^{1}\times\{\infty\}$. As
  in~\cite[§2.3.1]{ayoub:galois1}, we can use
  $T^{\An}=\Lambda(U^{\An})/\Lambda(u^{\An})$ to form symmetric
  spectra (hence $T^{\An}$ is a cofibrant replacement of
  $\Lambda(\A^{1,\An})/ \Lambda(\G_{m}^{\An})$). Fix
  $\hat{\beta}\in\sgsd_{-2}(U^{\An},u^{\An};\Lambda)$ whose class in
  $\h^{2}(U^{\An},u^{\An};\Lambda)\cong\Lambda$ is a generator. Define
  a symmetric $T^{\An}$-spectrum $\sgsdbb$ which in level $n$ is
  $\sgsd[-2n]$ with the trivial $\Sigma_{n}$-action, and whose bonding
  maps are given by the adjoints of the quasi-isomorphism
  \begin{align*}
    \hat{\beta}\times\bullet:\sgsd(X)[-2n]\to \sgsd((U,u)\times
    X)[-2(n+1)]
  \end{align*}
  for any complex manifold $X$.

  The canonical morphism $\Lambda_{\mathrm{cst}}\to\sgsd$ induces by
  adjunction a morphism of symmetric spectra
  $\sus_{T^{\An}}\Lambda_{\mathrm{cst}}\to\sgsdbb$ which in level $n$
  is given by the composition
  \begin{align*}
    (T^{\An})^{\otimes n}\otimes\Lambda_{\mathrm{cst}}&\to(T^{\An})^{\otimes
      n}\otimes\sgsd\\
    &\xrightarrow{\id\otimes(\hat{\beta}\times\bullet)^{n}}(T^{\An})^{\otimes
      n}\otimes\uhom((T^{\An})^{\otimes
      n},\sgsd[-2n])\\
    &\xrightarrow{\mathrm{ev}}\sgsd[-2n].
  \end{align*}
  The first arrow is a $\mathrm{usu}$-local equivalence, the second
  arrow is a quasi-isomorphism, and the third is a
  $(\Done,\mathrm{usu})$-local equivalence since $T^{\An}$ is
  invertible in $\andae$, by~\cite[Lem.~1.10]{ayoub:betti}. It follows
  that $\sus_{T^{\An}}\Lambda_{\mathrm{cst}}\to\sgsdbb$ is a levelwise
  $(\Done,\mathrm{usu})$-local equivalence. Since the source is an
  $\Omega$-spectrum so is $\sgsdbb$. Also, since
  $\Lambda_{\mathrm{cst}}$ is $\Done$-local so is $\sgsdbb$
  levelwise. Finally, for any $\mathrm{usu}$-hypercover
  $X_{\bullet}\to X$ of a complex manifold $X$,
  $\sgsd(X)\to\sgsd(X_{\bullet})$ is a quasi-isomorphism which proves
  that $\sgsdbb$ is levelwise $\mathrm{usu}$-fibrant.

  Summing up, we have proved that $\sgsdbb$ is a projective stable
  $(\Done,\mathrm{usu})$-fibrant replacement of
  $\sus_{T^{\An}}\Lambda_{\mathrm{cst}}$.
\end{rem}

\section{Ayoub's Galois group}
\label{sec:A}
We recall here the construction of Ayoub's motivic Galois group
in~\cite{ayoub:galois1}. In section~1 of that paper he develops a weak
Tannaka duality theory which allows to factor certain monoidal
functors $f:{\cal M}\to {\cal E}$ between monoidal categories
universally as
\begin{equation}\label{eq:A-factorization}
  {\cal M}\xrightarrow{\tilde{f}} \coMod{{\cal H}(f)}\xrightarrow{\ff}{\cal E}
\end{equation}
for a commutative bialgebra ${\cal H}(f)\in{\cal E}$, where $o$ is the
forgetful functor, and where both functors in the factorization are
monoidal. This was applied in~\cite{ayoub:galois1} to the monoidal
(effective) Betti realization functor
\begin{equation*}
  \bti:\da\to\Der{\Lambda}\qquad  (\text{resp. }\btie:\dae\to\Der{\Lambda}),
\end{equation*}
see~Definitions~\ref{dfi:btie} and~\ref{dfi:bti}.
\begin{dfi}
  \begin{enumerate}
  \item \emph{Ayoub's effective motivic bialgebra} is
    $\ealga:=\ealga{\Lambda}:={\cal H}(\btie)\in\Der{\Lambda}$.
  \item \emph{Ayoub's motivic Hopf algebra} is
    $\alga:=\alga{\Lambda}:={\cal H}(\bti)\in \Der{\Lambda}$. It is
    indeed a (commutative) Hopf algebra, as shown
    in~\cite{ayoub:galois1}.
  \end{enumerate}
\end{dfi}
The bialgebras do not depend (up to canonical isomorphism) on the
topology chosen.  Explicitly, as objects in $\Der{\Lambda}$ they are
given by $\alga=\bti\rbti\Lambda$ and $\ealga=\btie\rbtie\Lambda$.

We said above that these bialgebras enjoy a universal property; let us
recall the precise statement for the effective case (an analogous
statement holds in the stable situation but we will not use this).
\begin{fac}[{\cite[Pro.~1.55]{ayoub:galois1}}]\label{pro:A-universal}
  Suppose we are given a commutative bialgebra $K$ in $\Der{\Lambda}$
  and a commutative diagram in the category of monoidal
  categories
  \begin{equation*}
    \xymatrix{\dae\ar[dr]_{\btie}\ar[r]^-{f}&\coMod{K}\ar[d]^{\ff}\\
      &\Der{\Lambda}}
  \end{equation*}
  where $o$ is the forgetful functor, such that $f(A_{\mathrm{cst}})$
  is the trivial $K$-comodule associated to $A$, for any
  $A\in\Der{\Lambda}$. Then there exists a unique morphism of bialgebras
  $\ealga\to K$ making the following diagram commutative:
  \begin{equation*}
    \xymatrix{\dae\ar[d]_{\tbtie}\ar[r]^-{f}&\coMod{K}\ar[d]^{\ff}\\
      \coMod{\ealga}\ar[r]_-{\ff}\ar[ru]&\Der{\Lambda}}
  \end{equation*}
\end{fac}

Now, consider the functor $\h_{0}:\Der{\Lambda}\to \Mod{\Lambda}$
which associates to a complex its 0th
homology. By~\cite[Cor.~2.105]{ayoub:galois1}, the homology of
$\ealga$ and $\alga$ is concentrated in non-negative degrees and it
follows that the bialgebra (resp. Hopf algebra) structure descends to
the 0th homology of $\ealga$ (resp. $\alga$).
\begin{dfi}
  \begin{enumerate}
  \item We set $\Ha:=\h_{0}(\ealga)\in\Mod{\Lambda}$. This is thus a
    bialgebra over $\Lambda$.
  \item We set $\Has:=\h_{0}(\alga)\in\Mod{\Lambda}$. This is thus a
    Hopf algebra over $\Lambda$.
  \item \emph{Ayoub's motivic Galois group} $\grpa:=\grpa{\Lambda}$ is
    the spectrum of $\Has$. Thus it is a pro-group scheme over
    $\spec(\Lambda)$.
  \end{enumerate}
\end{dfi}

\begin{rem}
  By~\cite[Thm.~2.14]{ayoub:galois1}, $\alga$ (resp.\ $\Has$) is
  obtained by localization from $\ealga$ (resp.\ $\Ha$), as
  follows. Choose an isomorphism
  \begin{equation}\label{eq:SA}
    \btie(T[2])\xrightarrow{\sim}\Lambda.
  \end{equation}
  We then let $\sa\in\ealga$ be the image of $1\in\Lambda$ under the
  composition
  \begin{equation*}
    \Lambda\xleftarrow[\sim]{\eqref{eq:SA}}\btie(T[2])\xrightarrow{\coa}\ealga\otimes\btie(T[2])\xrightarrow[\sim]{\eqref{eq:SA}}\ealga\otimes\Lambda\cong \ealga,
  \end{equation*}
  where $\coa$ denotes the coaction of $\ealga$ on
  $\btie(T[2])$. Clearly, $\sa$ does not depend on the choice of the
  isomorphism~\eqref{eq:SA}. By~\cite[Thm.~2.14]{ayoub:galois1},
  $\alga$ is the sequential homotopy colimit of the diagram
  \begin{equation*}
    \ealga\xrightarrow{\sa\times \bullet}  \ealga\xrightarrow{\sa\times \bullet}\cdots.
  \end{equation*}
  Applying $\h_{0}$ we see that $\Has=\Ha[\sa^{-1}]$ as an algebra.
\end{rem}

In order to apply the results on the category of $\Haors$-comodules in
appendix~\ref{sec:comod} we will need the following result.
\begin{lem}\label{lem:A-flat}
  Let $\Lambda$ be a principal ideal domain. Then $\Ha$ and $\Has$ are
  flat $\Lambda$-modules.
\end{lem}
\begin{proof}
  The proof is the same in both cases; we do it for
  $\Ha$. By~\cite[Cor.~1.27]{ayoub:galois2}, $\ealga$ sits in a
  distinguished triangle
  \begin{equation*}
    C'\to\ealga\to C\to C'[-1],
  \end{equation*}
  where $C$ is a complex in $\Der{\Lambda\otimes_{\Z}\Q}$. (Explicitly,
  $C'=C^{0}(\mathrm{Gal}(\overline{k},k),\Lambda)$, the
  $\Lambda$-module of locally constant functions on the absolute
  Galois group of $k$ with values in $\Lambda$, which maps canonically
  to $\ealga$. Essentially due to the Rigidity Theorem of
  Suslin-Voevodsky, this map is a quasi-isomorphism for torsion
  coefficients.) Looking at the associated long exact sequence in
  homology one sees that all homologies of $\alga$ must be
  torsion-free thus flat.
\end{proof}

\begin{rem}
  The Betti realization can be constructed in a similar way also for
  Voevodsky motives (see~\cite[§1.1.2]{ayoub:galois2}), and the same
  weak Tannakian formalism applies to give two bialgebras in
  $\Der{\Lambda}$. It is proved in~\cite[Thm.~1.13]{ayoub:galois2} that
  they are canonically isomorphic to $\ealga$ and $\alga$,
  respectively. In the case of the Hopf algebras (and this is the case
  we are chiefly interested in), this follows from the fact that the
  canonical functor
  \begin{equation*}
    \da^{\mathrm{\acute{e}t}}(k)\to\dm^{\mathrm{\acute{e}t}}(k)
  \end{equation*}
  is an equivalence.
\end{rem}

\section{Motivic representation}
\label{sec:NA}

The goal of this section is to factor the homology representation
$\h_{\bullet}:\diaeff\to\fMod{\Lambda}$ through the Betti realization
$\h_{0}\circ\bti:\da\to\Mod{\Lambda}$
(Propositions~\ref{pro:NA-eff}, \ref{pro:NA-eff-ugm}) in order
to obtain a morphism of bialgebras $\na:\Hns\to\Has$. Let us see how
to derive a solution $\rep:\diaeff\to\da$ to this task.

We saw in the previous section that for \emph{smooth} schemes $X$,
$\btie \Lambda(X)$ computes the Betti homology of $X$. A first guess
might be that for \emph{any} scheme $X$, $\btie\Lambda\hom(\bullet,X)$
also does. We don't know whether this is true. Instead we notice
that, for $X$ with smooth structure morphism $\pi:X\to k$, there are
canonical isomorphisms
\begin{equation*}
  \dL\sus_{T}\Lambda(X)\cong \dL\pi_{\#}\pi^{*}\Lambda\cong\dL\pi_{!}\pi^{!}\Lambda
\end{equation*}
in $\da$. The last expression makes sense for \emph{any} scheme $X$,
and we will prove below that the Betti realization of this object
indeed computes the Betti homology of $X$. We should remark that there
is nothing original about this idea. The object
$\dL\pi_{!}\pi^{!}\Lambda$ (and not the presheaf
$\Lambda\hom(\bullet,X)$) is commonly considered to be the ``correct''
representation of $X$ in $\da$, and is therefore also called the
\emph{(homological) motive of $X$}. The six functors formalism
also allows to naturally define a relative motive associated to a pair
of schemes, and this will yield the representation $\rep$ we were looking
for.

\subsection{Construction}

Let $(X,Z,n)$ be a vertex in Nori's diagram of pairs. Fix the
following notation:
\begin{equation*}
  Z\xrightarrow{i}X\xleftarrow{j}U,\qquad \pi:X\to k,
\end{equation*}
where $U=X\backslash Z$ is the open complement. Set
\begin{equation*}
  \rep(X,Z,n)=\dL\pi_{!}\dR j_{*}j^{*}\pi^{!}\Lambda[n]\in\da.
\end{equation*}
This extends to a representation $\rep:\diaeff\to\da$ as follows:
\begin{itemize}
\item The first type of edge in $\diaeff$ is $(X,Z,n)\to
  (Z,W,n-1)$. We have the following distinguished (``localization'')
  triangle (of endofunctors) in $\da(X)$ (and similarly for the pair
  $(Z,W)$):
  \begin{equation}\label{eq:loc-triangle-hom}
    i_{!}i^{!}\xrightarrow{\mathrm{adj}}
    \id\xrightarrow{\mathrm{adj}} \dR
    j_{*}j^{*}\xrightarrow{\partial} i_{!}i^{!}[-1].
  \end{equation}
  (Here, as in the sequel, $\mathrm{adj}$ denotes the unit or counit
  of an adjunction.) Applying $\dL\pi_{!}$ and evaluating at
  $\pi^{!}\Lambda[n]$, we therefore obtain a morphism
  \begin{align*}
    \rep(\partial):\rep(X,Z,n)&\xrightarrow{\partial}\rep(Z,\emptyset,n-1)\\
    &\xrightarrow{\mathrm{adj}}\rep(Z,W,n-1).
  \end{align*}
\item The second type of edge $(X,Z,n)\to (X',Z',n)$ is induced from a
  morphism of varieties $f:X\to X'$ with $f(Z)\subset Z'$. We have the
  following commutative diagram of solid arrows in (endofunctors of)
  $\da(X')$:
  \begin{equation*}
    \xymatrix{\dL f_{!}i_{!}i^{!}f^{!}\ar[r]\ar[d]^{\mathrm{adj}}&\dL
      f_{!}f^{!}\ar[r]\ar[d]^{\mathrm{adj}}&\dL f_{!}\dR
      j_{*}j^{*}f^{!}\ar[r]\ar@{.>}[d]&\dL f_{!}i_{!}i^{!}f^{!}[-1]\ar[d]^{\mathrm{adj}}\\
      i'_{!}i'^{!}\ar[r]&\mathrm{id}\ar[r]&\dR j'_{*}j^{'*}\ar[r]&i'_{!}i'^{!}[-1]}
  \end{equation*}
  where the rows are distinguished (``localization'') triangles, and
  where the dotted arrow is the unique morphism making the vertical
  arrows into a morphism of triangles. (Uniqueness follows from the
  isomorphism $\dL f_{!}i_{!}\cong i'_{!}\dL (f|_{Z})_{!}$ and the fact that
  there are no non-zero morphisms from $i'_{!}$ to $\dR j'_{*}$.)
  After applying $\dL\pi'_{!}$, shifting by $n$, and evaluating at
  $\pi^{\prime!}\Lambda$ this dotted arrow gives the morphism
  $\rep(f):\rep(X,Z,n)\to \rep(X',Z',n)$ associated to $f$.
\end{itemize}
We will prove in a moment that this representation has the expected
properties. Before doing so we would like to recall the following
classical result.
\begin{fac}[{\cite{hironaka:triangulation}}]\label{lem:varieties-topology}
  Let $(X,Z)$ be a pair of varieties. Then its analytification
  $(X^{\An},Z^{\An})$ is a locally finite CW-pair. In particular,
  $X^{\An}$ and $Z^{\An}$ are paracompact, locally contractible, and
  locally compact.
\end{fac}

\begin{pro}\label{pro:NA-eff}
  Suppose that $\Lambda$ is a principal ideal domain. Then there is an
  isomorphism of representations
  \begin{equation*}
    \xymatrix@C=5em{\diaeff\ar[r]^-{\h_{0}\tbti
        \rep}\ar[dr]_{\h_{\bullet}}&\fcoMod{\Has}\ar[d]^{\ff}\\
      &\fMod{\Lambda}}
  \end{equation*}
\end{pro}
\begin{proof}
  By Fact~\ref{lem:varieties-topology} and \ref{lem:coh-pair-model}
  the complex of relative singular cochains
  $\sgs(X^{\An},Z^{\An};\Lambda)^{\vee}$ provides a model for
  $\dR\pi^{\An}_{*}j^{\An}_{!}\Lambda$ in $\Der{\Lambda}$. We claim
  that the complex $\sgs(X^{\An},Z^{\An})$ defines a strongly
  dualizable object in $\Der{\Lambda}$. Indeed, using the
  distinguished triangle
\begin{equation*}
  \sgs(Z^{\An})\to \sgs(X^{\An})\to \sgs(X^{\An},Z^{\An})\to\sgs(Z^{\An})[-1],
\end{equation*}
we reduce to prove it for $\sgs(X^{\An})$. As we will see in
§\ref{sec:basic-lemma}, this complex is quasi-isomorphic to a bounded
complex of finitely generated free $\Lambda$-modules, thus it is a
strongly dualizable object.\footnote{In the notation of
  section~\ref{sec:asp.eff}, this quasi-isomorphic complex is
  $C^{\mathcal{Y}}(X,\emptyset)$ for a finite affine open cover
  $\mathcal{Y}$ of $X$.}

Therefore the canonical map from $\sgs(X^{\An},Z^{\An})$ to its double
dual is a quasi-isomorphism, and we obtain the following sequence of
isomorphisms in $\Mod{\Lambda}$, for every $n\in\Z$:
  \begin{align*}
    \h_{n}(X^{\An},Z^{\An})&=\h_{n}\sgs(X^{\An},Z^{\An})\\
    &\cong \h_{n}(\sgs(X^{\An},Z^{\An})^{\vee\vee})\\
    &\cong \h_{n}((\dR\pi^{\An}_{*}\dR
    j^{\An}_{!}\Lambda)^{\vee})&&\text{see
    appendix~\ref{sec:coh-pair}}\\\notag
  &\cong
    \h_{n}\bti((\dR\pi_{*} j_{!}\Lambda)^{\vee})&&\text{by the main results
    of~\cite{ayoub:betti}}\\\notag 
    &\cong\h_{n}\bti\dL\pi_{!}\dR
    j_{*}j^{*}\pi^{!}\Lambda&&\text{by duality}\\
    &\cong \h_{0}\bti \rep(X,Z,n)&&\text{since }\bti\text{ is triangulated.}
  \end{align*}
  This defines the isomorphism in the proposition. We have to check
  that it is compatible with the two types of edges in $\diaeff$.

  Let $f:(X,Z,n)\to (X',Z',n)$ be an edge in $\diaeff$. Compatibility
  with respect to $f$ will follow from the commutativity of the outer
  rectangle in the following diagram (namely, after applying $\h_{n}$
  and noticing that $\bti$ commutes with shifts):
  \begin{equation*}
    \xymatrix@C=5em{\sgs(X^{\An},Z^{\An})\ar[r]^{\sgs(f)}\ar@{-}[d]_{\sim}&\sgs(X^{\prime\An},Z^{\prime\An})\ar@{-}[d]^{\sim}\\
      \sgs(X^{\An},Z^{\An})^{\vee\vee}\ar[r]^{\sgs(f)^{\vee\vee}}\ar@{-}[d]_{\sim}&\sgs(X^{\prime\An},Z^{\prime\An})^{\vee\vee}\ar@{-}[d]^{\sim}\\
      (\dR\pi^{\An}_{*}j^{\An}_{!}\Lambda)^{\vee}\ar[r]^{\rep^{\An,\vee}(f)^{\vee}}\ar@{-}[d]_{\sim}&(\dR\pi^{\prime\An}_{*}j^{\prime\An}_{!}\Lambda)^{\vee}\ar@{-}[d]^{\sim}\\
      \bti(\dR\pi_{*}j_{!}\Lambda)^{\vee}\ar[r]^{\rep^{\vee}(f)^{\vee}}\ar@{-}[d]_{\sim}&\bti(\dR\pi^{\prime}_{*}j^{\prime}_{!}\Lambda)^{\vee}\ar@{-}[d]^{\sim}\\
      \bti\dL\pi_{!}\dR
      j_{*}j^{*}\pi^{!}\Lambda\ar[r]^{\rep(f)}&\bti\dL\pi'_{!}\dR j'_{*}j^{\prime*}\pi^{\prime!}\Lambda}
  \end{equation*}
  We will describe the arrows as we go along proving each square
  commutative. Starting at the bottom,
  $\rep^{\vee}(f):j^{\prime}_{!}\Lambda\to \dR f_{*}j_{!}\Lambda$ is
  defined dually to $\rep(f)$. It is then clear that the bottom square
  commutes. The same definition (using the six functors formalism,
  that is) in the analytic setting gives rise to the arrow
  $\rep^{\An,\vee}(f)$ in the third row. Again, by the main results
  of~\cite{ayoub:betti}, the third square commutes as
  well. Commutativity of the first square is clear, while
  commutativity of the second square follows from
  Lemma~\ref{lem:coh-pair-functoriality}.

  Let $(X,Z,n)$ be a vertex in $\diaeff$ and $W\subset Z$ a closed
  subvariety, giving rise to the edge $(X,Z,n)\to
  (Z,W,n-1)$. Compatibility with respect to this edge will follow from
  commutativity of the diagram
  \begin{equation*}
    \xymatrix{\h_{n}(X,Z)\ar@{-}[d]_{\sim}\ar[r]^{\partial}&\h_{n-1}(Z)\ar@{-}[d]_{\sim}\ar[r]&\h_{n-1}(Z,W)\ar@{-}[d]^{\sim}\\
      \h_{0}\bti \rep(X,Z,n)\ar[r]_-{\partial}&\h_{0}\bti
      \rep(Z,\emptyset,n-1)\ar[r]_-{\mathrm{adj}}&\h_{0}\bti \rep(Z,W,n-1)}
  \end{equation*}
  where the vertical arrows are the isomorphisms constructed above,
  and where the horizontal arrows on the right are induced by
  $(Z,\emptyset,n-1)\to (Z,W,n-1)$. In particular, the right square
  commutes by what we have shown above, and we reduce to prove
  commutativity of the left square. It can be decomposed as follows
  (before applying $\h_{n}$, using again that $\bti$ commutes with
  shifts; the horizontal arrows will be made explicit below):
  \begin{equation}\label{eq:coh-pair-boundary}
    \xymatrix@C=5em{\sgs(X^{\An},Z^{\An})\ar[r]^{\partial}\ar@{-}[d]_{\sim}&\sgs(Z^{\An})[-1]\ar@{-}[d]^{\sim}\\
      \sgs(X^{\An},Z^{\An})^{\vee\vee}\ar[r]^{\delta^{\vee}}\ar@{-}[d]_{\sim}&(\sgs(Z^{\An})[-1])^{\vee\vee}\ar@{-}[d]^{\sim}\\
      (\dR\pi^{\An}_{*}j^{\An}_{!}\Lambda)^{\vee}\ar[r]^{\delta^{\vee}}\ar@{-}[d]_{\sim}&(\dR\pi^{\An}_{*}i^{\An}_{*}\Lambda[1])^{\vee}\ar@{-}[d]^{\sim}\\
    \bti
    (\dR\pi_{*}j_{!}\Lambda)^{\vee}\ar[r]^{\delta^{\vee}}\ar@{-}[d]_{\sim}&\bti(\dR\pi_{*}i_{*}\Lambda[1])^{\vee}\ar@{-}[d]^{\sim}\\
  \bti\dL\pi_{!}\dR j_{*}j^{*}\pi^{!}\Lambda\ar[r]^{\rep(\partial)}&\bti\dL\pi_{!}i_{!}i^{!}\pi^{!}\Lambda[-1]}
  \end{equation}
  Starting at the bottom, the morphism $\delta$ arises from the
  distinguished triangle of motives over $X$:
  \begin{equation}\label{eq:coh-pair-def-delta}
    i_{*}\Lambda[1]\xrightarrow{\delta} j_{!}\Lambda\to \Lambda\to i_{*}\Lambda.
  \end{equation}
  Taking the dual we obtain the other localization
  triangle~\eqref{eq:loc-triangle-hom}. 
  Thus commutativity of the bottom square follows.

  We can consider the exact same distinguished triangle
  as~\eqref{eq:coh-pair-def-delta} in the analytic setting. This gives
  rise to the arrow $\delta^{\vee}$ in the third row
  of~\eqref{eq:coh-pair-boundary}. Thus commutativity of the third
  square in~\eqref{eq:coh-pair-boundary} follows from the fact that
  the compatibility of the Betti realization with the six functors
  formalism is also compatible with the triangulations.

  By Lemma~\ref{lem:coh-pair-triangles} in the appendix, the second
  square in~\eqref{eq:coh-pair-boundary} commutes if we take
  $\delta^{\vee}$ in the second row to be induced by the short exact
  sequence of singular cochain complexes. We leave it as an exercise
  to prove that this renders the top square
  in~\eqref{eq:coh-pair-boundary} commutative after applying
  $\h_{n}$.
\end{proof}

\subsection{Monoidality}

Our next goal is to prove that the isomorphism of the proposition
preserves the \ugm{} structures of the two representations (restricted
to $\diagoodeff$; cf.~appendix~\ref{sec:N-ugm}), in other words that
it is compatible with the K\"{u}nneth isomorphism in Betti
homology. But first we must define the \ugm{} structure on the
representation $\h_{0}\tbti \rep$.

Let $(X_{1},Z_{1})$ and $(X_{2},Z_{2})$ be two pairs of varieties, and
set $\overline{X}=X_{1}\times X_{2}$, $\overline{Z}_{1}=Z_{1}\times
X_{2}$, $\overline{Z}_{2}=X_{1}\times Z_{2}$,
$\overline{Z}=\overline{Z}_{1}\cup \overline{Z}_{2}$. There is a
canonical morphism (a motivic ``cup product'')
\begin{equation}\label{eq:cupproduct-motives}
  \dR\overline{\pi}_{*}\overline{j}_{1!}\Lambda\otimes^{\dL}\dR\overline{\pi}_{*}\overline{j}_{2!}\Lambda\to\dR\overline{\pi}_{*}(\overline{j}_{1!}\Lambda\otimes^{\dL}\overline{j}_{2!}\Lambda)\cong \dR\overline{\pi}_{*}\overline{j}_{!}\Lambda
\end{equation}
and we obtain (for $\overline{n}=n_{1}+n_{2}$)
\begin{multline*}
  \tilde{\tau}:\rep(\overline{X},\overline{Z},\overline{n})\xrightarrow{\eqref{eq:cupproduct-motives}^{\vee}}(\rep(\overline{X},\overline{Z}_{1},0)\otimes^{\dL}
  \rep(\overline{X},\overline{Z}_{2},0))[\overline{n}]\xrightarrow[\sim]{\gamma}\\ \rep(\overline{X},\overline{Z}_{1},n_{1})\otimes^{\dL}
  \rep(\overline{X},\overline{Z}_{2},n_{2})\xrightarrow{\rep(p_{1})\otimes^{\dL}
    \rep(p_{2})}\rep(X_{1},Z_{1},n_{1})\otimes^{\dL} \rep(X_{2},Z_{2},n_{2})
\end{multline*}
where $p_{i}:\overline{X}\to X_{i}$ denotes the projection onto the
$i$th factor. One word about the isomorphism $\gamma$: In the
category of complexes there are two natural choices for $\gamma$, by
following one of the two paths in the following square:
\begin{equation*}
  \xymatrix{(\bullet_{1}\otimes
    \bullet_{2})[\overline{n}]\ar[r]\ar[d]&(\bullet_{1}\otimes
    \bullet_{2}[n_{2}])[n_{1}]\ar[d]\\
    (\bullet_{1}[n_{1}]\otimes
    \bullet_{2})[n_{2}]\ar[r]&\bullet_{1}[n_{1}]\otimes \bullet_{2}[n_{2}]}
\end{equation*}
This square commutes up to the sign $(-1)^{n_{1}\cdot n_{2}}$. We
choose the $\gamma$ which is the identity in degree 0. (Which of the
two paths we choose thus depends on the sign conventions for the
tensor product and shift in the category of chain complexes.)

We can now define the \ugm structure on $\h_{0}\tbti \rep$ as the following
composition (for any $v_{1},v_{2}\in\diaeff$):
\begin{multline*}
  \tau_{(v_{1},v_{2})}:\h_{0}\tbti  \rep(v_{1}\times
  v_{2})\xrightarrow{\tilde{\tau}}\h_{0}\tbti(\rep(v_{1})\otimes^{\dL}\rep(v_{2}))\\
  \xrightarrow{\sim}  \h_{0}(\tbti \rep(v_{1})\otimes^{\dL}\tbti \rep(v_{2}))\to\h_{0}\tbti \rep(v_{1})\otimes\h_{0}\tbti \rep(v_{2}).
\end{multline*}

\begin{pro}\label{pro:NA-eff-ugm}Assume that $\Lambda$ is a principal
  ideal domain. Then:
  \begin{enumerate}
  \item The morphisms $\tau_{(v_{1},v_{2})}$ define a \ugm structure on
    the representation $\h_{0}\tbti \rep:\diagoodeff\to\fcoMod{\Has}$.
  \item The isomorphism of the previous proposition is compatible with
    the \ugm structures, \ie{} it induces an isomorphism of \ugm representations
    \begin{equation*}
      \xymatrix@C=5em{\diagoodeff\ar[r]^-{\h_{0}\tbti
          \rep}\ar[dr]_{\h_{\bullet}}&\fcoMod{\Has}\ar[d]^{\ff}\\
        &\fMod{\Lambda}}
    \end{equation*}
  \end{enumerate}
\end{pro}
\begin{proof}
  For the first part we need to check that in $\fcoMod{\Has}$, some
  morphisms are invertible and some diagrams commute. Both these
  properties can be checked after applying
  $\ff:\fcoMod{\Has}\to\fMod{\Lambda}$. Since the corresponding
  properties are true for the representation $\h_{\bullet}$, we see
  that to prove the proposition, it suffices to show that the
  isomorphism of the previous proposition takes $\ff\circ\tau$ to the
  Künneth isomorphism.

  Write $v_{1},v_{2},\overline{v}$ for the motives $\rep(X_{1},Z_{1},0),
  \rep(X_{2},Z_{2},0)$ and $\rep(\overline{X},\overline{Z},0)$,
  respectively. Consider then the following diagram:\footnote{Here, as
    in the sequel, we often write $\sgs(X,Z)$ for
    $\sgs(X^{\An},Z^{\An})$.}
  \begin{equation*}
    \xymatrix{\ar[r]_-{\gamma^{-1}\circ\tilde{\tau}}\h_{\overline{n}}\bti
      \overline{v}\ar@{-}[d]_{\sim}&\h_{\overline{n}}(\bti
      v_{1}\otimes^{\dL}\bti v_{2})\ar@{-}[d]_{\sim}\ar[r]_{\gamma}^{\sim}&\h_{n_{1}}\bti v_{1}\otimes\h_{n_{2}}\bti
      v_{2}\ar@{-}[d]^{\sim}\\
      \h_{\overline{n}}(\overline{X},\overline{Z})\ar[r]^-{\sim}_-{\text{"AW"}}&
      \h_{\overline{n}}(\sgs(X_{1},Z_{1})\otimes\sgs(X_{2},Z_{2}))\ar[r]_-{\gamma}^-{\sim}&\h_{n_{1}}(X_{1},Z_{1})\otimes\h_{n_{2}}(X_{2},Z_{2})}
  \end{equation*}
  where we have written $\overline{n}$ for the sum $n_{1}+n_{2}$. The
  right square clearly commutes. The bottom horizontal arrow on the
  left is induced by the Alexander-Whitney map (it is really a zig-zag
  on the level of complexes) and $\gamma$ in the bottom right induces
  the canonical isomorphism of the (algebraic) Künneth formula, hence
  it follows that the composition of the arrows in the bottom row is
  nothing but the (topological) Künneth isomorphism. On the other
  hand, the composition of the arrows in the top row is $\tau$. Hence
  we are reduced to prove commutativity of the left square in the
  diagram above, and it suffices to do so before applying
  $\h_{\overline{n}}$.

  We now write $\overline{v}_{i}$ for the motive
  $\rep(\overline{X},\overline{Z}_{i},0)$. Decompose $\tilde{\tau}$ according
  to its definition, and use the fact that the Alexander-Whitney map
  admits a similar decomposition in $\Der{\Lambda}$:
  \begin{equation*}
    \xymatrix@C=5em{\bti\overline{v}\ar[r]^-{\eqref{eq:cupproduct-motives}^{\vee}}\ar@{-}[dd]^{\sim}&\bti(\overline{v}_{1}\otimes^{\dL}\overline{v}_{2})\ar[r]^{\rep(p_{1})\otimes^{\dL}\rep(p_{2})}\ar@{-}[d]_{\sim}&\bti(v_{1}\otimes^{\dL}v_{2})\ar@{-}[d]_{\sim}\\
      &\bti\overline{v}_{1}\otimes^{\dL}\bti\overline{v}_{2}\ar[r]^{\rep(p_{1})\otimes^{\dL}\rep(p_{2})}\ar@{-}[d]_{\sim}&
      \bti v_{1}\otimes^{\dL}\bti v_{2}\ar@{-}[d]_{\sim}\\
      \sgs(\overline{X},\overline{Z})\ar[r]^-{\text{"AW-diag"}}&\sgs(\overline{X},\overline{Z}_{1})\otimes\sgs(\overline{X},\overline{Z}_{2})\ar[r]^{\sgs(p_{1})\otimes
        \sgs(p_{2})}&\sgs(X_{1},Z_{1})\otimes\sgs(X_{2},Z_{2})}
  \end{equation*}
  We wrote ``AW-diag'' for the Alexander-Whitney diagonal
  approximation which is a zig-zag of morphisms of complexes (see
  below). It is clear that the upper right square commutes, as does
  the lower right square by the proof of the previous proposition. For
  the square on the left, notice that~\eqref{eq:cupproduct-motives}
  equally defines a morphism in the category $\anda$. Thus we now
  denote by $\overline{v}$, $\overline{v}_{i}$ the same expressions in
  terms of the four functors in $\anda$ instead of $\da$. Then the
  proof of the proposition will be complete if we can prove
  commutative the following diagram (in which all vertical arrows are
  the canonical invertible ones):
  \begin{equation*}
    \xymatrix@C=3em{\overline{v}\ar@{-}[d]\ar[rr]^{\eqref{eq:cupproduct-motives}^{\vee}}&&\overline{v}_{1}\otimes^{\dL}\overline{v}_{2}\ar@{-}[d]\\
      (\overline{\pi}_{*}\overline{j}_{!}\Lambda)^{\vee}\ar[rr]^{\eqref{eq:cupproduct-motives}^{\vee}}\ar@{-}[d]&&      (\overline{\pi}_{*}\overline{j}_{1!}\Lambda\otimes^{\dL}\overline{\pi}_{*}\overline{j}_{2!}\Lambda)^{\vee}\ar@{-}[d]\\
      \sgs(\overline{X},\overline{Z})^{\vee\vee}\ar@{-}[d]&\sgs(\overline{X},\overline{Z}_{1}+\overline{Z}_{2})^{\vee\vee}\ar@{-}[d]\ar[l]_-{\sim}\ar[r]^-{\cupproduct^{\vee}}&(\sgs(\overline{X},\overline{Z}_{1})^{\vee}\otimes\sgs(\overline{X},\overline{Z}_{2})^{\vee})^{\vee}\ar@{-}[d]\\
      \sgs(\overline{X},\overline{Z})&\sgs(\overline{X},\overline{Z}_{1}+\overline{Z}_{2})\ar[l]_-{\sim}\ar[r]^-{\text{AW-diag}}&\sgs(\overline{X},\overline{Z}_{1})\otimes\sgs(\overline{X},\overline{Z}_{2})}
  \end{equation*}
  Here, $\sgs(\overline{X},\overline{Z}_{1}+\overline{Z}_{2})$ denotes
  the free $\Lambda$-module on simplices in $\overline{X}$ which are
  neither contained in $\overline{Z}_{1}$ nor in
  $\overline{Z}_{2}$. The first rectangle clearly commutes, the second
  does so by Lemma~\ref{lem:coh-pair-monoidal} (which may be applied
  because of Fact~\ref{lem:varieties-topology}), the bottom right
  square is well-known to commute (see
  \eg{}~\cite[VII,~8]{dold:lectures-at}), and the bottom left one
  obviously commutes as well.
\end{proof}

\subsection{Bialgebra morphism}

Using the proposition we obtain the following commutative (up to a \ugm
isomorphism) rectangle of \ugm representations:
\begin{equation}
  \label{eq:NA-eff}
  \xymatrix@C=5em{\diagoodeff\ar[r]^-{\h_{0}\tbti \rep}\ar[d]_{\tilde{\h}_{\bullet}}&\fcoMod{\Has}\ar[d]^{\ff}\\
    \fcoMod{\Hn}\ar[r]_-{\ff}\ar@{.>}[ur]&\fMod{\Lambda}}
\end{equation}
Still assuming that $\Lambda$ is a principal ideal domain we know, by
Lemma~\ref{lem:A-flat} together with Fact~\ref{thm:comod-basic}, that
$\fcoMod{\Has}$ is an abelian category. Hence
Theorem~\ref{pro:N-universal} yields a monoidal functor
$\overline{\nae}$ represented by the dotted arrow in the
diagram~\eqref{eq:NA-eff}, rendering it commutative (up to monoidal
isomorphism). It then follows from~\cite[II,
3.3.1]{saavedra-cat.tann.1972} that $\overline{\nae}$ necessarily
arises from a map of bialgebras
$\nae:\Hn\to\Has$. 

Recall (from~§§\ref{sec:N}, \ref{sec:A}) that the Hopf algebras $\Hns$
and $\Has$ are obtained from $\Hn$ and $\Ha$ by localization with
respect to elements $\sn$ and $\sa$ respectively. Using the
commutativity of~\eqref{eq:NA-eff} one easily checks that
$\nae(\sn)=\sa\in\Has$ hence $\nae$ factors through $\Hns$:
\begin{equation}\label{eq:NA}
  \xymatrix{\Hn\ar[r]^{\nae}\ar[d]_{\iota}&\Has\\
    \Hns\ar[ru]_{\na}}
\end{equation}

\section{Basic Lemma, and applications}
\label{sec:basic-lemma}
Now we would like to construct a morphism $\an$ in the other
direction. The construction relies on Nori's functor which associates
to an affine variety a complex in $\ehm$ computing its homology, and
which in turn relies on the ``Basic Lemma''. We recall them both in
this section (basically following~\cite{gartz-liealgebra}), while we
prove the existence of $\an$ in the next section.

We first recall the Basic Lemma in the form Nori formulated it
\cite[Thm.~2.1]{nori-lectures}; see
also~\cite[Lem.~4.4]{ivorra-hodge-motivic-real}. It was independently
proven by Beilinson in a more general context
\cite[Lem.~3.3]{beilinson87}.

\begin{fac}[Basic Lemma]\label{lem:basic-lemma}
  Let $X$ be an affine variety of dimension $n$, and $W\subset X$ a
  closed subvariety of dimension $\leq n-1$. Then there exists a
  closed subvariety $W\subset Z\subset X$ of dimension $\leq n-1$ such
  that $\h_{\bullet}(X^{\An},Z^{\An};\Z)$ is a free abelian group
  concentrated in degree $n$.
\end{fac}

We call a pair $(X,Z,n)$ \emph{very good}
(cf.~\cite[Def.~D.1]{huber-mueller:nori}) if either $X$ is
affine of dimension $n$, $Z$ is of dimension $\leq n-1$, $X\backslash
Z$ is smooth, and $\h_{\bullet}(X^{\An},Z^{\An};\Z)$ is a free abelian group
concentrated in degree $n$, or if $X=Z$ is affine of dimension less
than $n$. Thus the Basic Lemma implies that any pair $(X,W,n)$ with
$\dim(X)=n>\dim(W)$ can be embedded into a very good pair $(X,Z,n)$.

Nori applied this result to construct ``cellular decompositions'' of
affine varieties as follows. Let $X$ be an affine variety. A
\emph{filtration of $X$} is an increasing sequence
$F_{\bullet}=(F_{i}X)_{i\in\Z}$ of closed subvarieties of $X$ such that
\begin{itemize}
\item $\dim(F_{i}X)\leq i$ for all $i$ (in particular,
  $F_{-1}X=\emptyset$),
\item $F_{n}X=X$ for some $n\in\Z$.
\end{itemize}
The minimal $n\in\Z$ such that $F_{n}X=X$ is called the \emph{length}
of $F_{\bullet}$ (by convention, for $X=\emptyset$ this length is
defined to be $-\infty$). Clearly the filtrations of $X$ form a
directed set. A filtration $F_{\bullet}$ is called \emph{very good} if
$(F_{i}X,F_{i-1}X,i)$ is a very good pair for each $i$. The following
result says that very good filtrations form a cofinal set. It is also
used in~\cite{nori-lectures}.

\begin{cor}\label{cor:cellular-decomposition}
  Let $X$ be an affine variety, and $F_{\bullet}$ a filtration of
  $X$. Then there exists a very good filtration $G_{\bullet}\supset
  F_{\bullet}$ of $X$. In particular, every affine variety of
  dimension $n$ admits a very good filtration of length $n$.
\end{cor}
\begin{proof}
  We do induction on the length $n$ of $F_{\bullet}$. Every filtration
  of length $n=-\infty$ or $n=0$ is very good. Assume now $n>0$. Set
  $G_{i}X=X$ for all $i\geq n$. If $\dim(X)<n$, let $G_{n-1}X=X$. If
  $\dim(X)=n$ then, applying the Basic Lemma to the pair
  $(X,F_{n-1}X)$, we obtain a closed subvariety $F_{n-1}X\subset
  Z\subset X$ such that $(X,Z,n)$ is very good, and we set
  $G_{n-1}X=Z$ in this case. Now apply the induction hypothesis to the
  filtration $\emptyset\subset F_{0}X\subset\cdots\subset
  F_{n-2}X\subset G_{n-1}X$.
\end{proof}

To any filtration $F_{\bullet}$ of $X$ we associate the complex
$\h_{\bullet}(X,F_{\bullet})=\h_{\bullet}(X,F_{\bullet};\Lambda)$,
\begin{equation*}
 \h_{n}(X^{\An},F_{n-1}X^{\An})\to\h_{n-1}(F_{n-1}X^{\An},F_{n-2}X^{\An})\to\cdots\to \h_{0}(F_{0}X^{\An},\emptyset),
\end{equation*}
concentrated in the range of degrees $[0,n]$, where the differentials
are the boundary maps from the homology sequence of a triple. It
follows that $\h_{\bullet}(X,F_{\bullet})$ can be considered as an
object of $\ch(\ehm)$. For a very good filtration, this complex
computes the singular homology of $X^{\An}$ for the same reason that
cellular homology and singular homology agree. (For a more precise
statement see Fact~\ref{lem:N-sing} below.) It then follows from the
corollary that also
\begin{equation*}
  C(X):=\varinjlim_{F_{\bullet}}\h_{\bullet}(X,F_{\bullet})\in \ch(\iehm)
\end{equation*}
computes singular homology of $X^{\An}$.

Given a morphism of affine varieties $f:X\to Y$, and a filtration
$F_{\bullet}$ on $X$, we obtain a filtration
\begin{equation*}
  \overline{f(X)}\supset\overline{f(F_{n-1}X)}\supset\cdots\supset\overline{f(F_{0}X)}\supset\emptyset
\end{equation*}
of $\overline{f(X)}$. Let $m=\dim(Y)$, and define a filtration
$G_{\bullet}$ on $Y$ by
\begin{equation*}
  G_{i}Y=
    \begin{cases}
      \overline{f(F_{i}X)}&:i<m\\
      Y&:i\geq m.
    \end{cases}
\end{equation*}
This induces a morphism $\h_{\bullet}(X,F_{\bullet})\to
\h_{\bullet}(Y,G_{\bullet})$ in $\ch(\ehm)$. It follows that $C$
defines a functor $\aff\to\ch(\iehm)$ on affine varieties.

Now given filtrations $F_{\bullet}$ and $G_{\bullet}$ on affine
varieties $X$ and $Y$, respectively, we form the filtration $(F\times
G)_{\bullet}$ on $X\times Y$, setting $(F\times G)_{i}(X\times Y)$ to
be $\cup_{p+q=i}F_{p}X\times G_{q}Y$. There is a canonical morphism
$\h_{\bullet}(X,F_{\bullet})\otimes\h_{\bullet}(Y,G_{\bullet})\to\h_{\bullet}(X\times
Y,(F\times G)_{\bullet})$ which induces a morphism $C(X)\otimes
C(Y)\to C(X\times Y)$. One can check that this endows $C$ with a lax
monoidal structure.

To go further we have to make precise the relation between the
functors $C$ and $\sgs\circ\Anal$. For this,
following~\cite{gartz-liealgebra}, we consider the subcomplex $P(X)$
of $\sgs(X^{\An})$ which in degree $p$ consists of singular $p$-chains
in $X^{\An}$ whose image is contained in $Z^{\An}$ for some closed
subvariety $Z\subset X$ of dimension $\leq p$, and whose boundary lies
in $W^{\An}$ for some closed subvariety $W\subset X$ of dimension
$\leq p-1$. Such a singular chain defines a homology class in
$\h_{p}(Z^{\An},W^{\An};\Lambda)$ hence there is a canonical map
$P(X)\surj \ff C(X)$ (here, as usual, $\ff$ forgets the comodule
structure). The following result follows from the Basic Lemma and some
linear algebra.

\begin{fac}[{\cite[Lem.~4.14]{gartz-liealgebra}}]\label{lem:N-sing}
  Let $X$ be an affine variety. Both maps of chain complexes of
  $\Lambda$-modules
  \begin{equation*}
    \ff C(X)\leftarrow P(X)\to \sgs(X^{\An})
  \end{equation*}
  are quasi-isomorphisms.
\end{fac}

It is clear that $P$ defines a functor $\aff\to\ch(\Lambda)$ and that
the two maps above are natural in $X$. Moreover $P$ comes with a
canonical lax monoidal structure induced from the one on $\sgs$ (the
Eilenberg-Zilber transformation), and which is compatible with the one
on $C$ defined before.
\begin{cor}\label{cor:N-sing-monoidal}
  The maps of the previous lemma define monoidal transformations
  between lax monoidal functors
  \begin{equation*}
    \ff C\leftarrow P\to\sgs\circ\Anal
  \end{equation*}
  from $\aff$ to $\ch(\Lambda)$. If $\Lambda$ is a principal ideal
  domain then after composing with the canonical (lax monoidal)
  functor $\ch(\Lambda)\to\Der{\Lambda}$ these become monoidal
  transformations of monoidal functors.
\end{cor}
\begin{proof}
  Given affine varieties $X$ and $Y$, we have a commutative diagram:
  \begin{equation*}
    \xymatrix{\ff C(X)\otimes \ff C(Y)\ar[d]&P(X)\otimes
      P(Y)\ar[l]\ar[d]\ar[r]&\sgs(X^{\An})\otimes \sgs(Y^{\An})\ar[d]\\
      \ff C(X\times Y)&P(X\times Y)\ar[r]\ar[l]&\sgs(X^{\An}\times
      Y^{\An})}
  \end{equation*}
  $\sgs$ takes values in (complexes of) free $\Lambda$-modules, and
  $\ff C(X)$ is a direct limit of (complexes of) finitely generated
  free $\Lambda$-modules (by
  Corollary~\ref{cor:cellular-decomposition}) hence is itself a
  complex of flat $\Lambda$-modules. 
  If $\Lambda$ is a principal ideal domain then $P$ necessarily takes
  values in (complexes of) free $\Lambda$-modules as well. 
  In conclusion, under the assumptions of the corollary, all
  tensor factors in the diagram above are flat.
  
  It follows that all horizontal arrows in the diagram above are
  quasi-isomorphisms, and that in the upper row, all tensor products
  are equal to their derived versions. The corollary now follows from
  the fact that the right-most vertical arrow is a quasi-isomorphism.
\end{proof}

\begin{rem}
  There are several ways to extend the functor $C$ to all varieties,
  as explained in~\cite[p.~9]{nori-lectures}. However, for our
  purposes this will not be necessary as in the end we are interested
  only in the induced functor $\dae\to \Der{\iehm}$, and here we can
  use the equivalence between $\dae$ and $\daeaff$ of
  Lemma~\ref{lem:affine-da}.
\end{rem}

\section{Motivic realization}
\label{sec:AN}
We would now like to explain how the lax monoidal functor $C$
constructed in the previous section induces a functor on categories of
motives. The case of effective motives is treated
in~§\ref{sec:AN.construction} and as an application we deduce a
morphism of bialgebras $\ane:\Ha\to\Hn$
in~§\ref{sec:AN.an}. In~§\ref{sec:v2n} and §\ref{sec:a2n-stable} we
treat the case of effective motives with transfers and non-effective
motives, respectively.

\subsection{Construction}
\label{sec:AN.construction}
First we define the functor
\begin{equation*}
  C^{*}:\U\smaff\to\ch(\iehm)
\end{equation*}
by the coend formula for a ``left dg Kan extension''
(cf.~\cite[Fact~2.1]{choudhury-gallauer:dg-hty})
\begin{equation*}
  K\longmapsto\int^{X\in\smaff}K(X)\otimes C(X),
\end{equation*}
where the comodule structure is induced from the one on the right
tensor factor. Recall that the coend appearing in the definition is
nothing but the coequalizer of the diagram
\begin{equation*}
  \bigoplus_{X\to Y}K(Y)\otimes C(X)\rightrightarrows\bigoplus_{X}K(X)\otimes C(X),
\end{equation*}
where the two arrows are induced by the functoriality of $K$ and $C$,
respectively. We will prove below that $C^{*}$ is left Quillen for the
projective model structure on the domain (even induced from the
injective model structure on $\ch(\Lambda)$) and the injective model
structure on the codomain
(cf. Fact~\ref{thm:comod-model-injective}).

\begin{pro}\label{pro:lc} Let $\Lambda$ be a principal ideal
  domain. $\dL C^{*}$ inherits a monoidal structure, and takes
  $(\A^{1},\tau)$-local equivalences to quasi-isomorphisms. Moreover,
  it makes the following square commutative up to monoidal
  triangulated isomorphism.
  \begin{equation*}
    \xymatrix{\daeaff\ar[d]_{\dL C^{*}}\ar[r]^-{\sim}&\dae\ar[d]^{\btie}\\
      \Der{\iehm}\ar[r]_-{\Der{\ff}}&\Der{\Lambda}}
  \end{equation*}
\end{pro}

\begin{rem}
  In~\cite[p.~9]{nori-lectures}, Nori remarks that for an arbitrary
  variety $X$ and an affine open cover
  $\mathcal{U}=(U_{1},\ldots,U_{q})$ of $X$, the complex
  \begin{equation}\label{eq:lc-N}
    \mathrm{Tot}(\cdots\to \oplus_{1\leq i_{1}<\cdots< i_{p}\leq
      q}C(U_{i_{1}}\cap\cdots\cap U_{i_{p}})\to\cdots)
  \end{equation}
  ``computes the homology of $X$''. This can also be explained using
  the proposition, at least if $X$ is smooth. Namely, in that case,
  the complex
  \begin{equation*}
    \cdots\to \oplus_{1\leq i_{1}<\cdots< i_{p}\leq
      q}\Lambda(U_{i_{1}}\cap\cdots\cap U_{i_{p}})\to\cdots
  \end{equation*}
  defines a cofibrant replacement of $\Lambda(X)$ (as we will see in
  Lemma~\ref{lem:asp-lc}), and $C^{*}$ applied to it is
  just~\eqref{eq:lc-N}, as follows from~\cite[Lemma~3.21]{choudhury-gallauer:dg-hty}. 
  Hence the proposition tells
  us that the underlying complex of $\Lambda$-modules
  in~\eqref{eq:lc-N} is nothing but $\btie\Lambda(X)\cong
  \sgs(X^{\An})$ (by Proposition~\ref{pro:sgs-bti}).

  We will come back to this explicit description of $\dL C^{*}$ in
  section~\ref{sec:asp.eff} where~\eqref{eq:lc-N} is denoted by
  $C^{\mathcal{U}}(X)$.
\end{rem}

\begin{proof}[Proof of Proposition~\ref{pro:lc}]
  Just as $C$ admits a left Kan extension, so do $P$, $\sgs$ and
  $\Anal$:
  \begin{equation*}
    \xymatrix{\smaff\ar[r]\ar[d]_{C}&\U\smaff\ar[ld]^{C^{*}}\ar[ldd]^{P^{*}}\ar[dd]^{\Anal^{*}}\\
      \ch(\iehm)\ar[d]_{\ff}\\
      \ch(\Lambda)&\U\ansm\ar[l]^-{\sgs^{*}}}
  \end{equation*}
  Endow $\ch(\Lambda)$ and $\ch(\iehm)$ with the injective model
  structures (cf.~Fact~\ref{thm:comod-model-injective}), and the
  presheaf categories with the projective model structures deduced
  from the injective model structure on $\ch(\Lambda)$. We then use~\cite[Lemma~2.5]{choudhury-gallauer:dg-hty} 
  to prove that all these Kan extensions are left Quillen
  functors. For $\sgs$, $P$, and $C$ this follows from the fact that
  they take values in complexes of flat objects (see the proof of
  Corollary~\ref{cor:N-sing-monoidal}) hence the tensor product with
  these complexes is a left Quillen functor for the injective model
  structure. For $\Anal$, this is because evaluation at a smooth
  affine scheme $X$ clearly preserves (trivial) fibrations.

  Also, $C^{*}$ as well as $P^{*}$ and $\sgs^{*}\Anal^{*}\cong
  (\sgs\circ\Anal)^{*}$ inherit canonically lax monoidal structures
  (\cite[Lemma~2.2]{choudhury-gallauer:dg-hty}). 
  From Corollary~\ref{cor:N-sing-monoidal} we deduce monoidal
  transformations
  \begin{equation}
    \ff{}\circ C^{*}\xleftarrow{} P^{*}\to \sgs^{*}\circ \Anal^{*}
  \end{equation}
  of lax monoidal functors defined on $\U\smaff$ taking values in
  $\ch(\Lambda)$. They give rise to monoidal
  triangulated 
  transformations between the corresponding left derived functors
  (this uses Lemma~\ref{lem:comod-monoidal-dedekind}), and to prove
  that these transformations are invertible, it suffices to check it
  on objects of the form $\Lambda(X)$, where $X\in \smaff$ (by~\cite[Lemma~3.20]{choudhury-gallauer:dg-hty} 
  these compactly generate the derived category). These objects are
  cofibrant, and we conclude since the maps
  \begin{equation*}
    \ff{}\circ C^{*}\Lambda(X)\leftarrow P^{*}\Lambda(X)\rightarrow
    \sgs^{*}\circ \Anal^{*}\Lambda(X)
  \end{equation*}
  are identified with the quasi-isomorphisms
  \begin{equation*}
    \ff{}\circ C(X)\leftarrow P(X)\rightarrow \sgs\circ\Anal(X)
  \end{equation*}
  (see Fact~\ref{lem:N-sing}).

  We have now constructed a diagram of lax monoidal triangulated
  functors
  \begin{equation*}
    \xymatrix@C=5em{\Der{\iehm}\ar[d]_{\Der{\ff}}&\Der{\U\smaff}\ar[l]_-{\dL
        C^{*}}\ar[ld]^-{\dL P^{*}}\ar[d]^{\Anal^{*}}\\
      \Der{\Lambda}&\Der{\U\ansm}\ar[l]^-{\dL\sgs^{*}}}
  \end{equation*}
  which commutes up to monoidal triangulated isomorphism. Using the
  identification $\daeaff\xrightarrow{\sim}\dae$
  (Lemma~\ref{lem:affine-da}) the result therefore follows from
  Proposition~\ref{pro:sgs-bti}.
\end{proof}

\subsection{Bialgebra morphism}
\label{sec:AN.an}
Since $\Hn$ is a flat $\Lambda$-module, $\ch(\iehm)$ is canonically
equivalent to the category of $\Hn$-comodules in $\ch(\Lambda)$
(Fact~\ref{thm:comod-basic}), and we can consider the following
composition of monoidal functors:
\begin{equation}\label{eq:rea-def}
  \rea:\dae\simeq\daeaff\xrightarrow{\dL C^{*}}\Der{\iehm}\to\coMod{\Hn}^{\Der{\Lambda}},
\end{equation}
where the last term denotes the category of $\Hn$-comodules in
$\Der{\Lambda}$. The upshot of the discussion so far is that we obtain a
diagram
\begin{equation*}
  \xymatrix@C=5em{\dae\ar[r]^-{\rea}\ar[dr]_{\btie}&\coMod{\Hn}^{\Der{\Lambda}}\ar[d]^{\ff{}}\\
    &\Der{\Lambda}}
\end{equation*}
of monoidal functors which commutes up to monoidal isomorphism. To
invoke the universal property of $\ealga$ (Fact~\ref{pro:A-universal}) we
still need to verify the following Lemma.
\begin{lem}
  Let $K\in\Der{\Lambda}$, $\Lambda$ a principal ideal domain. Then the
  coaction of $\Hn$ on $\rea(K_{\mathrm{cst}})$ is trivial.
\end{lem}
\begin{proof}
  We may assume that $K$ is projective cofibrant consisting of free
  $\Lambda$-modules in each degree (for example
  by~\cite[Proposition~4.4]{choudhury-gallauer:dg-hty}).
  $K_{\mathrm{cst}}$ is projective cofibrant, and in each degree
  consists of a direct sum of $\Lambda(\spec(k))$, the presheaf
  represented by $\spec(k)$. Hence
  $\dL C^{*}(K_{\mathrm{cst}})=C^{*}(K_{\mathrm{cst}})$ is a complex
  which in each degree consists of a direct sum of $C(\spec(k))$ on
  which $\Hn$ coacts trivially. Hence $\Hn$ also coacts trivially on
  $\rea(K_{\mathrm{cst}})$.
\end{proof}

\begin{cor}\label{cor:AN-eff}
  Assume that $\Lambda$ is a principal ideal domain. There is a
  morphism of bialgebras $\ealga\to\Hn$ inducing $\ane:\Ha\to\Hn$ and
  rendering the following diagrams commutative up to monoidal
  isomorphism:
  \begin{align*}
  \xymatrix{\dae\ar[d]_{\tbtie}\ar[r]^-{\rea}&\coMod{\Hn}^{\Der{\Lambda}}\ar[d]^{\ff{}}\\
    \coMod{\ealga}\ar[r]_-{\ff}\ar[ru]&\Der{\Lambda}}&&
    \xymatrix{\dae\ar[d]_{\h_{0}\tbtie}\ar[r]^-{\h_{0}\rea}&\coMod{\Hn}\ar[d]^{\ff{}}\\
      \coMod{\Ha}\ar[r]_-{\ff}\ar[ru]_{\overline{\ane}}&\Mod{\Lambda}}
  \end{align*}
\end{cor}
There are two ways to obtain similar statements in the stable
setting. The easier one is to check that $\ealga\to\Hn$ passes to the
localizations $\alga\to\Hns$ and consider the composition
\begin{equation*}
  \da\xrightarrow{\tbti}\coMod{\alga}\to\coMod{\Hns}^{\Der{\Lambda}}.
\end{equation*}
This will be sufficient for our main theorem, and we will pursue it
in~§\ref{sec:asp.eff}. However, it might seem more natural and lead to
stronger results to extend the construction of $C^{*}$ to the level of
spectra and derive the resulting functor. This will be done
in~§\ref{sec:a2n-stable}.

\subsection{Transfers}
\label{sec:v2n}
The remainder of~§\ref{sec:AN} will not be strictly necessary for our
main theorem but the results obtained here are of independent
interest. In~§\ref{sec:v2n}, our goal is Proposition~\ref{pro:lctr}
where we prove that $\dL C^{*}$ extends to effective motives with
transfers.

Recall (\cite[§3.1]{nori-lectures}) that to a finite correspondence
$X\to S^{d}Y$ of degree $d$ between affine schemes, Nori associates a
morphism $C(X)\to C(Y)$, defined as the composition
\begin{equation*}
  C(X)\to C(S^{d}Y)\xleftarrow{\sim} C(Y^{d})_{\Sigma_{d}}\xrightarrow{\sum_{i=1}^{d}
  C(p_{i})} C(Y),
\end{equation*}
where $(\bullet)_{\Sigma_{d}}$ denotes the $\Sigma_{d}$-coinvariants,
and where the $p_{i}:Y^{d}\to Y$ are the canonical projections. As
proved in~\cite[Thm.~4.8.1]{harrer:phd}, this induces a functor
$C_{\mathrm{tr}}:\smaffcor\to\ch(\iehm)$ on smooth affine
correspondences, and we thus obtain a commutative triangle
  \begin{equation}\label{eq:c-add-transfer}
    \xymatrix{\smaff\ar[r]^-{C}\ar[d]&\ch(\ehm)\\
      \smaffcor\ar[ru]_-{C_{\mathrm{tr}}}}
  \end{equation}
  where the vertical arrow is the canonical inclusion. The same
  procedure as above yields a left Quillen functor
  $C_{\mathrm{tr}}^{*}:\U\smaffcor\to\ch(\iehm)$ for the projective
  model structure on the domain and the injective model structure on
  the codomain.
\begin{pro}\label{pro:lctr}
  Let $\Lambda$ be a principal ideal domain. $\dL C^{*}_{\mathrm{tr}}$
  inherits a monoidal structure, and takes $(\A^{1},\tau)$-local
  equivalences to quasi-isomorphisms. Moreover,
  (\ref{eq:c-add-transfer}) induces the following diagram, commutative
  up to monoidal triangulated isomorphism.
  \begin{equation*}
    \xymatrix{\dae\ar[d]&\daeaff\ar[l]_-{\sim}\ar[r]^-{\dL C^{*}}\ar[d]&\Der{\iehm}\\
      \dme&\dmeaff\ar[l]^-{\sim}\ar[ur]_-{\dL C^{*}_{\mathrm{tr}}}}
  \end{equation*}
\end{pro}
\begin{proof}
  It is proved in~\cite[Thm.~4.9.6]{harrer:phd} that the lax monoidal
  structure on $C$ is natural with respect to finite correspondences
  so that~(\ref{eq:c-add-transfer}) becomes a commutative triangle of
  lax monoidal functors. It follows that $C^{*}_{\mathrm{tr}}$ and
  $\dL C^{*}_{\mathrm{tr}}$ inherit lax monoidal structures,
  compatible with those of $C^{*}$ and $\dL C^{*}$. Fix a smooth
  affine scheme $X$ and consider the natural transformation (in $F$)
  \begin{equation}\label{eq:lc-mon-nat}
    C_{\mathrm{tr}}(X)\otimes^{\dL} \dL C_{\mathrm{tr}}^{*}(F)\to \dL C_{\mathrm{tr}}^{*}(\Lambda(X)\otimes F)
  \end{equation}
  of functors $\Der{\U(\smaffcor)}\to\Der{\iehm}$. Since the
  representable presheaves compactly generate the triangulated
  category $\Der{\U(\smaffcor)}$
  (see~\cite[Lemma~3.20]{choudhury-gallauer:dg-hty}),
  (\ref{eq:lc-mon-nat}) will be an isomorphism for all $F$ if it is so
  for $F=\Lambda(Y)$ a smooth affine scheme. But in this case,
  (\ref{eq:lc-mon-nat}) can be identified with
  \begin{equation*}
    C_{\mathrm{tr}}(X)\otimes C_{\mathrm{tr}}(Y)\to
    C_{\mathrm{tr}}(X\times Y),
  \end{equation*}
  hence by~(\ref{eq:c-add-transfer}) with
  \begin{equation*}
    C(X)\otimes C(Y)\to C(X\times Y),
  \end{equation*}
  which we know to be a quasi-isomorphism. Similarly, fix an object
  $F\in\Der{\U(\smaffcor)}$ and consider now the natural
  transformation (in $G$)
  \begin{equation*}
    \dL C_{\mathrm{tr}}^{*}(G)\otimes^{\dL} \dL C_{\mathrm{tr}}^{*}(F)\to \dL C_{\mathrm{tr}}^{*}(G\otimes^{\dL} F)
  \end{equation*}
  of functors $\Der{\U(\smaffcor)}\to\Der{\iehm}$. Again, it will be
  an isomorphism for all $G$ if it is so on representables
  $G=\Lambda(X)$. But this we just proved. We conclude that
  $\dL C^{*}_{\mathrm{tr}}$ is monoidal.

  We now claim that $\dL C^{*}_{\mathrm{tr}}$ takes
  $(\A^{1},\tau)$-local equivalences to quasi-isomorphisms. Notice
  that by the theory of Bousfield localizations, this is equivalent to
  the claim that the right adjoint $C_{\mathrm{tr},*}$ takes fibrant
  objects $K$ to $(\A^{1},\tau)$-fibrant presheaves of complexes. In
  other words (\cite[Thm.~5.7]{choudhury-gallauer:dg-hty}), we need to
  check that
  \begin{itemize}
  \item $C_{\mathrm{tr},*}K$ satisfies descent with respect to
    $\tau$-hypercovers; and
  \item $C_{\mathrm{tr},*}K(X)\to C_{\mathrm{tr},*}K(\A^{1}_{X})$ is a
    quasi-isomorphism for every smooth affine scheme $X$.
  \end{itemize}
  Both these properties can be checked on the site $\smaff$ (instead
  of $\smaffcor$) but restricted to this site $C_{\mathrm{tr},*}K$
  coincides with $C_{*}K$. Thus we conclude with
  Proposition~\ref{pro:lc}.

  Commutativity of the left square in the statement is obvious. The
  fact that the top horizontal arrow is an equivalence is
  Lemma~\ref{lem:affine-da}. Similarly, the bottom horizontal arrow is
  an equivalence (\cite[Cor.~5.16]{choudhury-gallauer:dg-hty}).
\end{proof}

\subsection{Stabilization}
\label{sec:a2n-stable}
In this subsection we will develop the stable motivic realizations for
motives with and without transfers in parallel. Statements containing
the symbol $(\mathrm{tr})$ thus have two obvious interpretations.

For any flat complex of comodules $K$, there is an injective stable
model structure on the category of symmetric $K$-spectra, by
Proposition~\ref{pro:comod-stable-model-injective}. Denote by $T$, as
in section~\ref{sec:betti}, a cofibrant replacement of
$\Lambda(\A^{1})/\Lambda(\G_{m})$ and set
$T_{\mathrm{tr}}=a_{\mathrm{tr}}T$. Notice that, canonically,
$C^{*}_{\mathrm{tr}}T_{\mathrm{tr}}\cong C^{*}T$.
\begin{lem}\mbox{}
  \begin{enumerate}
  \item The canonical morphism of bialgebras $\iota:\Hn\to\Hns$
    induces a functor
    \begin{equation*}
      \spt_{C^*T}\ch(\iehm)\xrightarrow{\overline{\iota}}\spt_{\overline{\iota}C^*T}\ch(\ihm)
    \end{equation*}
    which preserves stable weak equivalences.
  \item There is a canonical Quillen equivalence
    \begin{equation*}
      (\sus_{\overline{\iota}C^{*}T},\ev):\ch(\ihm)\xrightarrow{}\spt_{\overline{\iota}
        C^*T}\ch({\ihm}).
     \end{equation*}
  \end{enumerate}
\end{lem}
\begin{proof}
  The functor is obtained by applying $\overline{\iota}$ levelwise
  (cf.~the following proof for the details,
  or~\cite[Déf.~4.3.16]{ayoub07-thesis}). The first part is then
  obvious, and the second part follows
  from~\cite[Thm.~9.1]{hovey:spectra} since, as proved in the
  following section, tensoring with $C^*T[2]$ (and hence with $C^*T$)
  is a Quillen equivalence. 
\end{proof}

We will prove in the following proposition that
$C^{*}_{(\mathrm{tr})}$ induces a left Quillen functor
$C^{*}_{(\mathrm{tr}),s}$ on the level of spectra. Thus we may define
the compositions
\begin{align*}
  \reas&:\da\simeq\daaff\xrightarrow{\dL C^{*}_{s}}\hot(\spt_{C^{*}T}\ch(\iehm))\xrightarrow{\dR\ev\circ\overline{\iota}}\Der{\ihm},\\
  \reatrs&:\dm\simeq\dmaff\xrightarrow{\dL
    C^{*}_{\mathrm{tr},s}}\hot(\spt_{C^{*}T}\ch(\iehm))\xrightarrow{\dR\ev\circ\overline{\iota}}\Der{\ihm}.
\end{align*}
These are triangulated functors, and we will prove that they are in
addition monoidal, at least if $\Lambda$ is a field.
\begin{pro}\label{pro:lc-stable}\mbox{}
  \begin{enumerate}
  \item\label{pro:lc-stable.quillen} The functors $C^{*}$ and
    $C^{*}_{\mathrm{tr}}$ induce canonically lax monoidal left Quillen
    functors
    \begin{align*}
      C^{*}_{s}&:\spt_{T}\U(\smaff)/(\A^{1},\tau)\to \spt_{C^*T}\ch(\iehm),\\
      C^{*}_{\mathrm{tr},s}&:\spt_{T_{\mathrm{tr}}}\U(\smaffcor)/(\A^{1},\tau)\to \spt_{C^*T}\ch(\iehm).
  \end{align*}
\item\label{pro:lc-stable.comp} The following triangles commute up to
  triangulated isomorphisms
   \begin{align}\label{eq:AN-stable-commutation}
    \xymatrix{\da\ar[dr]_{\bti}\ar[r]^-{\reas}&\Der{\ihm}\ar[d]^{\Der{\ff}}\\
      &\Der{\Lambda}}&&    \xymatrix{\da\ar[d]_{\dL a_{\mathrm{tr}}}\ar[r]^-{\reas}&\Der{\ihm}\\
      \dm\ar[ru]_{\reatrs}}
  \end{align}
\item\label{pro:lc-stable.mon} If $\Lambda$ is a field then $\reaTRs$
  is monoidal, and the triangles in~\eqref{eq:AN-stable-commutation}
  commute up to monoidal isomorphisms.
\item The Nori realization functors restrict to functors
  \begin{align*}
    \reas:\da{\mathrm{ct}}\to\Derb{\hm},&&\reatrs:\dm{\mathrm{ct}}\to\Derb{\hm}
  \end{align*}
  on the categories of constructible motives.
  \end{enumerate}
\end{pro}
Recall that the category of constructible motives is the thick
subcategory generated by smooth schemes.

\begin{proof}
  We will prove the first part for $C^{*}$ but the case with transfers
  is literally the same. $C^{*}$ together with the natural
  transformation $\theta: C^{*}T\otimes C^{*}(\bullet)\to
  C^{*}(T\otimes\bullet)$ induces a functor
  \begin{equation*}
    C^{*}_{s}:\spt_{T}\U\smaff\to \spt_{C^*T}\ch(\iehm)
  \end{equation*}
  (cf.~\cite[Déf.~4.3.16]{ayoub07-thesis}). Explicitly, it takes a
  symmetric $T$-spectrum $\mathbf{E}$ to the symmetric $C^*T$-spectrum
  which in level $n$ is given by $C^{*}(\mathbf{E}_{n})$ and whose
  bonding maps are given by
  \begin{equation*}
    C^*T\otimes C^{*}(\mathbf{E}_{n})
    \xrightarrow{\theta}C^{*}(T\otimes \mathbf{E}_{n})\to C^{*}(\mathbf{E}_{n+1}),
  \end{equation*}
  the second arrow being induced by the bonding map of
  $\mathbf{E}$. The lax monoidal structure on $C^{*}$ induces
  canonically a lax monoidal structure on $C^{*}_{s}$.

  It is clear that $C^{*}_{s}$ is cocontinuous hence admits a right
  adjoint, by the adjoint functor theorem for locally presentable
  categories. Let $f$ be a projective cofibration in
  $\spt_{T}\U\smaff$. Then $f$ is in particular levelwise a
  cofibration (\cite[Cor.~4.3.23]{ayoub07-thesis}) and by the
  discussion in the previous section, $C^{*}$ takes these to
  monomorphisms. Thus $C^{*}_{s}(f)$ is a monomorphism. The same
  argument shows that $C^{*}_{s}$ takes projective cofibrations which
  are levelwise $(\A^{1},\tau)$-local equivalences to monomorphisms
  which are levelwise quasi-isomorphisms. In other words, $C^{*}_{s}$
  is a left Quillen functor for the \emph{unstable} model
  structures. To prove the first part of the proposition, it remains
  to prove that $C^{*}_{s}$ takes the morphism
  \begin{equation*}
    \zeta_{n}^{D}:\sus{n+1}_{T}(T\otimes D)\to \sus{n}_{T}D
  \end{equation*}
  to a stable equivalence for every cofibrant object $D$ and every
  $n\geq 0$ (cf.~\cite[Def.~8.7]{hovey:spectra}). But in the unstable
  homotopy category we can factor the image of $\zeta_{n}^{D}$ as
  follows:
  \begin{align*}
    C^{*}_{s}\sus{n+1}_{T}(T\otimes
    D)&\xleftarrow{}\sus{n+1}_{C^{*}T}C^{*}(T\otimes D)\\
    &\xleftarrow{}\sus{n+1}_{C^{*}T}(C^{*}T\otimes C^{*}D)\\
    &\xrightarrow{}\sus{n}_{C^*T}C^{*}D\\
    &\xrightarrow{}C^{*}_{s}\sus{n}_{T}D.
  \end{align*}
  The first, second and fourth arrows are all levelwise
  quasi-isomorphisms because $\dL C^{*}$ is monoidal on the level of
  derived categories. Moreover, the third arrow is a stable
  equivalence by definition.

  We now come to the second part of the proposition. Commutativity of
  the triangle on the right follows from (the proof of)
  Proposition~\ref{pro:lctr}. 
  For the triangle on the left, recall that
  $\sgs^{*}:\U(\ansm)/(\Done,\mathrm{usu})\to \ch(\Lambda)$ is a lax
  monoidal left Quillen functor. As for $C^{*}$ above this implies
  that there is an induced lax monoidal left Quillen functor
  $\sgs^{*}_{s}$ on the level of spectra (for the projective,
  respectively injective stable model structures). The Betti
  realization can then also be described as the following composition:
  \begin{equation*}
    \da\xrightarrow{\Anal^{*}}\anda\xrightarrow[\sim]{\dL\sgs^{*}_{s}}\hot(\spt_{\sgs^{*}\Anal^{*}T}\ch(\Lambda))\xrightarrow[\sim]{\dR\ev}\Der{\Lambda}.
  \end{equation*}
  Analogously, $\Der{\ff}\reas$ can be described as the composition
  \begin{equation*}
    \daaff\xrightarrow{\dL
      C^{*}_{s}}\hot(\spt_{C^{*}T}\ch(\coMod{\Hn}))\xrightarrow{\Der{\ff}}\hot(\spt_{\ff
    C^{*}T}\ch(\Lambda))\xrightarrow[\sim]{\dR\ev}\Der{\Lambda}.
  \end{equation*}
  One is then essentially reduced to compare $\Der{\ff}\dL C^{*}_{s}$
  and $\dL\sgs^{*}_{s}\Anal^{*}$ which is done, as in the
  effective case, by means of the intermediate functor $P$.

  We come to the third part, and assume now that $\Lambda$ is a
  field. Using Lemma~\ref{lem:comod-monoidal-field} together
  with~\cite[Thm.~8.11]{hovey:spectra} we see that the categories
  occurring in the definition of $\reaTRs$ all carry induced monoidal
  structures. By the previous lemma, $\dR\ev\circ\overline{\iota}$ is
  lax monoidal, as is $\dL C^{*}_{(\mathrm{tr}),s}$ by the first part
  of the
  proposition. 
  It follows that $\reaTRs$ is a lax monoidal functor, and the
  comparisons in part~\ref{pro:lc-stable.comp} are compatible with
  these lax monoidal structures.

  Monoidality of $\reas$ now follows from monoidality of $\bti$ and
  the fact that the derived forgetful functor is
  conservative. Monoidality of $\reatrs$ in the étale case follows
  from this since $\dL a_{\mathrm{tr}}$ is an equivalence of
  categories (cf.~\cite[Cor.~B.14]{ayoub:galois1}). Finally, the
  Nisnevich realization factors through the étale realization via a
  monoidal functor. 

  The last part of the proposition holds because $\reaTRs$ takes a
  smooth affine scheme into $\Derb{\hm}$. (For this we use that
  $\Derb{\hm}$ is a full subcategory of $\Derb{\ihm}$;
  see~\cite[Pro.~8.6.11 and
  Thm.~15.3.1.(i)]{kashiwara-schapira:cat-sheaves}.)
\end{proof}

\begin{rem}
  During the preparation of the present article, Ivorra
  in~\cite{ivorra-hodge-motivic-real} independently defined such a
  motivic realization for étale motives without transfers. While his
  construction is more general in that it applies also to a relative
  case (involving his generalization of Nori motives to ``perverse
  Nori motives'' over a base), he does not consider monoidality of the
  functor nor its behaviour with respect to transfers.
\end{rem}

Denote by $\mathbf{MHS}_{\Q}^{\mathrm{pol}}$ the category of
polarizable mixed $\Q$-Hodge structures, and by
$\ind\mathbf{MHS}_{\Q}^{\mathrm{pol}}$ its Ind-category. There is a
monoidal exact mixed Hodge realization for Nori motives whose
composition with the forgetful functor yields the forgetful functor on
Nori motives. Composing its derived counterpart with $\reaTRs$ from
the previous proposition yields the following immediate corollary.

\begin{cor}\label{cor:mixed-hodge}
  There are \emph{mixed Hodge realization} functors
    \begin{align*}
      \reah:\da{\Q}\longrightarrow \Der{\ind\mathbf{MHS}_{\Q}^{\mathrm{pol}}},&&\reahtr:\dm{\Q}\longrightarrow\Der{\ind\mathbf{MHS}_{\Q}^{\mathrm{pol}}}
    \end{align*}
    satisfying the following properties:
  \begin{enumerate}
  \item They are triangulated monoidal.
  \item They make the following triangles commute up to monoidal
    triangulated isomorphisms.
    \begin{align*}
      \xymatrix{\da{\Q}\ar[r]^-{\reah}\ar[rd]_{\bti}&\Der{\ind\mathbf{MHS}_{\Q}^{\mathrm{pol}}}\ar[d]^{\Der{\ff}}\\
        &\Der{\Q}}&&\xymatrix{\da{\Q}\ar[r]^-{\reah}\ar[d]_{\dL a_{\mathrm{tr}}}&\Der{\ind\mathbf{MHS}_{\Q}^{\mathrm{pol}}}\\
        \dm{\Q}\ar[ru]_{\reahtr}}
    \end{align*}
  \item They restrict to triangulated monoidal functors
    \begin{align*}
      \reah:\da{\Q,\mathrm{ct}}\longrightarrow\Derb{\mathbf{MHS}^{\mathrm{pol}}_{\Q}},&&\reahtr:\dm{\Q,\mathrm{ct}}\longrightarrow\Derb{\mathbf{MHS}^{\mathrm{pol}}_{\Q}}
    \end{align*}
    on the categories of constructible motives.
  \end{enumerate}
\end{cor}

\section{Almost smooth pairs}
\label{sec:asp.gen}

In the sequel we will want to manipulate the Morel-Voevodsky motives
of pairs of varieties $(X,Z)$, and describe their images under certain
functors explicitly. This is easy if both $X$ and $Z$ are smooth but
turns out to be rather difficult in general. What we need is a class
of pairs which on the one hand are close enough to smooth ones so that
explicit computations are feasible, and on the other hand flexible
enough so that we are able to reduce our arguments from general pairs
to this smaller class. This is provided by the class of \emph{almost
  smooth} pairs, \ie{} pairs of varieties $(X,Z)$ where $X$ is smooth
and $Z$ a simple normal crossings divisor. By resolution of
singularities and excision, every good pair receives a morphism from
an almost smooth one which induces isomorphisms in Betti homology. In
this section, we will give rather explicit motivic models for almost
smooth pairs, both on the effective and the stable level, and compute
their images under various functors. One immediate consequence of our
discussion here is that the morphism of bialgebras $\ane$ passes to the
stable level.

\subsection{Effective level}
\label{sec:asp.eff}

$(X,Z)$ will now be our running notation for an almost smooth pair. We
always denote the irreducible components of $Z$ by
$Z_{1},\ldots,Z_{p}$ and endow them with the reduced structure. The
(smooth) intersection of $Z_{i}$ and $Z_{j}$ is denoted by $Z_{ij}$,
and similarly for intersections of more than two components. The
presheaf $\Lambda(X,Z)$ is defined to be the cokernel of the morphism
$\oplus_{i=1}^{p}\Lambda(Z_{i})\to\Lambda(X)$.

In addition, let $\mathcal{Y}=(Y_{1},\ldots,Y_{q})$ be an open affine
cover of $X$. For any functor $F:\smaff\to \ch(\mathcal{C})$
into the category of complexes on an abelian category $\mathcal{C}$,
we define $F^{\mathcal{Y}}(X,Z_{\bullet})\in\ch(\mathcal{C})$ to be
the (sum) total complex of the tricomplex
whose $(i,j,k)$-th term is
\begin{equation*}
  \bigoplus_{a_{0}<\cdots<a_{i},\
  b_{1}<\cdots <b_{j}}F_{k}\left( Y_{a_{0}\cdots a_{i}}\cap
Z_{b_{1}\cdots b_{j}}\right),
\end{equation*}
where by convention the empty intersection of the $Z_{i}$'s is
$X$. This can also be understood as the mapping cone of the morphism
\begin{equation*}
  F^{\mathcal{Y}\cap Z_{\bullet}}(Z_{\bullet})\to F^{\mathcal{Y}}(X)
\end{equation*}
with an obvious interpretation of the first term. If $F$ is defined on
all smooth schemes, we set $F(X,Z_{\bullet})$ to be
$F^{(X)}(X,Z_{\bullet})$, and if $F$ is defined on all affine varieties,
we can similarly define $F^{\mathcal{Y}}(X,Z)$.

For example, we can consider the presheaf of complexes
$\Lambda^{\mathcal{Y}}(X,Z_{\bullet})$ with the canonical map to
$\Lambda(X,Z)$. This defines a cofibrant replacement as we now prove.
\begin{lem}\label{lem:asp-lc}
  The canonical morphism
  $\Lambda^{\mathcal{Y}}(X,Z_{\bullet})\to\Lambda(X,Z)$ is a cofibrant
  replacement for the $\tau$-local model structure, and the complexes
  \begin{equation*}
    C^{*}\Lambda^{\mathcal{Y}}(X,Z_{\bullet})\xrightarrow{\sim}C^{\mathcal{Y}}(X,Z_{\bullet})\xrightarrow{\sim}C^{\mathcal{Y}}(X,Z)
  \end{equation*}
  all provide models for $\dL C^{*}\Lambda(X,Z)$.
\end{lem}
\begin{proof}
  To prove the first statement consider the following morphism of
  distinguished triangles in the derived category of presheaves on
  smooth schemes:
  \begin{equation*}
   \xymatrix{
     \Lambda^{\mathcal{Y}}(Z_{\bullet})\ar[r]\ar[d]&\Lambda^{\mathcal{Y}}(X)\ar[d]\ar[r]&\Lambda^{\mathcal{Y}}(X,Z_{\bullet})\ar[r]\ar[d]&\Lambda^{\mathcal{Y}}(Z_{\bullet})[-1]\ar[d]\\
     \Lambda(Z_{\bullet})\ar[r]&\Lambda(X)\ar[r]&\Lambda(X,Z)\ar[r]&\Lambda(Z_{\bullet})[-1]}
  \end{equation*}
  (It should be clear what the first term denotes although we haven't
  formally defined it above. That the second row arises from a short
  exact sequence of complexes of presheaves (and hence is indeed a
  distinguished triangle) is~\cite[2.1.4]{voevodsky-homology};
  see~\cite[Lem.~1.4]{scholbach-htopology} for a proof.) The second
  vertical arrow is a $\tau$-local equivalence as is the left vertical
  arrow by induction on the number of irreducible components of
  $Z$. 
  It follows that the third vertical arrow is a $\tau$-local
  equivalence as well. Since $\Lambda^{\mathcal{Y}}(X,Z_{\bullet})$ is
  a bounded below complex of representables, it is projective
  cofibrant (\cite[Fact~3.10]{choudhury-gallauer:dg-hty}).

  We now come to the second statement of the lemma. It is clear that
  the first arrow is invertible. For the second arrow consider the
  following diagram:
  \begin{equation*}
    \xymatrix{\ff C^{\mathcal{Y}}(X,Z_{\bullet})\ar[d]&P^{\mathcal{Y}}(X,Z_{\bullet})\ar[l]\ar[r]\ar[d]&(\sgs)^{\mathcal{Y}}(X,Z_{\bullet})\ar[d]\\
      \ff C^{\mathcal{Y}}(X,Z)&P^{\mathcal{Y}}(X,Z)\ar[l]\ar[r]&(\sgs)^{\mathcal{Y}}(X,Z)}
  \end{equation*}
  By the discussion in section~\ref{sec:basic-lemma}, we know that the
  horizontal arrows are all quasi-isomorphisms. Since the right-most
  vertical arrow is a quasi-isomorphism so is the left-most.
\end{proof}

Recall that we equipped the Betti realization of any effective
Morel-Voevodsky motive $M$ with a coaction of $\Hn$, and we denoted
the resulting comodule by $\rea(M)$ (see~(\ref{eq:rea-def})). Of
course, for this to be a sensible construction, the comodule structure
should better be compatible with the canonical one on
$\Lambda$-modules of the form $\h_{n}(X,Z)$. The following lemma
states that this is the case for $(X,Z)$ almost smooth.

Define the following zig-zag of morphisms of complexes of
$\Lambda$-modules:
\begin{equation}\label{eq:comodule-iso}
  \ff C^{\mathcal{Y}}(X,Z)\xleftarrow{}P^{\mathcal{Y}}(X,Z)\xrightarrow{}(\sgs)^{\mathcal{Y}}(X,Z)\to\sgs(X,Z).
\end{equation}

\begin{lem}\label{lem:comodule-iso}
  Assume that $(X,Z)$ is an almost smooth pair, and that $\Lambda$ is
  a principal ideal domain. Then~\eqref{eq:comodule-iso} induces an
  isomorphism of $\Hn$-comodules
  \begin{equation*}
    \h_{n}\rea\Lambda(X,Z)\xrightarrow{\sim}\h_{n}(X,Z)
  \end{equation*}
  for all $n\in\Z$.
\end{lem}
\begin{proof}
  By Jouanolou's trick there exists a smooth affine variety $X'$ and a
  Zariski locally trivial morphism $p:X'\to X$ whose fibers are
  isomorphic to affine
  space. 
  Setting $Z'_{i}=Z_{i}\times_{X}X'$ we obtain an almost smooth pair
  $(X',Z')$ with $X'$ affine, and a morphism $p:(X',Z')\to (X,Z)$
  which induces an isomorphism in singular homology.

  Let $\mathcal{Y}'$ be the pullback of the affine cover to $X'$ and
  consider the following commutative diagram, where all the arrows are
  the canonical ones:
  \begin{equation*}
    \xymatrix{\ff C^{\mathcal{Y}}(X,Z)&P^{\mathcal{Y}}(X,Z)\ar[l]\ar[r]&(\sgs)^{\mathcal{Y}}(X,Z)\\
    \ff C^{\mathcal{Y}'}(X',Z')\ar[d]\ar[u]&\ar[d]P^{\mathcal{Y}'}(X',Z')\ar[u]\ar[l]\ar[r]&\ar[u](\sgs)^{\mathcal{Y}'}(X',Z')\ar[d]\\
    \ff C(X',Z')&P(X',Z')\ar[l]\ar[r]\ar[d]&\sgs(X',Z')\ar[d]\\
  &P(X')/P(Z')\ar[r]&\sgs(X')/\sgs(Z')}
  \end{equation*}
  By the discussion in section~\ref{sec:basic-lemma}, we know that the
  top horizontal arrows are both quasi-isomorphisms. All vertical
  arrows are quasi-isomorphisms. We thus reduce to prove that the
  zig-zag of morphisms $\ff C(X',Z')\leftarrow P(X',Z')\rightarrow
  \sgs(X')/\sgs(Z')$ induces an $\Hn$-comodule (iso)morphism in the
  $n$-th homology. Writing $(X,Z)$ for $(X',Z')$, this is expressed by
  commutativity of the following diagram, where the vertical arrows
  are the coaction of $\Hn$ on the objects in question:
  \begin{equation*}
    \xymatrix{\h_{n}\ff C(X,Z)\ar[d]_{\coa}&\ar[l]_-{\sim}\h_{n}P(X,Z)\ar[r]^-{\sim}&\h_{n}(X,Z)\ar[d]^{\coa}\\
      \h_{n}\ff C(X,Z)\otimes\Hn&\ar[l]^-{\sim}\h_{n}P(X,Z)\otimes\Hn\ar[r]_-{\sim}&\h_{n}(X,Z)\otimes\Hn}
  \end{equation*}
  Start with any $[(f,g)]\in\h_{n}P(X,Z)$. Thus there exist
  $X_{n}\subset X$, $Z_{n-1}\subset Z$ closed subvarieties of
  dimension at most $n$ and $n-1$, respectively, such that
  $f\in\sgs_{n}(X_{n})$, $g=\pm\partial f\in\sgs_{n-1}(Z_{n-1})$
  (depending on the sign conventions for the mapping
  cone). 
  It is then clear from the definition of the natural transformations
  $\ff C\xleftarrow{}P\xrightarrow{}\sgs$ that we reduce to prove
  commutativity of the following diagram
  \begin{equation*}
    \xymatrix{\h_{n}\ff C(X,Z)\ar[d]_{\coa}&\ar[l]\h_{n}(X_{n},Z_{n-1})\ar[d]_{\coa}\ar[r]&\h_{n}(X,Z)\ar[d]^{\coa}\\
      \h_{n}\ff C(X,Z)\otimes\Hn&\ar[l]\h_{n}(X_{n},Z_{n-1})\otimes\Hn\ar[r]&\h_{n}(X,Z)\otimes\Hn}
   \end{equation*}
   which is obvious.
\end{proof}

This lemma will be important later on as well but one immediate
application is that it allows us to extend the morphism of bialgebras
$\ane:\Ha\to\Hn$ constructed in section~\ref{sec:AN} to a morphism
$\an:\Has\to\Hns$. Indeed, we see that there is the following
isomorphism of $\Hn$-comodules:
\begin{align*}
  \h_{0}\overline{\an}\tbtie(T[2])&\xrightarrow{\sim}\h_{0}\rea(T[2])
  &&\text{by~Cor.~\ref{cor:AN-eff}}\\
  &\xrightarrow{\sim}\h_{0}\rea\Lambda(\G_{m},\{1\})[1]\\
  &\xrightarrow{\sim}\h_{1}(\G_{m},\{1\}),
\end{align*}
the last isomorphism by the previous lemma. One deduces easily that
$\ane(\sa)=\sn\in\Hn$ and hence the morphism $\ealga\to\Hn$ from
Corollary~\ref{cor:AN-eff} passes to the localization and induces the
following commutative squares:
\begin{align*}
  \xymatrix{\ealga\ar[r]\ar[d]_{\iota}&\Hn\ar[d]^{\iota}\\
    \alga\ar[r]&\Hns}&&  \xymatrix{\Ha\ar[r]^-{\ane}\ar[d]_{\iota}&\Hn\ar[d]^{\iota}\\
    \Has\ar[r]_-{\an}&\Hns}
\end{align*}

\begin{rem}\label{rem:AN}
  In particular, we can now define a stable version of the functor
  $\rea$ constructed in the effective case in section~\ref{sec:AN}
  (still assuming that $\Lambda$ is a principal ideal domain). Indeed,
  we set it to be the composition
  \begin{equation*}
    \rea:\da\xrightarrow{\tbti}\coMod{\alga}\xrightarrow{}\coMod{\Hns}^{\Der{\Lambda}}.
  \end{equation*}
  As the composition of two monoidal functors, $\rea$ is again
  monoidal. It follows also that the diagrams analogous to the ones in
  Corollary~\ref{cor:AN-eff} commute
  \begin{align*}
      \xymatrix{\da\ar[d]_{\tbti}\ar[r]^-{\rea}&\coMod{\Hns}^{\Der{\Lambda}}\ar[d]^{\ff{}}\\
    \coMod{\alga}\ar[r]_-{\ff}\ar[ru]&\Der{\Lambda}}&&
    \xymatrix{\da\ar[d]_{\h_{0}\tbti}\ar[r]^-{\h_{0}\rea}&\coMod{\Hns}\ar[d]^{\ff{}}\\
      \coMod{\Has}\ar[r]_-{\ff}\ar[ru]_{\overline{\an}}&\Mod{\Lambda}}
  \end{align*}
  as does the following square:
  \begin{equation*}
    \xymatrix@C=5em{\dae\ar[d]_{\dL\sus_{T}}\ar[r]^-{\rea}&\coMod{\Hn}^{\Der{\Lambda}}\ar[d]^{\overline{\iota}}\\
      \da\ar[r]_-{\rea}&\coMod{\Hns}^{\Der{\Lambda}}}
  \end{equation*}
\end{rem}

\subsection{Stable level}
\label{sec:asp.stb}
We continue our study of almost smooth pairs but now we work in the
context of non-effective motives. For such a pair, we will provide a
rather explicit model for both its homological as well as
cohomological Morel-Voevodsky motive in Theorem~\ref{thm:asp-motives},
and then similarly for its analytification in
Theorem~\ref{thm:asp-motives-analytic}. Subsequently we prove that the
Betti realization is in some sense compatible with these models
(Lemmas~\ref{lem:asp-motives-relation}
and~\ref{lem:compatibility-bti-as-models}).

Let $(X,Z)$ be an almost smooth pair. The inclusion of the complement
$U=X\backslash Z\to X$ is denoted by $j$. Recall
(from~\cite[§2.2.4]{ayoub:galois1}) the following constructions. Given
a presheaf $K$ of complexes on smooth schemes, one defines $K(X,Z)$ to
be the kernel of the map $K(X)\to\prod_{i=1}^{p}K(Z_{i})$.
The endofunctor $\uhom((X,Z),\bullet)$ is defined as the right adjoint
to tensoring with $\Lambda(X,Z)$. Explicitly,
\begin{equation*}
  \uhom((X,Z),K)(Y)=K(Y\times X,Y\times Z)
\end{equation*}
for any presheaf of complexes $K$ and for any smooth scheme
$Y$. $\uhom((X,Z),\bullet)$ canonically extends to an endofunctor on
symmetric $T$-spectra of presheaves of complexes.

In general, we denote the internal hom in symmetric $T$-spectra by
$\uhm$. We note that for a complex of presheaves $K$ and a symmetric
$T$-spectrum $\mathbf{E}$, the object $\uhm(\sus_{T} K,\mathbf{E})$
admits the following simple description. In level $n$, it is given by
$\uhm(K,\mathbf{E}_{n})$, the action of $\Sigma_{n}$ is on
$\mathbf{E}_{n}$, and the bonding maps are given by the composition
\begin{equation*}
  T\otimes\uhm(K,\mathbf{E}_{n})\to \uhm(K,T\otimes\mathbf{E}_{n})\to\uhm(K,\mathbf{E}_{n+1}),
\end{equation*}
where the first arises from the adjunction $(\otimes,\uhm)$, and the
second uses the bonding maps from $\mathbf{E}$. To emphasize
this description we write simply $\uhm(K,\mathbf{E})$ for this
symmetric $T$-spectrum.

Using the notation from~§\ref{sec:asp.eff}, $\Lambda(X,Z_{\bullet})$
denotes the augmented complex
\begin{equation*}
  \cdots\to \bigoplus_{i_{1}<\cdots < i_{l}}\Lambda(Z_{i_{1}\cdots
    i_{l}})\to\cdots \to \bigoplus_{i}\Lambda(Z_{i})\to \Lambda(X),
\end{equation*}
the last term being in homological degree 0.

For the next result, recall that on presheaves of complexes on smooth
schemes there is also an \emph{injective} $(\A^{1},\tau)$-local model
structure, obtained by $(\A^{1},\tau)$-localization from the
``injective model structure'' (\cite[Déf.~4.5.12]{ayoub07-thesis}).
The cofibrations and weak equivalences of the latter are defined
objectwise. One deduces then the existence of an ``injective stable
$(\A^{1},\tau)$-local model structure'' on symmetric $T$-spectra as
described in~§\ref{sec:betti}
(cf.~\cite[Déf.~4.5.21]{ayoub07-thesis}).
\begin{thm}\label{thm:asp-motives}
 Let $(X,Z)$ be almost smooth.
  \begin{enumerate}
  \item Let $\mathbf{E}$ be a projective stable
    $(\A^{1},\tau)$-fibrant symmetric $T$-spectrum of
    presheaves of complexes on $\sm$. Then
    $\uhm(\Lambda(X,Z_{\bullet}),\mathbf{E})$ provides a model for
    $\dR\pi_{*}j_{!}\mathbf{E}|_{U}$ in $\da$. Moreover, this
    identification is functorial in $\mathbf{E}$.
  \item If $\mathbf{E}$ is injective stable $(\A^{1},\tau)$-fibrant,
    then one can replace $\uhm(\Lambda(X,Z_{\bullet}),\mathbf{E})$ by
    $\uhom((X,Z),\mathbf{E})$ in the statement above.
  \item For $\mathbf{E}$ an injective stable $(\A^{1},\tau)$-fibrant
    replacement of the unit spectrum, $\uhom((X,Z),\mathbf{E})$
    provides a model for $\rep(X,Z,0)^{\vee}$ in $\da$.
  \item In $\da$, $\dL\sus_{T}\Lambda(X,Z)\cong \rep(X,Z,0)$
    canonically.
  \end{enumerate}
\end{thm}
\begin{proof}
  \begin{enumerate}
  \item Let $K_{\bullet}(\mathbf{E})$ be the object
    \begin{equation*}
      \mathbf{E}|_{X}\to\oplus_{i}i_{i*}\mathbf{E}|_{Z_{i}}\to\oplus_{i<j}i_{ij*}\mathbf{E}|_{Z_{ij}}\to\cdots
    \end{equation*}
    with differentials given by the canonical ``restriction'' morphism
    (the unit of an adjunction) with a suitable sign:
    \begin{equation*}
      i_{j_{1}\cdots \hat{j_{m}}\cdots j_{l}*}\mathbf{E}|_{Z_{j_{1}\cdots \hat{j_{m}}\cdots j_{l}}}\xrightarrow{(-1)^{m}} i_{j_{1}\cdots j_{l}*}\mathbf{E}|_{Z_{j_{1}\cdots j_{l}}}.
    \end{equation*}
    There is a canonical morphism
    \begin{equation}\label{eq:asp-resolution}
      j_{!}\mathbf{E}|_{U}\to \mathrm{Tot}(K_{\bullet}(\mathbf{E}))=:K(\mathbf{E})
    \end{equation}
    (the totalization functor is applied levelwise; up to canonical
    isomorphism it doesn't matter whether $\tots$ or $\totp$ is used),
    and we claim that this is a projective stable
    $(\A^{1},\tau)$-fibrant resolution. In other words, we claim that
    \begin{enumerate}[label=(\alph*)]
    \item \label{cl:asp-motives-equivalence}\eqref{eq:asp-resolution}
      is a stable $(\A^{1},\tau)$-local equivalence,
    \item \label{cl:asp-motives-levelwise}$K(\mathbf{E})$ is levelwise
      projective $(\A^{1},\tau)$-fibrant, and
    \item \label{cl:asp-motives-omega}$K(\mathbf{E})$ is an
      $\Omega$-spectrum.
    \end{enumerate}
    \begin{proof}[Proof of~\ref{cl:asp-motives-equivalence}]
      One can use conservativity of the couple $(j^{*},\dL i^{*})$,
      $i:Z\inj X$ being the closed immersion. It is obvious that
      $j^{*}$ applied to~\eqref{eq:asp-resolution} is an equivalence
      while in the case of $\dL i^{*}$ it is an easy induction
      argument on the number of irreducible components of $Z$.
    \end{proof}

    \begin{proof}[Proof of~\ref{cl:asp-motives-levelwise}]
      Fix a level $n$ and set $E=\mathbf{E}_{n}$. We know that for
      each $l$, $K_{l}(E)$ is $\tau$-fibrant hence so is $K(E)$
      by~\cite[Lemma~5.20]{choudhury-gallauer:dg-hty}.  The same
      argument shows that $K(E)$ is $\A^{1}$-local.
    \end{proof}
    \begin{proof}[Proof of~\ref{cl:asp-motives-omega}]
      Since $K_{l}(\mathbf{E}_{n})\to\uhom(T,K_{l}(\mathbf{E}_{n+1}))$
      is an $(\A^{1},\tau)$-local equivalence for each $l$ so is the
      totalization
      \begin{equation*}
        \mathrm{Tot}(K_{\bullet}(\mathbf{E}_{n}))\to\mathrm{Tot}(\uhom(T,K_{\bullet}(\mathbf{E}_{n+1})))=\uhom(T,\mathrm{Tot}(K_{\bullet}(\mathbf{E}_{n+1}))).
      \end{equation*}
    \end{proof}

    It follows that in $\da$,
    \begin{align}\label{eq:asp-motives-model}
      \begin{split}
        \uhm(\Lambda(X,Z_{\bullet}),\mathbf{E})&=\mathrm{Tot}(\uhm(\Lambda(X),\mathbf{E})\to\oplus_{i}\uhm(\Lambda(Z_{i}),\mathbf{E})\to\cdots)\\
        &=\mathrm{Tot}(\pi_{*}\pi^{*}\mathbf{E}\to\oplus_{i}\pi_{Z_{i}*}\pi_{Z_{i}}^{*}\mathbf{E}\to\cdots)\\
        &=\pi_{*}\mathrm{Tot}(K_{\bullet}(\mathbf{E}))\\
        &\xleftarrow{\sim}\dR\pi_{*}j_{!}\mathbf{E}|_{U}.
      \end{split}
    \end{align}
    Finally, functoriality in $\mathbf{E}$ is clear.
  \item We find an $(\A^{1},\tau)$-local equivalence
    \begin{equation}\label{eq:asp-cof}
      \Lambda(X,Z_{\bullet})\to \Lambda(X,Z)
    \end{equation}
    between injective cofibrant objects hence
    $\uhm(\bullet, \mathbf{E})$ will transform~\eqref{eq:asp-cof} into
    an (un)stable $(\A^{1},\tau)$-local equivalence. Since
    $\mathbf{E}$ is in particular projective stable
    $(\A^{1},\tau)$-fibrant it follows from the first part of the
    theorem just proved that
    \begin{align*}
      \dR\pi_{*}j_{!}\mathbf{E}|_{U}&\xrightarrow{\sim}\uhm(\Lambda(X,Z_{\bullet}),\mathbf{E})\\
      &\xrightarrow[\eqref{eq:asp-cof}]{\sim}\uhm(\Lambda(X,Z),\mathbf{E})=\uhom((X,Z),\mathbf{E}).
    \end{align*}
  \item This is an immediate consequence of the second part and the
    identification $\rep(X,Z,0)^{\vee}=\dR\pi_{*}j_{!}\one$.
  \item This follows from the third part by duality.
  \end{enumerate}
\end{proof}

We will need a similar result in the analytic setting. Thus let $X$ be
a complex manifold, and $Z$ a closed subset which is the union of
finitely many complex submanifolds. We call this an \emph{almost
  smooth analytic pair}, and as before, we denote by
$Z_{1},\ldots,Z_{p}$ the ``components'' of $Z$, namely the connected
components of the normalization of $Z$. We can then define,
analogously, $\Lambda(X,Z_{\bullet})$, $\Lambda(X,Z)$ and
$\uhom((X,Z),\bullet)$ (cf.~\cite[§2.2.1]{ayoub:galois1}).

\begin{thm}\label{thm:asp-motives-analytic}
  Let $\mathbf{E}$ be a projective stable
  $(\Done,\mathrm{usu})$-fibrant presheaf of complexes on
  $\ansm$. Then $\uhm(\Lambda(X,Z_{\bullet}),\mathbf{E})$ provides a
  model for $\dR\pi_{*}j_{!}\mathbf{E}|_{U}$ in $\anda$. Moreover,
  this identification is functorial in $\mathbf{E}$.
\end{thm}
\begin{proof}
  The proof is very similar to the one of the last theorem and we omit
  the details. (Also, the other parts of the previous theorem are
  equally true in the analytic setting, with almost identical proofs.)
\end{proof}

Later on, we will use the following relation between the two
descriptions of motives we just gave. Choose a projective stable
$(\A^{1},\tau)$-fibrant replacement $\mathbf{E}$ of the unit spectrum
$\one$. Also choose a projective stable $(\Done,\mathrm{usu})$-fibrant
replacement $\mathbf{E}'$ of $\one$. $\Anal^{*}\one\cong \one\to
\mathbf{E}'$ induces, by adjunction, $\one\to \Anal_{*}\mathbf{E}'$
and the latter is projective stable $(\A^{1},\tau)$-fibrant. It
follows that there is a morphism $\mathbf{E}\to\Anal_{*}\mathbf{E}'$
which induces $\Anal^{*}\mathbf{E}\to \mathbf{E}'$ rendering the
triangle
\begin{equation*}
  \xymatrix{\one\ar[r]\ar[rd]&\Anal^{*}\mathbf{E}\ar[d]\\
    &\mathbf{E}'}
\end{equation*}
commutative. Notice that the morphism $\Anal^{*}\mathbf{E}\to
\mathbf{E}'$ is necessarily a stable $(\Done,\mathrm{usu})$-local
equivalence.
\begin{lem}\label{lem:asp-motives-relation}
  Under the identifications of the two previous theorems, the
  following square commutes ($(X,Z)$ is an almost smooth pair):
  \begin{equation*}
    \xymatrix{\Anal^{*}\dR\pi_{*}j_{!}\one\ar@{=}[r]\ar[d]&\Anal^{*}\uhm(\Lambda(X,Z_{\bullet}),\mathbf{E})\ar[d]\\
      \dR\pi_{*}^{\An}j_{!}^{\An}\one\ar@{=}[r]&\uhm(\Lambda(X^{\An},Z_{\bullet}^{\An}),\mathbf{E}')}
  \end{equation*}
\end{lem}
\begin{proof}
  By adjunction, this is equivalent to the commutativity of the
  following outer diagram:
  \begin{equation*}
    \xymatrix{\dR\pi_{*}j_{!}\mathbf{E}|_{U}\ar@{=}[r]\ar[d]_{\mathrm{adj}}&\uhm(\Lambda(X,Z_{\bullet}),\mathbf{E})\ar[d]^{\mathrm{adj}}\\
      \dR\pi_{*}j_{!}(\Anal_{*}\Anal^{*}\mathbf{E})|_{U}\ar[d]&\uhm(\Lambda(X,Z_{\bullet}),\Anal_{*}\Anal^{*}\mathbf{E})\ar[d]\\
      \dR\pi_{*}j_{!}(\Anal_{*}\mathbf{E}')|_{U}\ar@{=}[r]\ar@{-}[d]_{\sim}&\uhm(\Lambda(X,Z_{\bullet}),\Anal_{*}\mathbf{E}')\ar@{-}[d]^{\sim}\\
      \dR\Anal_{*}\dR\pi_{*}^{\An}j_{!}^{\An}\mathbf{E}'|_{U^{\An}}\ar@{=}[r]&\Anal_{*}\uhm(\Lambda(X^{\An},Z_{\bullet}^{\An}),\mathbf{E}')}
  \end{equation*}
  Commutativity of the upper part follows from the functoriality
  statement in
  Theorem~\ref{thm:asp-motives}. For the
  lower square, one checks that $\Anal_{*}$ commutes with the
  relevant equalities in~\eqref{eq:asp-motives-model}. The main point is
  that $\Anal_{*}K(\mathbf{E}')=K(\Anal_{*}\mathbf{E}')$.
\end{proof}
This lemma states that the analytification functor is compatible with
our choices of models for the relative motives. We now want to prove
the analogous statement for the Betti realization functor. Factoring
the latter as $\Gamma\ev\Anal^{*}$ (where $\Gamma$ is the global
sections functor, cf.~section~\ref{sec:betti}), we reduce to showing
this compatibility for the composed functor $\Gamma\ev$. Thus let
$(X,Z)$ be an almost smooth analytic pair. By what we saw in
section~\ref{sec:NA} (or rather appendix~\ref{sec:coh-pair},
specifically Fact~\ref{lem:coh-pair-model}), the object
$\Gamma\ev\dR\pi_{*}j_{!}\one$ is modeled by the complex of
relative singular cochains on $(X,Z)$. Now suppose that in the
situation of the previous lemma, we choose $\mathbf{E}'$ to be
$\sgsdbb$ of Remark~\ref{rem:bounded-above-usu-fibrant}. Then we find
a canonical quasi-isomorphism
\begin{align}\label{eq:as-models}
  \begin{split}
    \sgs(X,Z)^{\vee}&\xrightarrow{\sim}
    \mathrm{Tot}(\sgs(X)^{\vee}\to\oplus_{i}\sgs(Z_{i})^{\vee}\to\cdots)\\
    &=\Gamma\uhm(\Lambda(X,Z_{\bullet}),\sgsd)\\
    &=\Gamma\ev\uhm(\Lambda(X,Z_{\bullet}),\sgsdbb).
  \end{split}
\end{align}
\begin{lem}\label{lem:compatibility-bti-as-models}
  The following square commutes:
  \begin{equation*}
    \xymatrix{\Gamma\ev\dR\pi_{*}j_{!}\one\ar@{=}[r]\ar@{-}[d]_{\sim}&\Gamma\ev\uhm(\Lambda(X,Z_{\bullet}),\sgsdbb)\\
      \dR\tilde{\pi}_{*}\tilde{j}_{!}\Lambda_{\mathrm{cst}}\ar@{=}[r]&\sgs(X,Z)^{\vee}\ar[u]^{\eqref{eq:as-models}}_{\sim}}
  \end{equation*}
\end{lem}
Here, we temporarily decorate the functors operating on sheaves on
locally compact topological spaces with a tilde to distinguish them
from their counterparts in the complex analytic world.
\begin{proof}
  Since the identification of
  Theorem~\ref{thm:asp-motives-analytic} inducing the top
  horizontal arrow is levelwise, we may prove the lemma staying
  completely on the effective level, thus decomposing the square as
  follows:
  \begin{equation*}
    \xymatrix@C-1em{\dR\Gamma\dR\pi_{*}j_{!}\Lambda\ar[r]^-{\sim}\ar@{-}[d]_{\sim}&\Gamma\pi_{*}K(\sgsd)\ar@{=}[rr]\ar@{=}[d]&&\sgsd(X,Z_{\bullet})\ar[dl]_-{\sim}\\
      \dR\tilde{\pi}_{*}\dR\iota_{X*}j_{!}\Lambda\ar[r]^-{\sim}&\tilde{\pi}_{*}\iota_{X*}K(\sgsd)\ar[r]^-{\sim}&\tilde{\pi}_{*}\mathrm{Tot}(\mathcal{S}_{X}\to\oplus_{i}i_{i*}\mathcal{S}_{Z_{i}}\to\cdots)\\
      \dR\tilde{\pi}_{*}\tilde{j}_{!}\Lambda_{\mathrm{cst}}\ar[u]^{\sim}\ar[r]^-{\sim}&\tilde{\pi}_{*}(\mathcal{S}_{X}\otimes\Lambda_{U})\ar[r]_-{\alpha}^-{\sim}&\tilde{\pi}_{*}\mathcal{K}_{(X,Z)}\ar[u]_{\sim}&\sgs(X,Z)^{\vee}\ar[l]^-{\beta}_-{\sim}\ar[uu]_{\sim}}
  \end{equation*}
  Recall (from appendix~\ref{sec:coh-pair}) that $\mathcal{S}_{X}$ is
  the sheaf of singular cochains on the topological space $X$,
  $U=X\backslash Z$, and $\mathcal{K}_{(X,Z)}$ is the kernel of the
  canonical morphism $\mathcal{S}_{X}\to
  i_{*}\mathcal{S}_{Z}$. $K(\sgsd)$ is defined as in the proof of
  Theorem~\ref{thm:asp-motives}, the maps $\alpha$ and $\beta$ are
  also defined in appendix~\ref{sec:coh-pair}.

  Everything except possibly the lower left inner diagram clearly
  commutes. Commutativity of this remaining diagram can be proved
  before applying $\dR\tilde{\pi}_{*}$. We replace the constant
  presheaf $\Lambda$ by $\sgsd$, and the constant sheaf
  $\Lambda_{\mathrm{cst}}$ by $\mathcal{S}_{U}$. Then the lemma
  follows from commutativity of the following diagram, which is
  obvious.
  \begin{equation*}
    \xymatrix@C-1em{\dR\iota_{X*}j_{!}\sgsd|_{U}\ar[r]&\iota_{X*}\mathrm{Tot}(\sgsd|_{X}\to\oplus_{i}i_{i*}\sgsd|_{Z_{i}}\to\cdot)\ar[r]&\mathrm{Tot}(\mathcal{S}_{X}\to\oplus_{i}i_{i*}\mathcal{S}_{Z_{i}}\to\cdot)\\
      \tilde{j}_{!}\iota_{U*}\sgsd|_{U}\ar[u]\ar[ru]\ar[r]&\tilde{j}_{!}\mathcal{S}_{U}\ar[ru]&\mathcal{S}_{X}\otimes\Lambda_{U}\ar[u]\ar@{=}[l]}
  \end{equation*}
\end{proof}

We end this section with the following result expressing a duality
between relative Morel-Voevodsky motives associated to complements of
two different divisors in a smooth projective scheme. We will make
essential use of it in the following section.
\begin{lem}[{cf.~\cite[p.~13]{nori-lectures},
    \cite[Lem.~1.13]{huber-mueller:nori}, \cite[Lem.~I.IV.2.3.5]{levine:mixed-motives}}]\label{motives-duality}
  Let $W$ be a smooth projective scheme of dimension $d$, $W_{0}\cup
  W_{\infty}$ a simple normal crossings divisor. Then there is a
  canonical isomorphism
  \begin{equation*}
    \rep(W-W_{\infty},W_{0}-W_{\infty},n)^{\vee}\cong \rep(W-W_{0},W_{\infty}-W_{0},2d-n)(-d)
  \end{equation*}
  in $\da$.
\end{lem}
\begin{proof}
  Fix the notation as in the following diagram:
  \begin{equation*}
    \xymatrix@C=5em{W-W_{\infty}\ar[r]^-{j_{\infty}}&W\\
      W-(W_{\infty}\cup W_{0})\ar[u]^{j_{0}'}\ar[r]_-{j_{\infty}'}&W-W_{0}\ar[u]_{j_{0}}}
  \end{equation*}
  The left hand side of the equality to establish is
  \begin{equation*}
    \pi_{W*}j_{\infty*}j'_{0!}j_{0}^{\prime!}j_{\infty}^{*}\pi_{W}^{*}\one[-n]\cong \pi_{W!}j_{\infty*}j'_{0!}j_{\infty}^{\prime*}j_{0}^{!}\pi_{W}^{!}\one(-d)[2d-n]
  \end{equation*}
  by relative purity, hence to prove the lemma it suffices to provide
  a canonical isomorphism $j_{0!}j_{\infty*}'\one\cong
  j_{\infty*}j'_{0!}\one$. The candidate morphism is obtained by
  adjunction from the composition
  \begin{equation*}
    j_{\infty*}'\xrightarrow[\sim]{\mathrm{adj}}j_{\infty*}'j_{0}^{\prime!}j_{0!}'\cong j_{0}^{!}j_{\infty*}'j_{0!}'.
  \end{equation*}
  It is clear that the candidate morphism is invertible on
  $W-W_{0}$, hence by localization, it remains to prove the same
  on $W_{0}$. Denote by $i_{\bullet}^{\bullet}$ the closed
  immersion complementary to $j_{\bullet}^{\bullet}$. We add a second
  subscript $0$ (resp.\ $\infty$) to denote the pullback of a morphism
  along $i_{0}$ (resp.\ $i_{\infty}$).

  Note that $i_{0}^{*}j_{0!}=0$ hence it suffices to prove
  $i_{0}^{*}j_{\infty*}j_{0!}'\one=0$. By one of the localization
  triangles for the couple $(W-W_{\infty},W-(W_{\infty}\cup W_{0}))$
  we can equivalently prove that the morphism
  \begin{equation}\label{eq:motives-duality-adj}
    \mathrm{adj}:i_{0}^{*}j_{\infty*}\one\to i_{0}^{*}j_{\infty*}i_{0*}'\one
  \end{equation}
  is invertible. The codomain of this morphism is isomorphic to
  \begin{equation*}
    i_{0}^{*}j_{\infty*}i_{0*}'\one\cong i_{0}^{*}i_{0*}j_{\infty0*}\one\cong j_{\infty0*}\one,
  \end{equation*}
  and under this identification, \eqref{eq:motives-duality-adj} corresponds
  to the morphism
  \begin{equation}\label{eq:motives-duality-bc}
    i_{0}^{*}j_{\infty*}\one\xrightarrow{\mathrm{bc}}
    j_{\infty0*}i_{0}^{\prime*}\one\cong j_{\infty0*}\one.
  \end{equation}
  Here, as in the rest of the proof, $\mathrm{bc}$ denotes the
  canonical base change morphism of the functors involved. Consider
  now the following diagram:
  \begin{equation*}
    \xymatrix{i_{0}^{*}i_{\infty!}i_{\infty}^{!}\one\ar[d]_{\alpha}\ar[r]&i_{0}^{*}\one\ar[d]_{\sim}\ar[r]&i_{0}^{*}j_{\infty*}\one\ar[r]\ar[d]_{\eqref{eq:motives-duality-bc}}&i_{0}^{*}i_{\infty!}i_{\infty}^{!}\one[-1]\ar[d]^{\alpha}\\
      i_{\infty0!}i_{\infty0}^{!}\one\ar[r]&\one\ar[r]&j_{\infty0*}\one\ar[r]&i_{\infty0!}i_{\infty0}^{!}\one[-1]}
  \end{equation*}
  The bottom row is a localization triangle, the top row arises from
  such by application of $i_{0}^{*}$. It is clear that the
  middle square commutes. $\alpha$ is defined as the
  composition
  \begin{equation*}
    i_{0}^{*}i_{\infty!}i_{\infty}^{!}\one\xrightarrow[\sim]{\mathrm{bc}}i_{\infty0!}i_{0
  \infty}^{*}i_{\infty}^{!}\one\xrightarrow{\mathrm{bc}}i_{\infty0!}i_{\infty0}^{!}i_{0}^{*}\one\cong
i_{\infty0!}i_{\infty0}^{!}\one,
  \end{equation*}
  and it is again easy to see that the left square commutes. Since
  there is only the zero morphism from
  $i_{0}^{*}i_{\infty!}i_{\infty}^{!}\one[-1]$ to $j_{\infty0*}\one$,
  this implies commutativity of the whole diagram. Now we have a
  morphism of distinguished triangles, and to prove invertibility
  of~\eqref{eq:motives-duality-bc} (and
  therefore~\eqref{eq:motives-duality-adj}) it suffices to prove
  invertibility of $\alpha$. Only the middle arrow in its definition
  needs to be considered, and for this we note that $\mathrm{bc}:i_{0
    \infty}^{*}i_{\infty}^{!}\one\to i_{\infty0}^{!}i_{0}^{*}\one$ is
  invertible by purity.
\end{proof}

\section{Main result}
\label{sec:main}

The goal of this section is to prove the following theorem. The two
main inputs are Proposition~\ref{pro:NA-AN} and
Theorem~\ref{thm:N-gen-A} which we prove subsequently.

\begin{thm}\label{thm:main}
  Assume that $\Lambda$ is a principal ideal domain. Then $\na$ and
  $\an$ are isomorphisms of Hopf algebras $\Has\cong\Hns$, inverse to
  each other. In particular, there is an isomorphism of affine
  pro-group schemes over $\spec(\Lambda)$:
  \begin{equation*}
    \grpa\cong\grpn.
  \end{equation*}
\end{thm}
\begin{proof}
We are going to prove
\begin{itemize}
\item $\an\na:\Hns\to\Hns$ is the identity;
\item $\na$ is surjective.
\end{itemize}
It will then follow that $\na$ and $\an$ are bialgebra
isomorphisms, inverse to each other. And since an antipode of a Hopf
algebra is unique, they are automatically isomorphisms of Hopf algebras.

\begin{proof}[Proof of~$\an\na=\id$]
  Consider the following triangle:
  \begin{equation*}
    \xymatrix@C=5em{\diagoodeff\ar[d]_{\tilde{\h}_{\bullet}}\ar[r]^-{\overline{\an}\h_{0}\tbti
        \rep}&\fcoMod{\Hns}\\
      \fcoMod{\Hn}\ar[ur]_{\varepsilon}}
  \end{equation*}
  By construction, the triangle commutes (up to \ugm isomorphism) for
  $\varepsilon=\overline{\an\nae}$. By Proposition~\ref{pro:NA-AN}
  below it also commutes (up to \ugm isomorphism) for
  $\varepsilon=\overline{\iota}$, where $\iota:\Hn\to\Hns$ is the
  canonical localization map. By universality of Nori's category
  (Theorem~\ref{pro:N-universal}), we must therefore have a monoidal
  isomorphism of functors $\overline{\an\nae}\cong
  \overline{\iota}$.
  By~\cite[II, 3.3.1]{saavedra-cat.tann.1972}, we must then have
  $\an\nae=\iota$. Hence $\an\na:\Hns\to\Hns$ is the identity.
\end{proof}

\begin{proof}[Proof of surjectivity of $\na$]
  We use Theorem~\ref{thm:N-gen-A} below to deduce that the
  comultiplication $\cm:\Has\to\Has\otimes\Has$ is equal to
  \begin{equation*}
    \Has\xrightarrow{\coa}\Hns\otimes\Has\xrightarrow{\na\otimes\id}\Has\otimes\Has
  \end{equation*}
  for some coaction $\coa$ of $\Hns$ on $\Has$. Composing $\cm$ with
  the counit
  $\id\otimes \cu:\Has\otimes\Has\to\Has\otimes\Lambda\cong\Has$
  yields the identity, therefore also
  \begin{equation*}
    \id=(\id\otimes \cu)\circ (\na\otimes
    \id)\circ \coa=\na\circ (\id\otimes \cu)\circ \coa.
  \end{equation*}
  It follows that $\na$ is surjective.
\end{proof}
\end{proof}

\begin{pro}\label{pro:NA-AN}
  Assume that $\Lambda$ is a principal ideal domain. The following
  square commutes up to an isomorphism of \ugm representations:
  \begin{equation*}
    \xymatrix@C=5em{\diagoodeff\ar[d]_{\tilde{\h}_{\bullet}}\ar[r]^-{\rep}&\da\ar[d]^{\h_{0}\circ\tbti}\\
      \coMod{\Hns}&\coMod{\Has}\ar[l]^-{\overline{\an}}}
  \end{equation*}
\end{pro}
\begin{proof}
  This is a rather long and tedious verification. We will proceed in several steps.

  \begin{enumerate}[wide,itemsep=1em,label=\textbf{Step \arabic*}.]
  \item By Proposition~\ref{pro:NA-eff-ugm}, we already know that
    after composition with the forgetful functor
    $\coMod{\Hns}\to\Mod{\Lambda}$, the two \ugm representations are
    naturally isomorphic. Call the isomorphism $\eta$. It therefore
    suffices to prove that the components of $\eta$ are compatible
    with the $\Hns$-comodule structure.
  \item 
  
    Let $v=(X,Z,n)$ be an arbitrary vertex in $\diagoodeff$. We find
    by resolution of singularities a vertex $v'=(X',Z',n)$ and an edge
    $p:v'\to v$ such that $(X',Z')$ is almost smooth and
    $\h_{\bullet}(p)$ is an isomorphism. Consider the following
    commutative square in $\Mod{\Lambda}$:
    \begin{equation*}
      \xymatrix@C=5em{\h_{\bullet}(v')\ar[r]^-{\eta_{v'}}\ar[d]_{\h_{\bullet}(p)}&\overline{\an}\h_{0}\tbti
        \rep(v')\ar[d]^{\overline{\an}\h_{0}\tbti \rep(p)}\\
        \h_{\bullet}(v)\ar[r]_-{\eta_{v}}&\overline{\an}\h_{0}\tbti \rep(v)}
    \end{equation*}
    All arrows are invertible and both vertical arrows are
    $\coMod{\Hns}$-morphisms. We may therefore assume that $(X,Z)$ is
    almost smooth.
  \item Now consider the following diagram in $\Mod{\Lambda}$:
    \begin{equation*}
      \xymatrix@C=5em{\h_{0}\bti
        \rep(X,Z,n)\ar[d]^{\eta}_{\sim}\ar[r]^{\sim}_{\text{Thm.}~\ref{thm:asp-motives}}&\h_{0}\btie\Lambda(X,Z)[n]\ar[d]^{\sim}_{\mathrm{Cor.~\ref{cor:AN-eff}}}\\
        \h_{n}(X,Z)&\h_{0}\rea\Lambda(X,Z)\ar[l]_{\sim}^{\eqref{eq:comodule-iso}}[n]}
    \end{equation*}
    By Lemma~\ref{lem:comodule-iso} , the bottom horizontal arrow is
    compatible with the $\Hn$-coaction; the same is clearly true for
    the top horizontal and the right vertical one. We are thus reduced
    to show commutativity of this square, and it suffices to do so
    before applying $\h_{n}$.
  \item  Modulo
    the identification of $\btie$ with $\dL\sgs^{*}\circ\Anal^{*}$ of
    Proposition~\ref{pro:sgs-bti}, the composition of the right
    vertical and the bottom horizontal arrow can be equivalently
    described as the composition
    \begin{equation*}
      \dL\sgs^{*}\Anal^{*}\Lambda(X,Z)\to\dL\sgs^{*}\Lambda(X^{\An},Z^{\An})\to\sgs(X^{\An},Z^{\An}).
    \end{equation*}
    Thus the square above will commute if the following diagram does:
    \begin{equation*}
      \xymatrix{\bti\dL\pi_{!}\dR
        j_{*}j^{*}\pi^{!}\one\ar@{-}[r]^{\sim}\ar@{-}[d]_{\sim}&\bti
        \dL\sus_{T}\Lambda(X,Z_{\bullet})\ar@{-}[d]_{\sim}\ar@{-}[r]^{\sim}&\dL\sgs^{*}\Anal^{*}\Lambda(X,Z_{\bullet})\ar[d]^{\sim}\\
        (\bti\dR\pi_{*}j_{!}\one)^{\vee}\ar@{=}[r]\ar@{-}[d]_{\sim}&(\bti\uhm(\Lambda(X,Z_{\bullet}),\mathbf{E}))^{\vee}\ar@{-}[d]_{\sim}&\dL\sgs^{*}\Lambda(X^{\An},Z_{\bullet}^{\An})\ar[d]^{\sim}\\
        (\dR\tilde{\pi}^{\An}_{*}\tilde{j}_{!}^{\An}\Lambda_{\mathrm{cst}})^{\vee}\ar@{=}[r]&\sgs(X^{\An},Z^{\An})^{\vee\vee}&\sgs(X^{\An},Z^{\An})\ar[l]^{\sim}}
    \end{equation*}
    The arrows in the top left square are induced by the
    identifications in Theorem~\ref{thm:asp-motives} ($\mathbf{E}$ is
    a projective stable $(\A^{1},\tau)$-fibrant replacement of the
    unit spectrum) and duality which makes the square clearly
    commutative. Commutativity of the lower left square is
    Lemma~\ref{lem:asp-motives-relation}
    and~\ref{lem:compatibility-bti-as-models}. We are reduced to prove commutativity of the right half.
  \item Consider the following diagram (all ``arrows'' are
    isomorphisms, either canonical or introduced before):
    \begin{equation*}
      \scalebox{0.671}{\xymatrix@C=.8em{(\bti\dL\sus\Lambda(X,Z_{\bullet}))^{\vee}\ar@{-}[r]\ar@{-}[d]&(\btie\Lambda(X,Z_{\bullet}))^{\vee}\ar@{-}[r]&(\dL\sgs^{*}\Anal^{*}\Lambda(X,Z_{\bullet}))^{\vee}\ar@{-}[r]&(\dL\sgs^{*}\Lambda(X^{\An},Z^{\An}_{\bullet}))^{\vee}\ar@{.}[d]\\
          \dR\Gamma\dR\ev(\Anal^{*}\dL\sus\Lambda(X,Z_{\bullet}))^{\vee}\ar@{-}[d]\ar@{-}[r]&\dR\Gamma\dR\ev(\dL\sus\Lambda(X^{\An},Z^{\An}_{\bullet}))^{\vee}\ar@{-}[d]\ar@{-}[r]&\dR\Gamma\dR\ev\dL\sus\Lambda(X^{\An},Z^{\An}_{\bullet})^{\vee}\ar@{-}[d]\ar@{-}[r]&\dL\sgs^{*}\Lambda(X^{\An},Z^{\An}_{\bullet})^{\vee}\ar@{.}[d]\\
          \bti\uhm(\Lambda(X,Z_{\bullet}),\mathbf{E})\ar@{-}[r]&\dR\Gamma\dR\ev\uhm(\Lambda(X^{\An},Z^{\An}_{\bullet}),\sgsdbb)\ar@{-}[r]&\dR\Gamma\uhm(\Lambda(X^{\An},Z^{\An}_{\bullet}),\sgsd)\ar@{.}[r]&\dL\sgs^{*}\uhm(\Lambda(X^{\An},Z^{\An}_{\bullet}),\sgsd)
        }}
    \end{equation*}
    The upper part clearly commutes as does the lower right
    square. For the lower left square we need to prove commutative
    \begin{equation*}
      \xymatrix{\Anal^{*}\uhm(\Lambda(X,Z_{\bullet}),\mathbf{E})\ar[d]\ar@{=}[r]&\Anal^{*}\dR\uhm(\dL\sus\Lambda(X,Z_{\bullet}),\one)\ar[d]\\
        \uhm(\Lambda(X^{\An},Z^{\An}_{\bullet}),\sgsdbb)\ar@{=}[r]&\dR\uhm(\dL\sus\Lambda(X^{\An},Z^{\An}_{\bullet}),\one)}
    \end{equation*}
    and this is done as in Lemma~\ref{lem:asp-motives-relation}.  The
    lower middle square is easily seen to commute hence, using
    duality, it only remains to prove that the composition of the
    dotted arrows is equal to
    \begin{equation*}
      \dR\Gamma\uhm(\Lambda(X^{\An},Z^{\An}_{\bullet}),\sgsd)\leftarrow\sgsd(X^{\An},Z^{\An})\rightarrow (\dL\sgs^{*}\Lambda(X^{\An},Z^{\An}_{\bullet}))^{\vee}.
    \end{equation*}
  \item Notice that $\sgsd=\sgs_{*}\Lambda$. Then, writing $B$
    for $\Lambda(X^{\An},Z^{\An}_{\bullet})$, we reduce to prove
    commutative:
    \begin{equation*}
      \xymatrix{\Gamma\uhm(B,\sgs_{*}\Lambda)\ar[d]\ar@{-}[r]&\dL\sgs^{*}\uhm(B,\sgs_{*}\Lambda)\ar[d]\ar@{=}[r]&\dL\sgs^{*}\dR\uhm(B,\Lambda)\ar[d]\\
        \Gamma\sgs_{*}\uhm(\sgs^{*}B,\Lambda)\ar@{-}[r]\ar@{-}[dr]&\dL\sgs^{*}\sgs_{*}\uhm(\sgs^{*}B,\Lambda)\ar[d]&\dR\uhm(\dL\sgs^{*}B,\Lambda)\\
        &\uhm(\sgs^{*}B,\Lambda)\ar[ru]
      }
    \end{equation*}
    Here we used that $\sgs\circ\iota_{\pt}:\pt\to \ch(\Lambda)$ takes
    the value $\sgs(\pt)\simeq\Lambda$ hence $\Gamma\circ\sgs_{*}$ is
    canonically quasi-isomorphic to the identity. Since the
    undecorated functors $\sgs^{*}$ and $\sgs_{*}$ appearing are only
    applied to cofibrant, respectively fibrant, objects they can be
    identified with their derived counterparts and the diagram is
    easily seen to commute.
  \end{enumerate}
\end{proof}

\begin{thm}\label{thm:N-gen-A}
  Assume that $\Lambda$ is a principal ideal domain. The bialgebra
  $\Has$ considered as a comodule over itself lies in the essential
  image of
  \begin{equation*}
    \coMod{\Hns}\xrightarrow{\overline{\na}}\coMod{\Has}.
  \end{equation*}
\end{thm}
\begin{proof}
  We may prove this statement for the Nisnevich topology. Ayoub gives
  in~\cite[Thm.~2.67]{ayoub:galois1} an explicit model for the
  symmetric $T$-spectrum $\rbtinis\Lambda$ which we are now going to
  describe at a level of detail appropriate for our proof.

  Recall the category $\mathcal{V}_{\et}(\Dbar^{n}/\A^{n})$ ($n\geq
  0$) whose objects are étale neighborhoods of the closed polydisk
  $\Dbar^{n}$ inside affine space $\A^{n}$ (for the precise definition
  see~\cite[§2.2.4]{ayoub:galois1}). It is a cofiltered
  category. Forgetting the presentation as a scheme over $\A^{n}$
  defines a canonical functor
  $\Dbar_{\et}^{n}:\mathcal{V}_{\et}(\Dbar^{n}/\A^{n})\to\sm$. In
  other words we obtain a pro-smooth scheme. We write $\Dbar_{\et}$
  for the associated cocubical object in pro-smooth schemes where the
  faces $d_{i,\varepsilon}$ are induced from the faces in $\A^{n}$
  (the ``coordinate hyperplanes'' through 0 and 1). For $n\in\N$ and
  $\varepsilon=0,1$ write $\partial_{\varepsilon}\Dbar^{n}_{\et}$ for
  the union of the faces $d_{i,\varepsilon}(\Dbar^{n-1}_{\et})$, where
  $i$ runs through $1,\ldots,n$. Also write $\partial\Dbar^{n}_{\et}$
  for the union
  $\partial_{0}\Dbar^{n}_{\et}\cup\partial_{1}\Dbar^{n}_{\et}$, and
  $\partial_{1,1}\Dbar^{n}_{\et}$ for the union of all
  $d_{i,\varepsilon}(\Dbar^{n-1}_{\et})$ except
  $(i,\varepsilon)=(1,1)$.

  We obtain the bicomplex $N(\uhom(\Dbar_{\et},K))$ which in degree
  $n$ (in the direction of the cocubical dimension) is given by
  $\uhom((\Dbar^{n}_{\et},\partial_{1,1}\Dbar^{n}_{\et}),K)$, and
  whose differential in degree $n>0$ is $d_{1,1}$.\footnote{Given a
    pro-object $(X_{i},Z_{i})_{i\in I}$ of almost smooth pairs,
    $\uhom((X_{i},Z_{i})_{i},K)$ takes a smooth scheme $Y$ to
    \begin{equation*}
      \varinjlim_{i\in I}K(Y\times X_{i},Y\times Z_{i}).      
    \end{equation*}} In particular, the cycles
  in degree $n>0$ are given by
  $\uhom((\Dbar^{n}_{\et},\partial\Dbar^{n}_{\et}),K)$. Thus we obtain
  a canonical morphism of bicomplexes
  \begin{equation*}
    \uhom((\Dbar^{n}_{\et},\partial\Dbar^{n}_{\et}),K)[-n]\to N^{\leq n}(\uhom(\Dbar_{\et},K))
  \end{equation*}
  where the right hand side denotes the bicomplex truncated at degree
  $n$ from above. One can check that this induces a quasi-isomorphism
  on the associated total complexes whenever $K$ is injective fibrant.

  Taking the total complex of the bicomplex $N(\uhom(\Dbar_{\et},K))$
  (resp.\ $N^{\leq n}(\uhom(\Dbar_{\et},K))$) we obtain an endofunctor
  $\nsget$ (resp.\ $\nsgetn$) of presheaves of complexes on smooth
  schemes. It extends canonically to an endofunctor on symmetric
  $T$-spectra. Let $\mathbf{E}$ be an injective stable
  $(\A^{1},\mathrm{Nis})$-fibrant replacement of the unit spectrum
  $\one$. \cite[Thm.~2.67]{ayoub:galois1} states that $\rbtinis\one$
  is given explicitly by the symmetric $T$-spectrum
  \begin{equation*}
    \mathbf{Sing}_{\et}^{\mathbb{D},\infty}(\mathbf{E}):=\varinjlim_{r}s_{-}^{r}\nsget(\mathbf{E})[2r],
  \end{equation*}
  where $s_{-}$ denotes the ``shift down'' functor (so that
  $s_{-}(\mathbf{E})_{m}=\mathbf{E}_{m+1}$;
  see~\cite[Déf.~4.3.13]{ayoub07-thesis}). As we will not need a
  description of the transition morphisms in the sequential colimit
  above, we content ourselves with referring
  to~\cite[Déf.~2.65]{ayoub:galois1}. Let $Q:\spt_{T}\U\sm\to\danis$
  denote the canonical localization functor, and consider the
  following canonical morphisms:
  \begin{align}\label{eq:N-gen-A.1}\notag
    \varinjlim_{r,n} \h_{0}\tbtinis
    Q(s_{-}^{r}\uhom((\Dbar^{n}_{\et},\partial\Dbar^{n}_{\et}),\mathbf{E})[2r-n])&\to\varinjlim_{r,n} \h_{0}\tbtinis
    Q(s_{-}^{r}\nsgetn(\mathbf{E})[2r])\\
    &\to \h_{0}\tbtinis Q(\mathbf{Sing}_{\et}^{\mathbb{D},\infty}(\mathbf{E}))
  \end{align}
  in $\coMod{\Has}$. The last term is the bialgebra $\Has$ considered
  as a comodule over itself. We are going to show
  \begin{enumerate}[label=(\arabic*)]
  \item \label{cl:N-gen-A.invertible}that the composition
    in~\eqref{eq:N-gen-A.1} is invertible, and
  \item \label{cl:N-gen-A.essim}that the comodules in the filtered
    system on the left hand side are in the essential image of
    $\overline{\na}$.
  \end{enumerate}
  This is enough since, as seen in the proof of
  Theorem~\ref{thm:main}, the previous proposition implies that
  $\overline{\an\na}\cong\id$ hence the essential image of
  $\overline{\na}$ is a full subcategory of $\coMod{\Has}$ (since both
  $\overline{\na}$ and $\overline{\an}$ are faithful) closed under
  small colimits (by Fact~\ref{thm:comod-basic}).

  \begin{proof}[Proof of~\ref{cl:N-gen-A.invertible}]
    It suffices to prove this after forgetting the comodule
    structure. Just as in the case of the étale singular complex there
    is an endofunctor $\mathbf{Sing}^{\mathbb{D},\infty}$ on symmetric
    $\Anal^{*}(T)$-spectra defined using $\sgd$ instead of $\nsget$
    (cf.~\cite[Déf.~2.45]{ayoub:galois1}).  Denote by
    $F:\spt_{T}\U\sm\to\Mod{\Lambda}$ the composition of functors
    $\h_{0}\Gamma \ev \mathbf{Sing}^{\mathbb{D},\infty}\Anal^{*}$ and
    notice that
    \begin{enumerate}[label=(\alph*)]
    \item $F$ commutes with filtered colimits, by
      construction;\label{btimodel-colim}
    \item $F$ takes levelwise quasi-isomorphisms of symmetric
      $T$-spectra to isomorphisms of modules, as follows essentially
      from~\cite[Lem.~2.55]{ayoub:galois1};\label{btimodel-qis}
    \item $F$ applied to a projective stable
      $(\A^{1},\mathrm{Nis})$-fibrant spectrum $\mathbf{K}$ is a model
      for $\h_{0}\btinis \mathbf{K}$, by~\cite[Lem.~2.72 and
      Thm.~2.48]{ayoub:galois1}.\label{btimodel-model}
    \end{enumerate}
    We claim that the morphism of $\Lambda$-modules
    underlying~\eqref{eq:N-gen-A.1} can be identified with the
    composition
    \begin{align}\label{eq:N-gen-A.2}\notag
      \varinjlim_{r,n} F(s_{-}^{r}\uhom((\Dbar^{n}_{\et},\partial\Dbar^{n}_{\et}),\mathbf{E})[2r-n])&\to\varinjlim_{r,n} F(s_{-}^{r}\nsgetn(\mathbf{E})[2r])\\
                                                                                                    &\to F(\mathbf{Sing}_{\et}^{\mathbb{D},\infty}(\mathbf{E})).
    \end{align}
    This follows from~\ref{btimodel-model} because both
    $\mathbf{Sing}_{\mathrm{et}}^{\mathbb{D},\infty}(\mathbf{E})$ and
    $\uhom((\Dbar^{n}_{\et},\partial\Dbar^{n}_{\et}),\mathbf{E}))$ are
    projective stable $(\A^{1},\mathrm{Nis})$-fibrant, as follows
    from~\cite[Thm.~2.67]{ayoub:galois1} for the first, and from our
    proof of Theorem~\ref{thm:asp-motives} together
    with~\cite[Lem.~2.69]{ayoub:galois1} for the second.  But the
    first arrow in~\eqref{eq:N-gen-A.2} is invertible
    by~\ref{btimodel-qis}, and the second one by~\ref{btimodel-colim}
    so we conclude that~\eqref{eq:N-gen-A.1} is invertible.
  \end{proof}
  \begin{proof}[Proof of~\ref{cl:N-gen-A.essim}]
    Next we fix $(r,n)\in\N^{2}$ and consider the canonical morphism
    \begin{multline*}
      \varinjlim_{(X,x)\in\mathcal{V}_{\et}(\Dbar^{n}/\A^{n})}\h_{0}\tbtinis
      Q(s_{-}^{r}\uhom((X,\partial X),\mathbf{E})[2r-n])\to\\
      \h_{0}\tbtinis
      Q(s_{-}^{r}\uhom((\Dbar^{n}_{\et},\partial\Dbar^{n}_{\et}),\mathbf{E})[2r-n]).
    \end{multline*}
    The same argument as above establishes invertibility of this arrow
    and reduces us to show that the comodules in the filtered system
    on the left hand side lie in the essential image of
    $\overline{\na}$. Hence fix
    $(X,x)\in\mathcal{V}_{\et}(\Dbar^{n}/\A^{n})$. By resolution of
    singularities there is a smooth projective scheme $W$ and a simple
    normal crossings divisor $W_{0}\cup W_{\infty}$ on $W$ together
    with a projective surjective morphism $p:W - W_{\infty}\to X$ such
    that $p^{-1}(\partial X)=W_{0}-W_{\infty}$ and
    $p|_{W-p^{-1}(\partial X)}:W-p^{-1}(\partial X)\to X-\partial X$
    is an isomorphism.  Therefore, canonically,
    $\rep(X,\partial X,0)\cong \rep(W-W_{\infty},W_{0}-W_{\infty},0)$,
    and we obtain in $\danis$:
    \begin{align*}
      Q(s_{-}^{r}\uhom((X,\partial X),\mathbf{E})[2r-n])
      &\cong
        \uhom((X,\partial X),\mathbf{E})[-n](r)\\
      &\cong \rep(X,\partial X,0)^{\vee}[-n](r)\\
      &\cong \rep(W-W_{\infty},W_{0}-W_{\infty},0)^{\vee}[-n](r)\\
      &\cong \rep(W-W_{0},W_{\infty}-W_{0},n)(r-n)\\
      &\cong \rep(W-W_{0},W_{\infty}-W_{0},n)\otimes^{\dL}
        \rep(\mathbb{G}_{m},\{1\},1)^{\otimes^{\dL}(r-n)},
    \end{align*}
    where we used~\cite[Thm.~4.3.38]{ayoub07-thesis} for the first,
    Theorem~\ref{thm:asp-motives} for the second, and
    Lemma~\ref{motives-duality} for the penultimate
    isomorphism. Applying $\h_{0}\tbti$ to these isomorphisms, and
    using~\eqref{eq:NA-eff} as well as~\eqref{eq:NA} we obtain the
    following sequence of isomorphisms
    \begin{align*}
      &\h_{0}\tbtinis Q(s_{-}^{r}\uhom((X,\partial X),\mathbf{E})[2r-n])\\
      &\cong\h_{0}\tbtinis \left(
        \rep(W-W_{0},W_{\infty}-W_{0},n)\otimes^{\dL}
        \rep(\mathbb{G}_{m},\{1\},1)^{\otimes^{\dL}(r-n)}
        \right)\\
      &\cong\overline{\na} \left(
        \tilde{\h}_{\bullet}(W-W_{0},W_{\infty}-W_{0},n)\otimes
        \tilde{\h}_{\bullet}(\mathbb{G}_{m},\{1\},1)^{\otimes(r-n)}
        \right),
    \end{align*}
    which concludes the proof.
  \end{proof}

\end{proof}

\appendix{}
  \section{Nori's Tannakian formalism in the monoidal setting}
  \label{sec:N-ugm}
  In this section we indicate briefly which modifications
  to~\cite[App.~B]{huber-mueller:nori} have to be made in order to
  justify our arguments in the main body of the text regarding Nori's
  Tannakian formalism. Most importantly we seek to obtain a
  universality statement for Nori's construction in the monoidal
  setting. Something similar was undertaken by Bruguières
  in~\cite{bruguieres-tannaka-nori}, and for the main proof below we
  follow his ideas. However the results there on monoidal
  representations do not seem to apply directly to Nori's construction
  since there is no obvious monoidal structure (in the sense
  of~\cite{bruguieres-tannaka-nori}) on Nori's diagrams.\footnote{This
    is related to the problem discussed
    in~\cite[Rem.~B.13]{huber-mueller:nori}.}

  A \emph{graded diagram} and a \emph{commutative product structure}
  on such a graded diagram are defined as
  in~\cite[Def.~B.14]{huber-mueller:nori}. From now on, fix such a
  graded diagram $\mathscr{D}$ with a commutative product
  structure. Let $(\mathcal{C},\otimes)$ be an additive (symmetric,
  unitary) monoidal category. A \emph{graded multiplicative
    representation} $T:\mathscr{D}\to \mathcal{C}$ is a representation
  of $\mathscr{D}$ in $\mathcal{C}$ together with a choice of
  isomorphisms
  \begin{equation*}
    \tau_{(f,g)}:T(f\times
    g)\to T(f)\otimes T(g)
  \end{equation*}
  for any vertices $f$ and $g$ of $\mathscr{D}$, satisfying (1)-(5)
  of~\cite[Def.~B.14]{huber-mueller:nori}. \emph{Unital graded
    multiplicative (\ugm) representations} are then defined as
  in~\cite[Def.~B.14]{huber-mueller:nori}. A \emph{\ugm
    transformation} $\eta:T\to U$ between two unital graded
  multiplicative representations $T,U:\mathscr{D}\to \mathcal{C}$ is a
  family of morphisms in $\mathcal{C}$:
  \begin{equation*}
    \eta_{f}:T(f)\to U(f),
  \end{equation*}
  compatible with edges in $\mathscr{D}$ and the choices of
  isomorphisms $\tau$, and such that
  $\eta_{\mathrm{id}}=\mathrm{id}$. $\eta$ is a \emph{\ugm
    isomorphism} if all its components are invertible.

  From now on, fix also a \ugm representation
  $T:\mathscr{D}\to\fMod{\Lambda}$ taking values in projective modules
  (we assume $\Lambda$ to be of global dimension at most 2;
  see~\cite[§5.3]{bruguieres-tannaka-nori}). By Nori's theorem
  (\cite[Thm.~1.6]{nori-lectures},
  \cite[Pro.~B.8]{huber-mueller:nori}), there is a universal abelian
  $\Lambda$-linear category $\mathcal{C}(T)$ with a representation
  $\tilde{T}:\mathscr{D}\to \mathcal{C}(T)$, through which $T$ factors
  via a faithful exact $\Lambda$-linear functor
  $\ff_{T}:\mathcal{C}(T)\to\fMod{\Lambda}$. Nori also showed
  (see~\cite[Pro.~B.16]{huber-mueller:nori}) that in this case
  $\mathcal{C}(T)$ carries naturally a (right exact) monoidal
  structure such that $\ff_{T}$ is a monoidal functor. It is obvious
  from the construction of this monoidal structure that $\tilde{T}$ is
  a \ugm representation. The following theorem states that these data
  are \emph{universal}.
  \begin{thm}\label{thm:N-ugm-universal}
    Given a right exact monoidal abelian $\Lambda$-linear category
    $\mathcal{C}$ and a factorization of $T$ into
    \begin{equation*}
      \mathscr{D}\xrightarrow{S}\mathcal{C}\xrightarrow{\ff_{S}}\fMod{\Lambda}
    \end{equation*}
    with $S$ \ugm and $\ff_{S}$ a faithful exact $\Lambda$-linear
    monoidal functor, there exists a monoidal functor (unique up to
    unique monoidal isomorphism) $F:\mathcal{C}(T)\to \mathcal{C}$
    making the following diagram commutative (up to monoidal
    isomorphism).
    \begin{equation*}
      \xymatrix@C=5em{\mathscr{D}\ar[r]^-{S}\ar[d]_{\tilde{T}}&\mathcal{C}\ar[d]^{\ff_{S}}\\
        \mathcal{C}(T)\ar[r]_-{\ff_{T}}\ar@{.>}[ru]_{F}&\fMod{\Lambda}}
    \end{equation*}
    Moreover, $F$ is faithful exact $\Lambda$-linear.
  \end{thm}
  Explicitly, there exists a monoidal functor $F:\mathcal{C}(T)\to
  \mathcal{C}$, a \ugm isomorphism
  $\alpha:S\tilde{T}\xrightarrow{\sim}F$, and a monoidal isomorphism
  $\beta:\ff_{T}\xrightarrow{\sim}\ff_{S}F$ such that
  $\ff_{S}\alpha=\beta\tilde{T}$. Moreover given another triple
  $(F',\alpha',\beta')$ satisfying these conditions, there exists a
  unique monoidal isomorphism $\gamma:F'\xrightarrow{\sim}F$
  transforming $\alpha$ to $\alpha'$ and $\beta$ to $\beta'$.
  \begin{proof}
    Given \cite[Pro.~B.8 and B.16]{huber-mueller:nori}, the only thing
    left to prove is that the functor and transformations whose
    existence is asserted are monoidal. This could be proven by going
    through the construction of these and checking monoidality
    directly. Alternatively, one can deduce the monoidal structure
    from the existence of the functor and transformations alone
    without referring to their construction. We sketch the latter
    proof which is due to Bruguières. For the details we refer the
    reader to~\cite{bruguieres-tannaka-nori}.

    Let $\Psi_{X,Y}$ be the composition
    \begin{equation*}
      \ff_{S}(FX\otimes FY)= \ff_{S}FX\otimes
      \ff_{S}FY\cong\ff_{S}F(X\otimes Y),
    \end{equation*}
    for any $X,Y\in\mathcal{C}(T)$. This defines a natural isomorphism
    of functors. Since $\ff_{S}$ is faithful, it suffices to construct
    morphisms $\Phi_{X,Y}:FX\otimes FY\to F(X\otimes Y)$ which realize
    $\Psi_{X,Y}$. Thus consider the class
    \begin{equation*}
      L=\{(X,Y)\in \mathcal{C}(T)\times\mathcal{C}(T)\mid \exists
      \Phi_{X,Y}: \ff_{S}\Phi_{X,Y}=\Psi_{X,Y}\}.
    \end{equation*}
    Notice that $L$ contains all pairs in the image of $\tilde{T}$:
    \begin{equation*}
      F\tilde{T}f\otimes F\tilde{T}g\cong Sf\otimes Sg\to S(f\times
      g)\cong F\tilde{T}(f\times g)\to F(\tilde{T}f\otimes \tilde{T}g)
    \end{equation*}
    can (and has to) be taken as $\Phi_{\tilde{T}f,\tilde{T}g}$. Now
    for fixed $f$ the functors $F\tilde{T}f\otimes F(\bullet)$ and
    $F(\tilde{T}f\otimes\bullet)$ are
    exact 
    hence one can define $\Phi_{\tilde{T}f,\bullet}$ on the
    subcategory of $\mathcal{C}(T)$ containing the image of
    $\tilde{T}$ and closed under kernels, cokernels and direct
    sums. But this is all of $\mathcal{C}(T)$. By symmetry one sees
    that $L$ contains all pairs $(X,Y)$ where one of $X$ or $Y$ is
    contained in the image of $\tilde{T}$. Now a similar argument
    shows that $L$ also contains all pairs $(X,Y)$ where one of
    $\ff_{T}X$ or $\ff_{T}Y$ is projective (since then the functors
    considered above are still exact). Finally, one uses that every
    object in $\mathcal{C}(T)$ is a quotient of an object with
    underlying projective $\Lambda$-module to conclude that $L$
    consists of all pairs of objects in $\mathcal{C}(T)$.

    It is obvious from the definition of
    $\Phi_{\tilde{T}f,\tilde{T}g}$ that $\alpha$ is monoidal, and from
    the definition of $\Psi_{X,Y}$ that $\beta$ is as well. It is an
    easy exercise to prove that $\gamma$ is monoidal as well.
  \end{proof}

  \section{Relative cohomology}
  \label{sec:coh-pair}

  It is well-known that singular and sheaf cohomology agree on locally
  contractible topological spaces. The same is true for pairs of such
  spaces. However, we have not been able to find in the literature the
  statements in the form we need them in the main body of the chapter
  (in particular in section~\ref{sec:NA}) although the book of Bredon
  \cite{bredon-sheaftheory} comes close. We will freely use the
  results of~\cite[§III.1]{bredon-sheaftheory}. $\Lambda$ is a fixed
  principal ideal domain. All topological spaces are assumed locally
  contractible and paracompact.

  \subsection{Model}

  For a topological space $X$, denote by $\mathcal{S}_{X}$ the complex
  of sheaves of singular cochains on $X$ with values in
  $\Lambda$. This is a flabby resolution of the constant sheaf
  $\Lambda$. Moreover, the canonical map
  $\sgs(X)^{\vee}\to\mathcal{S}_{X}(X)$ is a quasi-isomorphism.

  Now let $i:Z\inj X$ a closed subset with open complement $j:U\inj
  X$. We denote by $\Lambda_{U}$ (respectively $\Lambda_{Z}$) the
  constant sheaf $\Lambda$ supported at $U$ (respectively $Z$), \ie{}
  $\Lambda_{U}=j_{!}j^{*}\Lambda_{X}$ (resp.\
  $\Lambda_{Z}=i_{*}i^{*}\Lambda_{X}$). The canonical morphism
  $\mathcal{S}_{X}\otimes\Lambda_{Z}\to i_{*}\mathcal{S}_{Z}$ induces
  the diagram of solid arrows in the category of complexes of sheaves
  on $X$ with exact rows:
  \begin{equation}\label{eq:coh-pair-alpha}
    \xymatrix{0\ar[r]&\mathcal{K}_{(X,Z)}\ar[r]&\mathcal{S}_{X}\ar[r]&i_{*}\mathcal{S}_{Z}\\
      0\ar[r]&\mathcal{S}_{X}\otimes\Lambda_{U}\ar[r]\ar@{.>}[u]_{\alpha}&\mathcal{S}_{X}\ar@{=}[u]\ar[r]&\mathcal{S}_{X}\otimes\Lambda_{Z}\ar[u]}
  \end{equation}
  We obtain a unique morphism $\alpha$ rendering the diagram
  commutative. It induces a quasi-isomorphism after taking global
  sections.

  Similarly, $\beta$ is the unique morphism of complexes making the
  following diagram commute:
  \begin{equation}\label{eq:coh-pair-beta}
    \xymatrix{0\ar[r]&\mathcal{K}_{(X,Z)}(X)\ar[r]&\mathcal{S}_{X}(X)\ar[r]&i_{*}\mathcal{S}_{Z}(X)\\
      0\ar[r]&\sgs(X,Z)^{\vee}\ar[r]\ar@{.>}[u]_{\beta}&\sgs(X)^{\vee}\ar[u]\ar[r]&\sgs(Z)^{\vee}\ar[u]}
  \end{equation}
  Again, it is a quasi-isomorphism.

  Now, $\mathcal{S}_{X}\otimes\Lambda_{U}$ is a resolution of
  $\Lambda_{U}$ which computes derived global sections hence we deduce
  the following result.
  \begin{fac}\label{lem:coh-pair-model}
    The zigzag of $\alpha$ and $\beta$ exhibits $\sgs(X,Z)^{\vee}$ as
    a model for $\dR\Gamma(X,\Lambda_{U})=\dR\pi_{*}j_{!}\Lambda$ in
    $\Der{\Lambda}$, where $\pi:X\to\pt$.
  \end{fac}

\subsection{Functoriality}

We now turn to functoriality of these constructions. Suppose we are
given a morphism of pairs of topological spaces $f:(X,Z)\to
(X',Z')$. We keep the notation from above, decorating the symbols with
a prime when the objects are associated to the second pair.

\begin{lem}\label{lem:coh-pair-functoriality}
  The following diagram commutes in $\Der{\Lambda}$:
  \begin{equation*}
    \xymatrix@C=5em{\dR\pi'_{*}j'_{!}\Lambda\ar[r]\ar[d]_{\beta^{-1}\alpha}^{\sim}&\dR\pi_{*}j_{!}\Lambda\ar[d]_{\beta^{-1}\alpha}^{\sim}\\
      \sgs(X',Z')^{\vee}\ar[r]_{\sgs(f)^{\vee}}&\sgs(X,Z)^{\vee}}
  \end{equation*}
  Here the top horizontal arrow is defined as $\dR\pi'_{*}$ applied to
  \begin{equation*}
    j'_{!}\Lambda\xrightarrow{\mathrm{adj}}\dR
    f_{*}f^{*}j'_{!}\Lambda\xrightarrow{}\dR
    f_{*}j_{!}f^{*}\Lambda\xrightarrow{\sim}\dR f_{*}j_{!}\Lambda.
  \end{equation*}
\end{lem}
\begin{proof}
  We will construct the two middle horizontal arrows below, and then
  prove that they make each square in the following diagram commute:
  \begin{equation*}
    \xymatrix{\dR\pi'_{*}j'_{!}\Lambda\ar[r]&\dR\pi_{*}j_{!}\Lambda\\
      \mathcal{S}_{X'}\otimes\Lambda_{U'}(X')\ar[r]^{\eqref{eq:coh-pair-edge-resolution}}\ar@{=}[u]\ar[d]_{\alpha}^{\sim}&\mathcal{S}_{X}\otimes\Lambda_{U}(X)\ar@{=}[u]\ar[d]_{\alpha}^{\sim}\\
      \mathcal{K}_{(X',Z')}(X')\ar[r]^{\mathcal{K}_{f}}&\mathcal{K}_{(X,Z)}(X)\\
      \sgs(X',Z')^{\vee}\ar[u]^{\beta}_{\sim}\ar[r]^{\sgs(f)^{\vee}}&\sgs(X,Z)^{\vee}\ar[u]^{\beta}_{\sim}}
  \end{equation*}

  From the inclusion $f^{-1}(U')\subset U$ we obtain a canonical
  morphism of sheaves on $X$:
  \begin{equation*}
    f^{*}\Lambda_{U'}\xrightarrow{\sim} \Lambda_{f^{-1}(U')}\xrightarrow{}\Lambda_{U}.
  \end{equation*}
  Composition with $f$ induces a morphism
  $\mathcal{S}_{f}:\mathcal{S}_{X'}\to f_{*}\mathcal{S}_{X}$ and thus
  by adjunction also $f^{*}\mathcal{S}_{X'}\to\mathcal{S}_{X}$.
  Together we obtain a morphism
  \begin{equation}
    \label{eq:coh-pair-edge-resolution}
    f^{*}(\mathcal{S}_{X'}\otimes\Lambda_{U'})\xrightarrow{\sim}f^{*}\mathcal{S}_{X'}\otimes
    f^{*}\Lambda_{U'}\xrightarrow{}\mathcal{S}_{X}\otimes \Lambda_{U}.
  \end{equation}
  Similarly, we define morphisms
  \begin{align*}
    f^{*}(\mathcal{S}_{X'}\otimes
    \Lambda_{Z'})\xrightarrow{\sim}f^{*}\mathcal{S}_{X'}\otimes
    f^{*}\Lambda_{Z'}\xrightarrow{}\mathcal{S}_{X}\otimes\Lambda_{Z}\intertext{and}
    f^{*}i'_{*}\mathcal{S}_{Z'}\xrightarrow{}i_{*}f^{*}\mathcal{S}_{Z'}\xrightarrow{}i_{*}\mathcal{S}_{Z}.
  \end{align*}
  It is then clear that the following diagram commutes
  \begin{equation*}
    \xymatrix{f^{*}(\mathcal{S}_{X'}\otimes\Lambda_{U'})\ar[r]\ar[d]&\mathcal{S}_{X}\otimes\Lambda_{U}\ar[d]\\
      f^{*}\mathcal{S}_{X'}\ar[r]\ar[d]&\mathcal{S}_{X}\ar[d]\\
      f^{*}(\mathcal{S}_{X'}\otimes\Lambda_{Z'})\ar[r]\ar[d]&\mathcal{S}_{X}\otimes\Lambda_{Z}\ar[d]\\
      f^{*}i'_{*}\mathcal{S}_{Z'}\ar[r]&i_{*}\mathcal{S}_{Z}}
  \end{equation*}
  so that, in particular, we deduce the existence of a morphism
  $f^{*}\mathcal{K}_{(X',Z')}\to\mathcal{K}_{(X,Z)}$ rendering the
  following two squares commutative:
  \begin{align*}
    \xymatrix{f^{*}\mathcal{K}_{(X',Z')}\ar[r]\ar[d]&\mathcal{K}_{(X,Z)}\ar[d]\\
      f^{*}\mathcal{S}_{X'}\ar[r]&\mathcal{S}_{X}}&&  \xymatrix{f^{*}\mathcal{K}_{(X',Z')}\ar[r]&\mathcal{K}_{(X,Z)}\\
      f^{*}(\mathcal{S}_{X'}\otimes\Lambda_{U'})\ar[r]\ar[u]^{\alpha}&\mathcal{S}_{X}\otimes\Lambda_{U}\ar[u]_{\alpha}}
  \end{align*}
  Denote by $\mathcal{K}_{f}:\mathcal{K}_{(X',Z')}\to
  f_{*}\mathcal{K}_{(X,Z)}$ the morphism obtained by adjunction. We
  now claim that also the following square of complexes commutes:
  \begin{equation*}
    \xymatrix{\mathcal{K}_{(X',Z')}(X')\ar[r]^{\mathcal{K}_{f}}&\mathcal{K}_{(X,Z)}(X)\\
      \sgs(X',Z')^{\vee}\ar[u]^{\beta}\ar[r]_{\sgs(f)^{\vee}}&\sgs(X,Z)^{\vee}\ar[u]_{\beta}}
  \end{equation*}
  Indeed, using the injection $\mathcal{K}_{(X,Z)}(X)\inj
  \mathcal{S}_{X}(X)$ one reduces to prove commutativity of
  \begin{equation*}
    \xymatrix{\mathcal{S}_{X'}(X')\ar[r]^{\mathcal{S}_{f}}&\mathcal{S}_{X}(X)\\
      \sgs(X')^{\vee}\ar[r]_{\sgs(f)^{\vee}}\ar[u]&\sgs(X)^{\vee}\ar[u]}
  \end{equation*}
  which is clear.

  Finally, notice that~\eqref{eq:coh-pair-edge-resolution} is
  compatible with the coaugmentations $\Lambda\to\mathcal{S}_{X'}$ and
  $\Lambda\to\mathcal{S}_{X}$, thus the lemma.
\end{proof}

\begin{lem}\label{lem:coh-pair-triangles}
  The following defines a morphism of distinguished triangles in
  $\Der{\Lambda}$:
  \begin{equation}\label{eq:coh-pair-triangles}
    \xymatrix{\dR
      \pi_{*}j_{!}\Lambda\ar[r]&\dR\pi_{*}\Lambda\ar[r]&\dR\pi_{*}i_{*}\Lambda\ar[r]&\dR
      \pi_{*}j_{!}\Lambda[-1]\\
      \sgs(X,Z)^{\vee}\ar[r]\ar[u]_{\alpha^{-1}\beta}^{\sim}&\sgs(X)^{\vee}\ar[r]\ar[u]_{\sim}&\sgs(Z)^{\vee}\ar[u]^{\sim}\ar[r]&\sgs(X,Z)^{\vee}[-1]\ar[u]^{\sim}_{\alpha^{-1}\beta}}
  \end{equation}
\end{lem}
Here, the first row is induced by the localization triangle while the
second row is the distinguished triangle associated to the short exact
sequence consisting of the first three terms.
\begin{proof}
  It is clear that the first two squares commute. We only need to
  prove this for the third one.

  Extend the first square in~\eqref{eq:coh-pair-alpha} to a morphism
  of triangles in $\ch(\sh(X))$
  \begin{equation*}
    \xymatrix{\mathcal{K}_{(X,Z)}\ar[r]^-{a}&\mathcal{S}_{X}\ar[r]&\mathrm{cone}(a)\ar[r]&\mathcal{K}_{(X,Z)}[-1]\\
      \mathcal{S}_{X}\otimes\Lambda_{U}\ar[r]_-{j}\ar[u]_{\alpha}&\mathcal{S}_{X}\ar@{=}[u]\ar[r]&\mathrm{cone}(j)\ar[r]\ar[u]&\mathcal{S}_{X}\otimes\Lambda_{U}[-1]\ar[u]_{\alpha}}
  \end{equation*}
  using the mapping cones. Since
  $\Gamma(X,\mathrm{cone}(a))=\mathrm{cone}(a_{X})$ the first square
  in~\eqref{eq:coh-pair-beta} extends to a morphism of triangles in
  $\ch(\Lambda)$:
  \begin{equation*}
    \xymatrix{\mathcal{K}_{(X,Z)}(X)\ar[r]^{a_{X}}&\mathcal{S}_{X}(X)\ar[r]&\mathrm{cone}(a)(X)\ar[r]&\mathcal{K}_{(X,Z)}(X)[-1]\\
      \sgs(X,Z)^{\vee}\ar[r]_{b}\ar[u]_{\beta}&\sgs(X)^{\vee}\ar[u]\ar[r]&\mathrm{cone}(b)\ar[u]\ar[r]&\sgs(X,Z)^{\vee}[-1]\ar[u]_{\beta}}
  \end{equation*}
  Notice that under the canonical identification
  $\mathrm{cone}(b)\xrightarrow{\sim} \sgs(Z)^{\vee}$, the bottom row
  is precisely the bottom row of~\eqref{eq:coh-pair-triangles}, while
  modulo the canonical identification
  $\mathrm{cone}(a)(X)\xrightarrow{\sim}
  \mathcal{S}_{Z}(Z)$, the top row of the first diagram induces the top row
  of~\eqref{eq:coh-pair-triangles} (taking global sections). Indeed,
  the latter contention follows from the fact that in $\Der{\sh(X)}$
  there is a \emph{unique} morphism $\delta$ making the following
  candidate triangle distinguished:
  \begin{equation*}
    \xymatrix{j_{!}\Lambda\ar[r]
      &\Lambda\ar[r]&i_{*}\Lambda\ar[r]^-{\delta}&j_{!}\Lambda[-1].}
  \end{equation*}

  The lemma now follows from the commutativity of the following
  diagram
  \begin{equation*}
    \xymatrix{\mathcal{S}_{X}\otimes\Lambda_{Z}(X)\ar[d]&\mathrm{cone}(j)(X)\ar[d]\ar[l]\\
      i_{*}\mathcal{S}_{Z}(X)&\mathrm{cone}(a)(X)\ar[l]\\
      \sgs(Z)^{\vee}\ar[u]&\mathrm{cone}(b)\ar[l]\ar[u]}
  \end{equation*}
  The first square commutes since the second square
  in~\eqref{eq:coh-pair-beta} does, while the second square does since
  the second square in~\eqref{eq:coh-pair-alpha} does.
\end{proof}

\subsection{Monoidality}

We come to the last compatibility of the model, namely with the cup
product. For this we fix a topological space $X$ and two closed
subspaces $Z_{1}$ and $Z_{2}$ of $X$. We write $Z=Z_{1}\cup Z_{2}$,
and we assume that there exist open neighborhoods $V_{i}$ of $Z_{i}$
in $X$ such that $V_{i}$ deformation retracts onto $Z_{i}$ and
$V_{1}\cap V_{2}$ deformation retracts onto $Z_{1}\cap Z_{2}$. This is
satisfied \eg{} if $X$ is a CW-complex and the $Z_{i}$ are
subcomplexes. 
The cup product in cohomology is denoted by $\cupproduct$, and
$\sgs(X,Z_{1}+Z_{2})$ is the free $\Lambda$-module on simplices in $X$
which are neither contained in $Z_{1}$ nor in $Z_{2}$.
\begin{lem}\label{lem:coh-pair-monoidal}
  The following diagram commutes in $\Der{\Lambda}$:
  \begin{equation*}
    \xymatrix{\dR\pi_{*}j_{1!}\Lambda\otimes^{\dL}\dR\pi_{*}j_{2!}\Lambda\ar[rr]^-{\cupproduct}\ar[d]_{\beta^{-1}\alpha}^{\sim}&&\dR\pi_{*}j_{!}\Lambda\ar[d]_{\beta^{-1}\alpha}^{\sim}\\
      \sgs(X,Z_{1})^{\vee}\otimes\sgs(X,Z_{2})^{\vee}\ar[r]_-{\cupproduct}&\sgs(X,Z_{1}+Z_{2})^{\vee}&\sgs(X,Z)^{\vee}\ar[l]^-{\sim}}
  \end{equation*}
  Here the top horizontal arrow is defined as the composition
  \begin{equation}\label{eq:coh-pair-cupproduct}
    \dR\pi_{*}j_{1!}\Lambda\otimes^{\dL}\dR\pi_{*}j_{2!}\Lambda\xrightarrow{}\dR\pi_{*}(j_{1!}\Lambda\otimes^{\dL}
    j_{2!}\Lambda)\xrightarrow{\sim}\dR\pi_{*}j_{!}\Lambda.
  \end{equation}
\end{lem}
\begin{proof}
  Notice that the composition
  \begin{equation*}
    \sgs(X,V_{i})^{\vee}\to\sgs(X,Z_{i})^{\vee}\xrightarrow{\beta}\mathcal{K}_{(X,Z_{i})}(X)
  \end{equation*}
  factors through
  $\alpha:\mathcal{S}_{X}\otimes\Lambda_{U_{i}}(X)\to\mathcal{K}_{(X,Z_{i})}(X)$
  because $V_{i}$ is open in $X$ and
  $\mathcal{S}_{X}\otimes\Lambda_{U_{i}}(X)$ consists of sections of
  $\mathcal{S}_{X}$ whose support is contained in $U_{i}$. It follows
  that the left vertical arrow in the lemma is equal to the
  composition of the left vertical arrows in the following diagram.
  \begin{equation}\label{eq:coh-pair-monoidal}
    \xymatrix{(\mathcal{S}_{X}\otimes\Lambda_{U_{1}})(X)\otimes^{\dL}
      (\mathcal{S}_{X}\otimes\Lambda_{U_{2}})(X)\ar[rr]^-{\cupproduct}&&(\mathcal{S}_{X}\otimes\Lambda_{U})(X)\ar[d]^{\alpha}_{\sim}\\
      \sgs(X,V_{1})^{\vee}\otimes\sgs(X,V_{2})^{\vee}\ar[r]^-{\cupproduct}\ar[d]_{\sim}\ar[u]^{\sim}&\sgs(X,V_{1}+V_{2})^{\vee}\ar[d]_{\sim}\ar[r]^-{\sim}&\mathcal{K}_{(X,Z)}(X)\\
      \sgs(X,Z_{1})^{\vee}\otimes\sgs(X,Z_{2})^{\vee}\ar[r]_-{\cupproduct}&\sgs(X,Z_{1}+Z_{2})^{\vee}&\ar[l]^-{\sim}\sgs(X,Z)^{\vee}\ar[u]_{\beta}^{\sim}}
  \end{equation}
  Recall that the sheaf $\mathcal{S}_{Z}$ is the quotient of the
  presheaf $V\mapsto \sgs(V)^{\vee}$ where a section $f\in
  \sgs(V)^{\vee}$ becomes 0 in $\mathcal{S}_{Z}(V)$ if there exists an
  open cover $(W_{i})_{i}$ of $V$ such that $f|_{W_{i}}=0$ for all
  $i$. Now, start with $f\in\sgs(X)^{\vee}$ vanishing on both $V_{1}$
  and $V_{2}$, \ie{} an element of $\sgs(X,V_{1}+V_{2})^{\vee}$. These
  two open subsets of $X$ cover $Z$, and by the description of
  $\mathcal{S}_{Z}$ just given, we see that $f$ defines the zero class
  in $i_{*}\mathcal{S}_{Z}(X)$ hence lands in
  $\mathcal{K}_{(X,Z)}(X)$. This yields the right horizontal arrow in
  the middle row. It follows that the upper half of the diagram
  commutes. Evidently the lower left square does as well. For the
  lower right square denote by $V$ the union of $V_{1}$ and
  $V_{2}$. Then we may decompose this square as follows:
  \begin{equation*}
    \xymatrix{\sgs(X,V_{1}+V_{2})^{\vee}\ar[d]_{\sim}\ar@/^2em/[rr]^{\sim}&\sgs(X,V)^{\vee}\ar[l]_-{\sim}\ar[d]&\mathcal{K}_{(X,Z)}(X)\\
      \sgs(X,Z_{1}+Z_{2})^{\vee}&\sgs(X,Z)^{\vee}\ar[l]^-{\sim}\ar[ru]_{\beta}^{\sim}}
  \end{equation*}
  Commutativity is now clear.

  It remains to prove that the top horizontal arrow
  in~\eqref{eq:coh-pair-monoidal} is a model
  for~\eqref{eq:coh-pair-cupproduct}. This follows from the fact that
  the resolution $\Lambda\xrightarrow{\sim}\mathcal{S}_{X}$ of the
  constant sheaf on $X$ is \emph{multiplicative}. Namely, this makes
  the right square of the following diagram commutative; the left one
  clearly commutes.
  \begin{equation*}
    \xymatrix@C-1em{\dR\pi_{*}\Lambda_{U_{1}}\otimes^{\dL}\dR\pi_{*}\Lambda_{U_{2}}\ar[r]\ar@{=}[d]&\dR\pi_{*}(\Lambda_{U_{1}}\otimes\Lambda_{U_{2}})\ar[r]^-{\sim}\ar@{=}[d]&\dR\pi_{*}\Lambda_{U}\ar@{=}[d]\\
      \pi_{*}(\mathcal{S}_{X}\otimes\Lambda_{U_{1}})\otimes^{\dL}\pi_{*}(\mathcal{S}_{X}\otimes\Lambda_{U_{2}})\ar[r]&\pi_{*}(\mathcal{S}_{X}\otimes\Lambda_{U_{1}}\otimes\mathcal{S}_{X}\otimes\Lambda_{U_{2}})\ar[r]_-{\sim}&\pi_{*}(\mathcal{S}_{X}\otimes\Lambda_{U})}
  \end{equation*}

\end{proof}

\section{Categories of comodules}
\label{sec:comod}
In this section we recall some facts about categories of (complexes
of) comodules used in the main body of the text. Throughout we fix a
ring $\Lambda$ and a \emph{flat} $\Lambda$-coalgebra $C$.  By a
$C$-comodule we mean a counitary left $C$-comodule. $\coMod{C}$
(respectively, $\fcoMod{C}$) denotes the category of $C$-comodules
(respectively, $C$-comodules finitely generated as $\Lambda$-modules).

The starting point is really the following result.
\begin{fac}\label{thm:comod-basic}\mbox{}
  \begin{enumerate}
  \item $\fcoMod{C}$ and $\coMod{C}$ are abelian $\Lambda$-linear
    categories, and there is a canonical equivalence of abelian
    $\Lambda$-linear categories $\ind\fcoMod{C}\simeq \coMod{C}$.
  \item The forgetful functor $\ff:\coMod{C}\to\Mod{\Lambda}$ is exact
    $\Lambda$-linear and creates colimits and finite limits.
  \item $\coMod{C}$ is a Grothendieck category, copowered over
    $\Mod{\Lambda}$. In particular, it is bicomplete.
  \end{enumerate}
\end{fac}
\begin{proof}
  The first statement follows from~\cite[II, 2.0.6 and
  2.2.3]{saavedra-cat.tann.1972}. The rest is proved
  in~\cite{wischnewsky-comod-1975}, see~\cite[Cor.~3 and~9, Pro.~38,
  Cor.~26]{wischnewsky-comod-1975}. Explicitly, the copower of a
  $\Lambda$-module $m$ and a $C$-comodule $c$ is given by the tensor
  product (as $\Lambda$-modules) $m\otimes c$ with the comodule
  coaction on $c$.
\end{proof}

Next, we are interested in different models for the derived category
of $\coMod{C}$. The following result is true more generally for any
Grothendieck category.
\begin{fac}[{\cite[Thm.~1.2]{cisinski-deglise:homalg-grothendieck-cats}}]\label{thm:comod-model-injective}
  $\ch(\coMod{C})$ is a proper cellular model category with
  quasi-isomorphisms as weak equivalences and monomorphisms as
  cofibrations.
\end{fac}
The model structure in the statement is called the \emph{injective
  model structure}.

From now on assume that $C$ is a (commutative) bialgebra. $\coMod{C}$
then becomes a monoidal $\Lambda$-linear category with $C$ coacting on
the tensor product (as $\Lambda$-modules) $c\otimes d$ by
\begin{equation*}
  c\otimes d\xrightarrow{\coa\otimes \coa}(c\otimes C)\otimes (d\otimes
  C)\xrightarrow{\sim}(c\otimes d)\otimes (C\otimes C)\xrightarrow{}
  (c\otimes d)\otimes C,
\end{equation*}
the last arrow being induced by the multiplication of $C$. In
particular, the forgetful functor $\ff:\coMod{C}\to\Mod{\Lambda}$ is
monoidal. The category $\ch(\coMod{C})$ inherits a monoidal structure
in the usual way.

\begin{pro}\label{pro:comod-stable-model-injective}
  Let $T$ be a flat object in $\ch(\coMod{C})$. Then there is a proper
  cellular model structure on $\spt_{T}\ch(\coMod{C})$ with stable
  equivalences as weak equivalences and monomorphisms as cofibrations.
\end{pro}
\begin{proof}
  The stable equivalences are described
  in~\cite[Def.~8.7]{hovey:spectra}, and the proof
  in~\cite[Pro.~6.31]{cisinski-deglise:homalg-grothendieck-cats}
  applies.\qedhere{}

\end{proof}
The model structure of the proposition is called the \emph{injective
  stable model structure}.

Unfortunately, the monoidal structure does not, in general, interact
well with the injective model structures. In the cases of interest in
the main body of the text (namely, when $\Lambda$ is a principal ideal
domain) we have the following result, essentially due to Serre.

\begin{lem}\label{lem:comod-monoidal-dedekind}
  Let $\Lambda$ be a Dedekind domain. In $\ch(\coMod{C})$ there exist
  functorial flat resolutions. In particular, $\Der{\coMod{C}}$ admits
  naturally a monoidal structure.
\end{lem}
\begin{proof}
  We follow~\cite[Pro.~3]{serre-groupe-grothendieck}. Let $E$ be a
  comodule and consider the morphism of comodules $\coa_{E}:E\to
  C\otimes E$ given by the coaction of $C$ on $E$, where the target
  has a comodule structure induced by the comultiplication on $C$
  (sometimes called the ``extended comodule associated to $E$''). In
  fact, the coaction $\coa_{\bullet}$ defines a natural transformation
  from the identity functor on $\coMod{C}$ to the ``extended
  comodule''-functor (this is the unit of an adjunction whose left
  adjoint is the forgetful functor $\ff:\coMod{C}\to\Mod{\Lambda}$;
  cf.~\cite[Lem.~1.53]{ayoub:galois1}). Since $E$ is counitary, this
  natural transformation is objectwise injective. Let
  $F:\Mod{\Lambda}\to\Mod{\Lambda}$ be the functor which associates to
  a $\Lambda$-module $M$ the free $\Lambda$-module $\oplus_{m\in
    M}\Lambda$. It comes with a natural transformation
  $\eta:F\to\mathrm{Id}$ which is objectwise an epimorphism. We obtain
  a diagram
  \begin{equation*}\label{eq:flat-resolutions-construction}
    \xymatrix{E\ar[r]^-{\coa_{E}}&C\otimes E\\
      &C\otimes F(E)\ar[u]_{1\otimes\eta_{E}}}
  \end{equation*}
  in the category of $C$-comodules (the module in the bottom row is
  again an extended comodule). Since the forgetful functor from
  $\coMod{C}$ to $\Mod{\Lambda}$ commutes with finite limits, we see
  that the pullback of this diagram is a comodule $E'$ which both maps
  surjectively onto $E$, and embeds into $C\otimes F(E)$. By
  assumption, $C\otimes F(E)$ is torsion-free thus so is $E'$.  It is
  clear from the construction that the association $E\mapsto E'$
  defines a functor together with a natural transformation $\eta'$
  from it to the identity functor.

  Using that $\coMod{C}$ is a Grothendieck category, the usual
  procedure leads to functorial flat resolutions. 
\end{proof}

If $\Lambda$ is a field then we can do better.
\begin{lem}\label{lem:comod-monoidal-field}
  Let $\Lambda$ be a field. The injective model structure on
  $\ch(\coMod{C})$ is a monoidal model structure.
\end{lem}
\begin{proof}
  Indeed, since the forgetful functor is monoidal, exact and creates
  colimits, the conditions for the injective model structure to be
  monoidal can be checked in
  $\ch(\Lambda)$. 
\end{proof}

\end{document}